\documentclass[twoside,11pt,reqno]{amsart}
\usepackage{amsmath,amssymb,amscd,mathrsfs,epic,wasysym,latexsym,tikz,mathrsfs,cite,hyperref}
\usepackage{stmaryrd}
\usepackage{pb-diagram}
\usepackage[matrix,arrow]{xy}

\makeatletter

\numberwithin{equation}{section}

\newtheorem{Proposition}[equation]{Proposition}
\newtheorem{Lemma}[equation]{Lemma}
\newtheorem{Theorem}[equation]{Theorem}
\newtheorem{Corollary}[equation]{Corollary}
\theoremstyle{definition}  
\newtheorem{Definition}[equation]{Definition}
\newtheorem{Remark}[equation]{Remark}

\newtheorem{Conjecture}[equation]{Conjecture}
\newtheorem{MainTheorem}{Theorem}

\let\<\langle
\let\>\rangle

\newcommand\Comment[2][\relax]{\space\par\medskip\noindent%
   \fbox{\begin{minipage}{\textwidth}\textbf{Comment\ifx\relax#1\else---#1\fi}\newline%
        #2\end{minipage}}\medskip
}

\def\bl{\text{\boldmath$l$}}
\def\bp{\text{\boldmath$p$}}
\def\bq{\text{\boldmath$q$}}
\def\bs{\text{\boldmath$s$}}

\def\bt{\text{\boldmath$t$}}
\def\bc{\text{\boldmath$c$}}
\def\br{\text{\boldmath$r$}}
\def\b1{\text{\boldmath$1$}}

\def\ba{\text{\boldmath$a$}}
\def\bb{\text{\boldmath$b$}}
\def\bw{\text{\boldmath$w$}}
\def\bu{\text{\boldmath$u$}}
\def\bv{\text{\boldmath$v$}}
\def\bx{\text{\boldmath$x$}}
\def\by{\text{\boldmath$y$}}
\def\bz{\text{\boldmath$z$}}
\def\bbe{\text{\boldmath$e$}}
\def\balpha{\text{\boldmath$\alpha$}}
\def\bbeta{\text{\boldmath$\beta$}}
\def\bkappa{\text{\boldmath$\kappa$}}
\def\btau{\text{\boldmath$\tau$}}

\def\fa{\mathfrak{a}}

\def\pmod#1{\text{ }(\text{\rm mod } #1)\,}

\newcommand{\cha}{\operatorname{char}}

\newcommand{\Stab}{\operatorname{Stab}}

\newcommand{\Z}{\mathbb{Z}}
\newcommand{\K}{\mathbb{K}}
\newcommand{\F}{\mathbb{F}}

\newcommand{\0}{{\bar 0}}
\renewcommand{\1}{{\bar 1}}
\def\eps{{\varepsilon}}
\def\phi{{\varphi}}

\newcommand{\funF}{{\mathcal F}}
\newcommand{\funG}{{\mathcal G}}

\newcommand{\Triple}{{\mathcal T}}

\newcommand{\ga}{\gamma}

\newcommand{\la}{\lambda}
\newcommand{\La}{\Lambda}
\newcommand{\al}{\alpha}
\newcommand{\be}{\beta}

\def\Si{\mathfrak{S}}
\newcommand{\si}{\sigma}
\newcommand{\om}{\omega}
\newcommand{\Om}{\Omega}

\newcommand{\de}{\delta}
\newcommand{\De}{\Delta}
\newcommand{\ka}{\kappa}

\newcommand{\rad}{{\mathrm {rad}\,}}

\newcommand{\ZZ}{{\mathbb Z}}

\newcommand{\D}{{\mathscr D}}

\renewcommand{\mod}{\bmod \,}

\def\K{\mathbb K}

\newcommand{\Zig}{{\bar{Z}}}
\newcommand{\EZig}{Z}
\renewcommand{\Alph}{{\mathscr A}}

\def\L{\text{\boldmath$L$}}

\def\col{{\operatorname{col}}}
\def\letter{{\operatorname{let}}}
\newcommand{\Nodes}{\mathsf N}
\newcommand{\Std}{\operatorname{Std}}
\newcommand{\Rst}{\operatorname{Rst}}
\newcommand{\Tab}{\operatorname{Tab}}
\newcommand{\Cst}{\operatorname{Cst}}

\def\Brackets#1{[ #1 ]}

\def\umu{{\underline{\mu}}}
\def\unu{{\underline{\nu}}}

\def\ueps{{\underline{\varepsilon}}}
\def\uom{{\underline{\omega}}}

\def\ud{{\underline{d}}}

\def\b{\mathfrak{b}}
\def\k{\Bbbk}

\def\T{\text{\boldmath$T$}}

\def\Stab{\text{\boldmath$S$}}

\def\spa{\operatorname{span}}

\def\op{{\mathrm{op}}}

\def\mod#1{#1\!\operatorname{-mod}}

\def\iso{\stackrel{\sim}{\longrightarrow}}
\def\B{{\mathcal B}}
\def\X{{\mathcal X}}
\def\Y{{\mathcal Y}}

\def\ch{\operatorname{ch}}
\def\lan{\langle}
\def\ran{\rangle}

\def\c{\mathfrak{c}}

\def\Seq{\operatorname{Tri}}

\def\cod{\mathcal B}

\def\tts{{\tt s}}

\def\a{\mathfrak{a}}
\def\z{\mathfrak{z}}

\def\m{\mathfrak{m}}

\def\brho{\text{\boldmath$\rho$}}
\def\bla{\text{\boldmath$\lambda$}}
\def\bmu{\text{\boldmath$\mu$}}
\def\bnu{\text{\boldmath$\nu$}}

\def\bkappa{\text{\boldmath$\kappa$}}
\def\bta{\text{\boldmath$\eta$}}
\def\bchi{\text{\boldmath$\chi$}}
\def\bga{\text{\boldmath$\gamma$}}

\def\bal{\text{\boldmath$\alpha$}}
\def\bbe{\text{\boldmath$\be$}}

{\catcode`\|=\active
  \gdef\set#1{\mathinner{\lbrace\,{\mathcode`\|"8000%
  \let|\midvert #1}\,\rbrace}}
}
\def\midvert{\egroup\mid\bgroup}

\colorlet{darkgreen}{green!50!black}
\tikzset{dots/.style={very thick,loosely dotted},
         greendot/.style={fill,circle,color=darkgreen,inner sep=1.5pt,outer sep=0},
         blackdot/.style={fill,circle,color=black,inner sep=1.5pt,outer sep=0},
         graydot/.style={fill,circle,color=gray,inner sep=1.1pt,outer sep=0}
}
\def\greendot(#1,#2){\node[greendot] at(#1,#2){}}
\def\blackdot(#1,#2){\node[blackdot] at(#1,#2){}}
\def\graydot(#1,#2){\node[graydot] at(#1,#2){}}

\newenvironment{braid}{
  \begin{tikzpicture}[baseline=6mm,black,line width=1pt, scale=0.32,
                      draw/.append style={rounded corners},
                      every node/.append style={font=\fontsize{5}{5}\selectfont}]%
  }{\end{tikzpicture}
}

\def\Grid(#1,#2){
  \draw[very thin,gray,step=2mm] (0,0)grid(#1,#2);
  \draw[very thin,darkgreen,step=10mm] (0,0)grid(#1,#2);
}

\newcommand\Tableau[2][\relax]{
  \begin{tikzpicture}[scale=0.5,draw/.append style={thick,black}]
    \ifx\relax#1\relax%
    \else 
      \foreach\box in {#1} { \filldraw[blue!30]\box+(-.5,-.5)rectangle++(.5,.5); }
    \fi
    \newcount\row\newcount\col
    \row=0
    \foreach \Row in {#2} {
       \col=1
       \foreach\k in \Row {
          \draw(\the\col,\the\row)+(-.5,-.5)rectangle++(.5,.5);
          \draw(\the\col,\the\row)node{\k};
          \global\advance\col by 1
       }
       \global\advance\row by -1
    }
  \end{tikzpicture}
}

\newcommand\YoungDiagram[2][\relax]{
  \begin{tikzpicture}[scale=0.5,draw/.append style={thick,black}]
    \ifx\relax#1\relax%
    \else 
    \foreach\box in {#1} {
      \filldraw[blue!30]\box rectangle ++(1,1);
    }
    \fi
    \newcount\row
    \row=0
    \foreach \col in {#2} {
       \draw(1,\the\row)grid ++(\col,1);
       \global\advance\row by -1
    }
  \end{tikzpicture}
}

\begin{document}

\title[Schurifying quasi-hereditary algebras]{{\bf Schurifying quasi-hereditary algebras}}

\author{\sc Alexander Kleshchev}
\address{Department of Mathematics\\ University of Oregon\\
Eugene\\ OR 97403, USA}
\email{klesh@uoregon.edu}

\author{\sc Robert Muth}
\address{Department of Mathematics\\ Washington \& Jefferson College
\\
Washington\\ PA 15301, USA}
\email{rmuth@washjeff.edu}


\thanks{The first author was supported by the NSF grant No. DMS-1700905 and the DFG Mercator program through the University of Stuttgart. This work was also supported by the NSF under grant No. DMS-1440140 while both authors were in residence at the MSRI during the Spring 2018 semester.}

\begin{abstract}
We define and study new classes of quasi-hereditary and cellular algebras which generalize Turner's double algebras. Turner's algebras provide a local description of blocks of symmetric groups up to derived equivalence. Our general construction allows one to `schurify' any quasi-hereditary algebra $A$ to obtain a generalized Schur algebra $S^A(n,d)$ which we prove is again quasi-hereditary if $d\leq n$.  We describe decomposition numbers of $S^A(n,d)$ in terms of those of $A$ and the classical Schur algebra $S(n,d)$. In fact, it is essential to work with quasi-hereditary {\em superalgebras} $A$, in which case the construction of the schurification involves a non-trivial 
full rank sub-lattice $T^A_\a(n,d)\subseteq S^A(n,d)$. 
\end{abstract}

\maketitle

\section{Introduction}

The goal of this paper is to define new quasi-hereditary algebras, and hence new highest weight categories, from old. Starting with a quasi-hereditary algebra $A$, we obtain a generalized Schur algebra $S^A(n,d)$ which we prove is again quasi-hereditary if $d\leq n$. The procedure of passing from $A$ to $S^A(n,d)$ is sometimes referred to as `schurification' of $A$.  We describe decomposition numbers of $S^A(n,d)$ in terms of those of $A$ and the classical Schur algebra $S(n,d)$. In fact, it is essential to study schurifications of  {\em superalgebras}, in which case  the construction involves a non-trivial choice of a sub-lattice $T^A_\a(n,d)\subseteq S^A(n,d)$. In the purely even case we have $T^A_\a(n,d)= S^A(n,d)$.

To describe our results more precisely, let $\k$ be a commutative domain of characteristic $0$, $A$ be a $\k$-superalgebra and $\a \subseteq A_\0$ be a subalgebra. The associated  generalized Schur algebra $T^A_\a(n,d)$
was defined in \cite{greenTwo} as a certain full sublattice in the algebra of (super)invariants: 
$$
T^A_\a(n,d)\subseteq S^A(n,d):=(M_n(A)^{\otimes d})^{\Si_d}.
$$
Thus extending scalars to a field $\K$ of characteristic $0$ produces the same algebras: $T^A_\a(n,d)_\K= S^A(n,d)_\K$. However, importantly, extending scalars to a field $\F$ of positive characteristic will in general yield {\em non-isomorphic} algebras $T^A_\a(n,d)_\F$ and $S^A(n,d)_\F$ of the same dimension. It turns out that in many situations it is the more subtly defined  algebra $T^A_\a(n,d)_\F$ that plays an important role. 

In  \cite{greenOne}, we defined the notion of a {\em based quasi-hereditary algebra}. If \(\k\) is a complete local Noetherian ring, then a $\k$-algebra is based quasi-hereditary if and only if it is {\em split quasi-hereditary} in the sense of Cline, Parshall and Scott \cite{CPS}.

The first main result of this paper is that under some reasonable assumptions, the algebra  $T^A_\a(n,d)$ is based quasi-hereditary if $A$ is based quasi-hereditary and $d\leq n$. This is proved by generalizing Green's work in \cite{GreenCod} for the classical Schur algebra \(S^\k(n,d)\). Green constructs a basis of {\em codeterminants}, and then proves (in effect) that this gives \(S^\k(n,d)\) the structure of a based quasi-hereditary algebra.
Similarly, we define a set of generalized codeterminants  and prove that these form a basis for  \(T^A_\a(n,d)\). This gives \(T^A_\a(n,d)\) the structure of a based quasi-hereditary algebra. 

Given a ring homomorphism \(\k \to \F\), where \(\F\) is a field of arbitrary characteristic, we define a quasi-hereditary $\F$-algebra by extending scalars: 
$$T^A_\a(n,d)_\F:=T^A_\a(n,d)\otimes_\k\F.
$$ 
Our second main result  describes (under some constraints) the decomposition numbers of standard \(T^A_\a(n,d)_\F\)-modules in terms of those of \(A_\F\), the classical Schur algebra \(S(n,d)_\F\), and Littlewood-Richardson coefficients. 
Our third main result describes conditions on \(A\) under which \(T^A_\fa(n,d)\) is known to be indecomposable, allowing for a classification of the blocks of \(T^A_\fa(n,d)\) in most cases.

Our motivation comes from Turner's double algebras \cite{T,T2,T3,EK1}, which arise as shurifications $T^{\bar Z}_{\bar\z}(n,d)$ of {\em zigzag superalgebras} $\bar Z$. As conjectured in \cite{T} and proved in \cite{EK2}, Turner's algebras $T^{\bar Z}_{\bar \z}(n,d)$ can be considered as a `local'  object replacing wreath products of Brauer tree algebras in the context of the Brou{\'e} abelian defect group conjecture for blocks of symmetric groups with {\em non-abelian} defect groups. 
We expect that various versions of generalized Schur algebras will be appearing in local descriptions of blocks of group algebras and other algebras arising in classical  representation theory.

As an application of the general techniques developed in this paper, we construct an explicit cellular basis of $T^{\bar Z}_{\bar\z}(n,d)$. To achieve this goal, we first construct quasi-hereditary algebras $T^Z_\z(n,d)$, 
where 
 \(\EZig\)  is a quasi-hereditary cover of $\bar Z$ known as the {\em extended zigzag superalgebra}. 
Special cases of the main results of this paper describe the quasi-hereditary structure and decomposition numbers of \(T^{\EZig}_{\z}(n,d)\). We then use an idempotent truncation technique to describe a cellular structure and the corresponding decompositon numbers of \(T^{\bar \EZig}_{\bar \z}(n,d)\). We formulate an explicit conjecture for RoCK blocks of classical Schur algebras in terms of the generalized Schur algebras $T^Z_\z(n,d)$, see Conjecture~\ref{Conj}.

We now describe the contents of the paper in more detail. In Section~\ref{SecQH} we recall the definition and basic results on based quasi-hereditary superalgebras \(A\) with quasiheredity data \(I,X,Y\). In Section~\ref{SecComb} we set up the combinatorics of colored alphabets and describe the poset of \(I\)-colored multipartitions \(\La^I_+(n,d)\). For any \(\bla \in \La^I_+(n,d)\), we define the sets \(\Std^X(\bla)\) and \(\Std^Y(\bla)\) of \(X\)-colored and \(Y\)-colored standard tableaux. In Section~\ref{SecGenSch} we recall the definition and basic results on the generalized Schur algebra \(T^A_\a(n,d)\), and prove some  mulitiplication lemmas.

In Section~\ref{SecCod}, we define generalized codeterminants. For every \(\bla \in \La^I_+(n,d)\), \(\Stab \in \Std^X(\bla)\), \(\T \in \Std^Y(\bla)\), we define the corresponding codeterminant \(\B_{\Stab,\T}^{\bla}\) as a product \(\X_\Stab\Y_\T\) of certain elements of \(T^A_\a(n,d)\). We show that the codeterminants form a basis for \(T^A_\a(n,d)\), using a generalization of Woodcock's `straightening' argument  in \cite{Woodcock} to prove  spanning, and the super RSK correspondence to show independence.

In Section~\ref{SecQHstruct}, we prove the first main theorem of the paper, which appears as Theorem~\ref{T290517_2}:

\begin{MainTheorem}\label{T1}
{\em 
Let $d\leq n$ and $A$ be a based quasi-hereditary $\k$-superalgebra with $\a$-conforming heredity data $I,X,Y$. 
Then $T^A_\a(n,d)$ is a  based quasi-hereditary $\k$-superalgebra 
with heredity data $\La_+^I(n,d), \X, \Y$.
}
\end{MainTheorem}

We then go on to describe the standard modules \(\Delta(\bla)\) over \(T^A_\a(n,d)\), as well as idempotent truncations and involutions of these algebras. In the final Section~\ref{SecDecomp} we focus on decomposition numbers of standard modules over \(T^A_\a(n,d)_\F\). 
We define a certain explicit set $\La_+^D(n)$ of multipartitions depending on the decomposition matrix $D$ of $A_\F$, an explicit statistics $\deg$ on $\La_+^D(n)$, certain classical Littlewood-Richardson coefficients $c^{\la^{(i)}}_{{}_i\overline{\bnu}}$ and $c^{\ga^{(i)}}_{\bnu_i}$ and products $d^{\,\textup{cl}}_{\bga,\bmu}$ of decomposition numbers for the classical Schur algebra over $\F$. Then our second main theorem, which appears as Theorem~\ref{NonBasicDecomp}, is as follows.

\begin{MainTheorem}\label{T2}
{\em Suppose that $(A,\a)$ is a unital pair and assume that \(A_{\overline 1} \subset J(A)\). Then for \(\bla, \bmu \in \La_+^I(n,d)\), the corresponding graded decomposition number is given by
\begin{align*}
d_{\bla, \bmu}= \sum_{\bga \in \La_+^I(n)} \sum_{\bnu \in \La_+^D(n)} d^{\,\textup{cl}}_{\bga,\bmu} \deg(\bnu)\left( \prod_{i \in I } c^{\la^{(i)}}_{{}_i\overline{\bnu}} c^{\ga^{(i)}}_{\bnu_i}\right).
\end{align*}
}
\end{MainTheorem}

After specialization $A:=Z$ and appropriate idempotent truncation, Theorem~\ref{T2} yields the formula of Turner \cite[Corollary 134]{T}, cf.  \cite[Theorem 6.2]{CT}, \cite[Corollary 10]{LM}, \cite[Theorem 4.1]{JLM}, see Remark~\ref{R140318} for further comments. As an application of Theorem 2, we prove our third main theorem, which appears as Theorem~\ref{indec} in the body of the paper.

\begin{MainTheorem}\label{T3}
{\em 
Suppose that $(A,\a)$ is a unital pair and \(A_{\overline 1} \subset J(A)\). Moreover, suppose that \(A\) is indecomposable, and \(|I| >1\). Then \(T^{ A}_{ \fa}(n,d)\) is indecomposable.
}
\end{MainTheorem}

This, coupled with the decomposition result described in Lemma~\ref{BlockDecomp}, allows one to classify the blocks of \(T^A_\fa(n,d)\) in many cases. In fact, we prove in Theorem~\ref{indec} a slightly stronger result, giving indecomposability conditions for \(T^{\bar A}_{\bar \fa}(n,d)\), where \((\bar A, \bar \fa)\) is an idempotent truncation of the unital pair \((A,\fa)\).

\section{Quasi-hereditary algebras}\label{SecQH}

Throughout the paper $\k$ is a commutative domain of characteristic $0$. 

\subsection{Based quasi-hereditary algebras}
\label{SSBQHA}
We begin by reviewing theory of quasi-hereditary algebras in the language of \cite{greenOne}. 

Let $V=\bigoplus_{n\in\Z,\,\eps\in\Z/2}V^n_\eps$ be a {\em graded $\k$-supermodule}. We set $V^n:=V^n_\0\oplus V^n_\1$ and $V_\eps:=\bigoplus_{n\in\Z}V^n_\eps$. 
An element $v\in V$ is called homogeneous if $v\in V_\eps^n$ for some $\eps$ and $n$. For $\eps\in\Z/2$, $n\in \Z$ and a set $S$ of homogeneous elements of $V$, we write 
\begin{equation}\label{E210218_2}
S_\eps:=S\cap V_\eps\qquad \text{and}\qquad S_\eps^n:=S\cap V^n_\eps.
\end{equation}

Let
\begin{equation}\label{ER}
R:=\Z[q,q^{-1}][t]/(t^2-1), 
\end{equation} 
and denote the image of $t$ in the quotient ring by $\pi$, so that $\pi^\eps$ makes sense for $\eps\in\Z/2$. 
For $v\in V^n_\eps$, we write 
\begin{equation}\label{EDeg}
\deg(v):=q^n\pi^\eps.
\end{equation}
For $v\in V_\eps$, we also write $\bar v:=\eps$. A map $f:V\to W$ of graded $\k$-supermodules is called {\em homogeneous} if $f(V^m_\eps)\subseteq W^m_\eps$ for all $m$ and $\eps$. 

For a free $\k$-module $W$ of finite rank $d$, we write $d=\dim W$. A graded $\k$-supermodule $V$ is free of finite rank if each $V^n_\eps$ is free of finite rank and we have $V^n=0$ for almost all $n$. Let $V$ be a free graded $\k$-supermodule of finite rank. 
A {\em homogeneous basis} of $V$ is a $\k$-basis all of whose elements are homogeneous. The {\em graded dimension} of $V$ is 
$$
\dim^q_\pi V:=\sum_{n\in\Z,\, \eps\in\Z/2}(\dim V^n_\eps)q^n\pi^\eps\in R.
$$
   
A (not necessarily unital) $\k$-algebra $A$ is called a {\em graded $\k$-superalgebra}, if $A$ is a graded $\k$-supermodule and $A_{\eps}^nA_{\de}^m\subseteq A_{\eps+\de}^{n+m}$ for all $\eps,\de$ and $n,m$.  
By a {\em graded $A$-supermodule} we understand an $A$-module $V$ which is a graded $\k$-supermodule and 
$A_{\eps}^nV_{\de}^m\subseteq V_{\eps+\de}^{n+m}$ for all $\eps,\de$ and $n,m$. 
We denote by $\mod{A}$ the category of all finitely generated graded $A$-supermodules and homogeneous $A$-homomorphisms. All ideals, submodules, etc. are assumed to be homogeneous. 
Given $V\in \mod{A}$, $n\in \Z$ and $\eps\in\Z/2\Z$, we denote by $q^n\pi^\eps V$ the graded $A$-supermodule which is the same as $V$ as an $A$-module but with $(q^n\pi^\eps V)^m_\de=V^{m-n}_{\de+\eps}$.


For a partially ordered set $I$ and $i\in I$, we let
\begin{equation}\label{E120218}
I^{>i}:=\{j\in I\mid j>i\}\quad\text{and}\quad I^{\geq i}:=\{j\in I\mid j\geq i\}. 
\end{equation}

\begin{Definition} \label{DCC} 
{\rm 
Let $A$ be a graded $\k$-superalgebra. 
A {\em heredity data} on $A$ consist of a finite partially ordered set $I$ and finite sets $X=\bigsqcup_{i\in I}X(i)$ and $Y=\bigsqcup_{i\in I}Y(i)$ of homogeneous elements of $A$ with distinguished {\em initial elements} $e_i\in X(i)\cap  Y(i)$ for each $i\in I$.
For $i\in I$, we set 
\begin{align*}
A^{>i}:=\spa(xy\mid i\in I^{>i},\, x\in X(i),\, y\in Y(i)).  
\end{align*}
We require that the following axioms hold: 
\begin{enumerate}
\item[{\rm (a)}] $B:=\{x y \mid i\in I,\, x\in X(i),\, y\in Y(i)\}$ is a basis of $A$; 

\item[{\rm (b)}] For all $i\in I$, $x\in X(i)$, $y\in Y(i)$ and $a\in A$, we have
$$
a x \equiv \sum_{x'\in X(i)}l^x_{x'}(a)x' \pmod{A^{>i}}
\ \ \text{and}\ \ 
ya \equiv \sum_{y'\in Y(i)}r^y_{y'}(a)y' \pmod{A^{>i}}
$$
for some $l^x_{x'}(a),r^y_{y'}(a)\in\k$;

\item[{\rm (c)}] For all $i,j\in I$ and $x\in X(i),\ y\in Y(i)$ we have  
\begin{align*}
&xe_i= x,\ e_ix= \de_{x,e_i}x,\ e_i y= y,\ ye_i= \de_{y,e_i}y 
\\
&e_jx=x\ \text{or}\ 0,\ ye_j=y\ \text{or}\ 0. 
\end{align*}
\end{enumerate}
}
\end{Definition}

If $A$ is endowed with a heredity data $I,X,Y$, we call $A$ {\em based quasi-hereditary (with respect to the poset $I$)}, and refer to $B$ as a {\em heredity basis} of $A$. \begin{Remark} 
{\rm 
The notion of a based quasi-hereditary algebra is closely related to that of a {\em split quasi-hereditary algebra} developed in 
 \cite{CPS,DS,Ro} for algebras over an arbitrary Noetherian commutative unital ring $\k$. In fact, if $\k$ is complete local Noetherian, which is sufficient for most applications, the two notions are equivalent. We refer the reader to \cite{greenOne} for the proof of a slightly stronger statement. 
}
\end{Remark}

We now record some basic results on a based quasi-hereditary algebra $A$ with heredity data as in Definition~\ref{DCC}. The proofs can be found in \cite{greenOne}.  Denote

\begin{eqnarray}
\label{EBIXY}
b^i_{x,y}&:=&xy\in B\qquad(i\in I,\ x\in X(i),\ y\in Y(i)).
\\
\label{EB(i)}
B(i)&:=&\{b^i_{x,y}\mid x\in X(i),\ y\in Y(i)\}\qquad(i\in I).
\end{eqnarray}

\begin{Lemma} \label{L120218} 
If $\Om$ is a coideal in $I$ then $A(\Om):=\spa(\sqcup_{i\in \Om}B(i))$ is an ideal in $A$. Moreover, if $\Theta$ is another coideal in $I$, we have $A(\Om)A(\Theta)\subseteq A(\Om\cap\Theta)$. 
\end{Lemma}

\begin{Lemma} \label{idemactionNew}
Let $i,j\in I$ and \(x \in X(i)\), \(y \in Y(i)\). 
\begin{enumerate}
\item[{\rm (i)}] $e_ie_j=\de_{i,j}e_i$
\item[{\rm (ii)}] If \(j \not \leq i\), then \(e_j x = y e_j = 0\).
\item[{\rm (ii)}] $yx\equiv f_i(y,x)e_i \pmod{A^{>i}}$
for some $f_i(y,x)\in\k$ with $f_i(e_i,e_i)=1$ and $f_i(y,x)=0$  unless $\deg(x)\deg(y)=1$.
\end{enumerate}
\end{Lemma}

Fix $i\in I$ and denote 
$\tilde A:=A/A^{>i}$, $\tilde a:=a+A^{>i}\in \tilde A$ for $a\in A$.  
By inflation, $\tilde A$-modules will be automatically considered as $A$-modules. The {\em standard module $\De(i)$} and the {\em right standard module $\De^\op(i)$} are defined as
$\De(i):=\tilde A \tilde e_i$ and $\De^\op(i):=\tilde e_i\tilde A$. 
We have that $\De(i)$ and $\De^\op(i)$ are free $\k$-modules with bases $\{v_x\mid x\in X(i)\}$ and $\{w_y\mid y\in Y(i)\}$, respectively, and the actions 
$$av_x=\sum_{x'\in X(i)}l_{x'}^x(a)v_{x'}\quad \text{and}\quad    
w_y a=\sum_{y'\in Y(i)}r_{y'}^y(a)w_{y'}
\qquad(a\in A).
$$ 
In particular, $e_iv_i=v_i$, $e_j\De(i)\neq 0$ implies $j\leq i$, and for all for all $x\in X(i)$ we have  $xv_i=v_x$, $e_i v_x=\de_{x,e_i}v_x$. We have a bilinear pairing $
(\cdot,\cdot)_i:\De(i)\times \De^\op(i)\to \k$ satisfying 
$
(v_x,w_y)_i= f_i(y,x)
$ 
with 
$
\rad \De(i)$
being a submodule of $\De(i)$.

Let $\k$ be a field. Then $L(i):=\De(i)/\rad\De(i)$ is an irreducible $A$-module and $$\{q^n\pi^\eps L(i)\mid i\in I,\ n\in\Z,\ \eps\in\Z/2\}$$ is a complete set of non-isomorphic irreducible graded $A$-supermodules. 
Recalling the ring $R$ from (\ref{ER}), the {\em bigraded decomposition numbers} 
\begin{equation}\label{EDNGr}
[\De(i):L(j)]_{q,\pi}=d_{i,j}(q,\pi):=\sum_{n\in\Z,\,\eps\in\Z/2}d_{i,j}^{n,\eps}q^n\pi^\eps\in R \qquad(i,j\in I),
\end{equation}
are determined from 
\begin{equation}\label{E080318}
d_{i,j}^{n,\eps}:=[\De(i):q^n\pi^\eps L(j)] \qquad(n\in\Z,\ \eps\in\Z/2).
\end{equation} 
Then $d_{ii}(q,\pi)=1$, and $d_{i,j}(q,\pi)\neq 0$ implies $j\leq i$.

\subsection{Additional properties of quasi-hereditary algebras}\label{SQHAReg}
Let $A$ be a based quasi-hereditary $\k$-superalgebra with heredity data $I,X,Y$. We continue reviewing results from \cite{greenOne} that will be needed later. 

A homogeneous anti-involution $\tau$ on $A$ is called {\em standard} if for all $i\in I$ there is a bijection $X(i)\iso Y(i),\ x\mapsto y(x)$ such that $y(e_i)=e_i$ and  
$\tau(x)= y(x)$. 
For a standard anti-involution $\tau$, we have 
$\tau(xy(x'))= x'y(x)$ for all $i\in I,\ x,x'\in X(i)$. 
If $\tau$ is a standard anti-involution on $A$ then $\{xy\mid (x,y)\in Z\}$ is a {\em cellular basis} of $A$ with respect to $\tau$. 

If $e\in A$ is a homogeneous idempotent, we consider the idempotent truncation $\bar A:=eAe$, and denote $\bar a:= eae\in\bar A$ for $a\in A$. We say that $e$ is {\em adapted} 
(with respect to $I,X,Y$) 
if for all $i\in I$ there exist subsets $\bar X(i)\subseteq X(i)$ and $\bar Y(i)\subseteq Y(i)$ such that for all $i\in I$ and $x\in X(i),\, y\in Y(i)$ we have:
\begin{equation*}\label{E230517}
ex= 
\left\{
\begin{array}{ll}
x &\hbox{if $x\in\bar X(i)$,}\\
0 &\hbox{otherwise,}
\end{array}
\right.
\qquad\text{and}\qquad\ \ 
ye=
\left\{
\begin{array}{ll}
y &\hbox{if $y\in\bar Y(i)$,}\\
0 &\hbox{otherwise.}
\end{array}
\right.
\end{equation*}
Set $
\bar I:=\{i\in I\mid \bar X(i)\neq \emptyset\neq \bar Y(i)\}.
$
We refer to $\bar I,\bar X,\bar Y$ as the {\em $e$-truncation} of $I,X,Y$. 
We say that $e$ is {\em strongly adapted} (with respect to $I,X,Y$)  if it is adapted and $
ee_i=e_ie= e_i
$
for all $i\in\bar I$.

\begin{Lemma} \label{L200517_4} 
Let $e\in A$ be an adapted idempotent. 
\begin{enumerate}
\item[{\rm (i)}] If $\tau$ is a standard anti-involution of $A$ such that $\tau(e)=e$, then $\{xy\mid i\in \bar I, x\in\bar X(i), y\in\bar Y(i)\}$ is a cellular basis of $\bar A$ with respect to the restriction $\tau|_{\bar A}$. 

\item[{\rm (ii)}] If $e$ is strongly adapted then $\bar A$ is based quasi-hereditary with heredity data $\bar I$, $\bar X:=\bigsqcup_{i\in \bar I}\bar X(i)$, $\bar Y:=\bigsqcup_{i\in \bar I}\bar Y(i)$. 
\end{enumerate}
\end{Lemma}

\begin{Lemma} \label{L210218} 
Let \(\k\) be a field, and $e\in A$ be an adapted idempotent. 
\begin{enumerate}
\item[{\rm (i)}] $eL(i)=0$ if and only if $e\De(i)\subseteq \rad\De(i)$.
\item[{\rm (ii)}] $eL(i)=0$ if and only if $yex\in A^{>i}$ for all $x\in X(i)$ and $y\in Y(i)$.

\item[{\rm (iii)}] $eL(i)=0$ if and only if $yx\in A^{>i}$ for all $x\in \bar X(i)$ and $y\in \bar Y(i)$.

\item[{\rm (iv)}] $eL(i)=0$ for all $i\in I\setminus \bar I$. 
In particular, there exists a subset $\bar I'\subseteq\bar I$ such that $\{eL(i)\mid i\in \bar I'$ is a complete and irredundant set of irreducible $\bar A$-modules up to isomorphism. 
\end{enumerate}
\end{Lemma}

We now turn to more subtle additional properties of heredity data, which have to do with the super-structure. Symbols $X_\0, Y_\0$ are understood in the sense of (\ref{E210218_2}). 

\begin{Definition}\label{D290517}
{\rm 
Suppose that $\fa\subseteq A_\0$ is a subalgebra. The heredity data $I,X,Y$ of $A$ is {\em $\fa$-conforming} if $I,X_\0,Y_\0$ is a heredity data for $\fa$. 
If, in addition, $A$ is unital and $\a$ is a {\em unital} subalgebra, i.e. $1_\a=1_A$, we call $(A,\a)$ a {\em unital pair}. 
}
\end{Definition}

If the heredity data $I,X,Y$ of $A$ is $\fa$-conforming then $\fa$ is recovered as 
$
\fa=\spa(xy\mid i\in I,\ x\in X(i)_\0,\ y\in Y(i)_\0), 
$
so sometimes we will just speak of a {\em conforming heredity data}. Even though in some sense $\fa$ is redundant in the definition of conormity, it is often convenient to use it. For example, we deal with generalized Schur algebras $T^A_\fa(n,d)$, which will only depend on $A$ and $\fa$, but not on $I,X,Y$.

\begin{Lemma} 
Suppose that $\k$ is a local ring with the maximal ideal $\m$ and the quotient field $F=\k/\m$. Then:
\begin{enumerate}
\item[{\rm (i)}] $A/\m A\cong A\otimes_\k F$ is based quasi-hereditary $F$-superalgebra. 
\item[{\rm (ii)}] For each $i\in I$, denote the corresponding canonical irreducible $A/\m A$-module by  $L_{A/\m A}(i)$ and denote by $L_A(i)$ the $A$-module obtained from $L_{A/\m A}(i)$ by inflation. Then $$\{q^n\pi^\eps L_A(i)\mid i\in I,\ n\in\Z,\ \eps\in\Z/2\}$$ is a complete and irredundant set of irreducible graded $A$-supermodules up to a homogeneous isomorphism.
\end{enumerate} 
\end{Lemma}

If $\k$ is a local ring, we call $A$ {\em basic} if the modules $L_{A/\m A}(i)$ are $1$-dimensional as $F$-vector spaces, equivalently if the modules $L_{A}(i)$ are free of rank $1$ as $\k$-modules. 

If the heredity data $I,X,Y$ of $A$ is $\fa$-conforming then by definition $\fa$ is also based quasi-hereditary and has its own standard $\fa$-modules $\De_\fa(i)$ and simple $\fa$-modules $L_\fa(i)$. The following theorem is proved in \cite[Theorem 4.13]{greenOne}.

\begin{Theorem}\label{MorBasic}
Let $\k$ be local and \(A\) be a based quasi-hereditary graded \(\k\)-superalgebra with $\fa$-conforming heredity data $I,X,Y$. Suppose that $(A,\a)$ is a unital pair. Then there exists an $\fa$-conforming heredity data $I,X',Y'$ with the same ideals $A(\Om)$ and $\fa(\Om)$ and such that the new initial elements $\{e'_i \mid i \in I\}$ are primitive idempotents in \(\mathfrak{a}\)  satisfying $e_ie_i'=e_i'=e_i'e_i$ 
and $e_i'\equiv e_i\pmod{\a^{>i}}$ for all $i\in I$. Moreover, setting $f:=\sum_{i\in I}e_i'$, we have:
\begin{enumerate}
\item[{\rm (i)}] $f$ is strongly adapted with respect to $(I,X',Y')$, so that $\bar A$ is based quasi-hereditary with heredity data $(I,\bar X',\bar Y')$. 
\item[{\rm (ii)}] $(I,\bar X',\bar Y')$ is $\bar\fa$-conforming;
\item[{\rm (iii)}] $\bar \a$ is basic and if \(A_{\overline{1}} \subset J(A)\) then \(\bar{A}\) is a basic as well;
\item[{\rm (iv)}] The functors
$$
\funF_A:\mod{A}\to\mod{\bar A},\ V\mapsto f V
\quad\text{and}\quad 
\funF_\mathfrak{a}:\mod{\mathfrak{a}}\to\mod{\bar{\mathfrak{a}}},\ V\mapsto fV
$$ 
are equivalences of categories, such that 
\begin{align*}
&\funF_A(L_A(i))\cong L_{\bar{A}}(i),\quad \funF_A(\De_A(i))\cong \De_{\bar{A}}(i),\\
 &\funF_\mathfrak{a}(L_\mathfrak{a}(i))\cong L_{\bar{\mathfrak{a}}}(i),\quad \funF_\mathfrak{a}(\De_\mathfrak{a}(i))\cong \De_{\bar{\mathfrak{a}}}(i).
\end{align*}
\end{enumerate}
\end{Theorem}

\section{Combinatorics}\label{SecComb}
We fix $n\in\Z_{>0}$ and $d\in\Z_{\geq 0}$, and a based quasi-hereditary graded \(\k\)-superalgebra $A$ with $\fa$-conforming heredity data $I,X,Y$ and the corresponding heredity basis $B= \bigsqcup_{i\in I}B(i)$. 
We denote by $H$ the set of all non-zero homogeneous elements of $A$. 
Let 
\begin{eqnarray*}\label{EBA}
B_\a&:=\{xy\mid i\in I,\, x\in X(i)_\0,\, y\in Y(i)_\0\},
\\
\label{EBC}
B_\c&:=\{xy\mid i\in I,\, x\in X(i)_\1,\, y\in Y(i)_\1\},
\end{eqnarray*}
so that 
\begin{align}\label{AaBasis}
B= B_{\fa}\sqcup B_{\c} \sqcup B_{\bar 1}. 
\end{align}
Note that $(A,\a)$ is a good pair in the sense of \cite{greenTwo}. We now review the theory developed in that paper following \cite{EK1}.

Without loss of generality, we assume that 
$$I=\{0,1,\dots, \ell\}$$ 
with the total order 
\begin{equation}\label{E190218}
0\prec 1\prec\dots\prec\ell
\end{equation}
 refining the fixed partial order on $I$.

\subsection{Compositions and partitions}\label{SSPar}
We set $\Nodes:= \Z_{>0}\times \Z_{>0}$, 
$\Nodes^I:= I\times\Z_{>0}\times \Z_{>0}$ 
and refer to the elements of $\Nodes$ and $\Nodes^I$ 
as {\em nodes}. 
Define a partial order $\leq$ on $\Nodes$ and $\Nodes^I$ as follows: $(r,s)\leq(r',s')$ if and only if $r\le r'$, and $s\le s'$, and 
\begin{equation}\label{E010617}
\text{$(i,r,s)\leq(i',r',s')$ if and only if 
$i=i'$, $r\le r'$, and $s\le s'$.}
\end{equation} 

We denote by $\La(n)$ (resp. $\La(I)$) the set of all compositions $\la=(\la_1,\la_2,\dots,\la_n)$ (resp. $(\la_0,\la_1,\dots,\la_\ell)$) with non-negative integer parts. 
For such a composition $\la$, we denote by $|\la|$ the sum of its parts, and set 
$$
\La(n,d):=\{\la\in\La(n)\mid |\la|=d\},\quad \La(I,d):=\{\la\in\La(I)\mid |\la|=d\}.
$$ 
We denote by $0\in\La(n,0)$ the composition with all zero parts. 
The Young diagram of $\la\in \La(n,d)$ is  
$$
\Brackets\la:=\{(r,s)\in \Nodes \mid s\leq \la_r\} .
$$ 

We define $\La^I(n,d)$ 
 to be the set of tuples $\bla=(\la^{(0)},\la^{(1)},\dots,\la^{(\ell)})$ of compositions $\la^{(i)}\in \La(n)$ such that $|\bla|:=\sum_{i\in I}|\la^{(i)}|=d$. For \(\bla \in \La^I(n,d)\), we set 
 $$
 \|\bla\|:= (|\la^{(0)}|, \ldots, | \la^{(\ell)}|) \in \La(I,d).
 $$ 
 Let $\bla=(\la^{(0)},\dots,\la^{(\ell)})\in\La^I(n,d)$. 
We denote by
$$
\Brackets\bla=\Brackets{\la^{(0)}}\sqcup\dots\sqcup \Brackets{\la^{(\ell)}}\subset \Nodes^I
$$
the Young diagram of $\bla$, where 
$$
\Brackets{\la^{(i)}}=\{(i,r,s)\in \Nodes^I \mid s\leq \la^{(i)}_r\} \qquad (i\in I).
$$

Let $\unlhd$ be the usual {\em dominance partial order} on $\La(n,d)$, i.e. 
$$
\la\unlhd \mu\quad\text{if and only if}\quad \sum_{r=1}^s\la_r\leq \sum_{r=1}^s\mu_r\quad\text{for all $s=1,\dots,n$}. 
$$
We denote by $\unlhd_I$ the following partial order on $\La(I,d)$: 
\begin{equation}\label{E120218_1}
\la\unlhd_I \mu\quad\text{if and only if}\quad \sum_{j\geq i}\la_j\leq \sum_{j\geq i}\mu_j\quad\text{for all $i\in I$}. 
\end{equation}
We denote by $\leq$ the partial order on $\La^I(n,d)$ defined as follows:
\begin{equation}\label{E140218_2}
\begin{split}
\bla\leq \bmu\quad\text{if and only if}\quad &\text{either\ }\|\bla\|\lhd_I\|\bmu\|
\\ &\text{or $\|\bla\|=\|\bmu\|$ and } \la^{(i)}\unlhd \mu^{(i)}\ \text{for all $i\in I$}. 
\end{split}
\end{equation}

We label the nodes of $\Brackets\bla$ with numbers $1,\dots,d$ going from left to right along the rows, starting with the first row of $\Brackets{\la^{(0)}}$, then going along the second row of $\Brackets{\la^{(0)}}$, and so on until the $n$th row of $\Brackets{\la^{(0)}}$, then along the first row of $\Brackets{\la^{(1)}}$, the second  row of $\Brackets{\la^{(1)}}$, and so on. 
For $1\leq k\leq d$, we denote the $k$th node of $\bla$ by $N_k(\bla)$. 

Let $i\in I$ and $d_i:=|\la^{(i)}|$. We can also label  the nodes of $\Brackets{\la^{(i)}}$ with numbers $1,\dots,d_i$ going from left to right along the rows, starting with the first row of $\Brackets{\la^{(i)}}$, then going along the second row of $\Brackets{\la^{(i)}}$, and so on. 
For $1\leq k\leq d_i$, we denote the $k$th node of $\la^{(i)}$ by $N_k(\la^{(i)})$. Note that $N_k(\la^{(i)})=N_{k+d_0+\dots+d_{i-1}}(\bla)$ for $1\leq k\leq d_i$. 

The {\em row stabilizer of $\la^{(i)}$} is the subgroup $\Si_{\la^{(i)}}\leq \Si_{d_i}$, consisting of all $\si\in\Si_{d_i}$ such that for all $k=1,\dots,d_i$ we have that $N_k(\la^{(i)}),N_{\si k}(\la^{(i)})$ are in the same row of $\Brackets{\la^{(i)}}$. 
The {\em row stabilizer of $\bla$} is the subgroup $\Si_\bla\leq \Si_d$, consisting of all $\si\in\Si_d$ such that for all $k=1,\dots,d$ we have that $N_k(\bla),N_{\si k}(\bla)$ are in the same row of some component $\Brackets{\la^{(i(k))}}$. We have $\Si_\bla\cong \Si_{\la^{(0)}}\times\dots\times \Si_{\la^{(\ell)}}$. 

If $\la\in\La(n)$ and $\la_1\geq\la_2\geq\dots\geq \la_n$ we say that $\la$ is a {\em partition} and write $\la\in\La_+(n)$. If $\bla=(\la^{(0)},\la^{(1)},\dots,\la^{(\ell)})\in\La^I(n,d)$ is such that each $\la^{(i)}$ is a partition, we say that $\bla$ is a {\em multipartition} and write $\bla\in\La_+^I(n,d)$. If $\mu\in\La(n)$, there is a unique partition $\mu_+\in\La_+^I(n,d)$ obtained from $\mu$ by permuting its parts.  

\subsection{Colored letters and tableaux}\label{SSColLet}

We introduce {\em colored alphabets}   
$$
\Alph_{X}:=[1,n]\times X \quad\text{and}\quad  \Alph_{X(i)}:=[1,n]\times X(i),
$$
so that $\Alph_{X}=\bigsqcup_{i\in I}\Alph_{X(i)}$. 
The colored alphabets $\Alph_{Y}$ and $\Alph_{Y(i)}$ are defined similarly. 
We think of elements of $\Alph_{X}$ as $X$-colored letters, and often write $l^x$ instead of $(l,x)\in \Alph_{X}$. 
 If  $L=l^x\in \Alph_{X}$, we denote
$$
\letter(L):=l\qquad\text{and}\qquad \col(L):=x.
$$
For all $i\in I$, we fix arbitrary total orders `$<$' on the sets $\Alph_{X(i)}$ which satisfy $r^x< s^x$ if  $r< s$ (in the standard order on $[1,n]$). 
Similarly we fix total orders on and on the sets $\Alph_{Y(i)}$ with $r^y< s^y$ if $r< s$.

All definitions of this subsection which involve $X$ have obvious analogues for $Y$. 
Let $\bla=(\la^{(0)},\dots,\la^{(\ell)})\in\La^I(n,d)$. Fix $i\in I$ and let $d_i:=|\la^{(i)}|$. 

An {\em $X(i)$-colored $\la^{(i)}$-tableau} is a function
$T:\Brackets{\la^{(i)}}\to \Alph_{X(i)}$ such that the following condition holds:
\begin{enumerate}
\item[$\bullet$] if $M\neq N$ are nodes in the same row of $\Brackets{\la^{(i)}}$, then $T(M) = T(N)$ implies $\col(T(M))\in X(i)_\0$.
\end{enumerate} 
We denote the set of all $X(i)$-colored $\la^{(i)}$-tableaux by $\Tab^{X(i)}(\la^{(i)})$. 

Recall the partial order (\ref{E010617}) on the nodes of $\la^{(i)}$ and a fixed  total order on $\Alph_{X(i)}$. 
Let $T\in \Tab^{X(i)}(\la^{(i)})$. Then $T$ is called {\em row standard} if the following condition holds:
\begin{enumerate}
\item[$\bullet$] If $M< N$ are nodes in the same row of $\Brackets{\la^{(i)}}$, then $T(M) \leq T(N)$.  
\end{enumerate}
On the other hand, $T$ is called {\em column standard} if the following condition holds:
\begin{enumerate}
\item[$\bullet$] If $M< N$ are nodes in the same column of $\Brackets{\la^{(i)}}$, then $T(M) \leq T(N)$ and the equality is allowed only if $\col(T(M))\in X(i)_\1$.
\end{enumerate}
Finally, $T$ is called {\em standard} if it is both row and column standard. Denote 
\begin{eqnarray*}
\Rst^{X(i)}(\la^{(i)})&:=&\{T\in \Tab^{X(i)}(\la^{(i)})\mid \text{$T$ is row standard}\},
\\
\Cst^{X(i)}(\la^{(i)})&:=&\{T\in \Tab^{X(i)}(\la^{(i)})\mid \text{$T$ is column standard}\},
\\
\Std^{X(i)}(\la^{(i)})&:=&\{T\in \Tab^{X(i)}(\la^{(i)})\mid \text{$T$ is standard}\}.
\end{eqnarray*}
Recalling the idempotents $e_i\in X(i)\cap Y(i)$, the {\em initial $\la^{(i)}$-tableau} $T^{\la^{(i)}}$ is
\begin{align*}
T^{\la^{(i)}}: \Brackets{\la^{(i)}}\to \Alph_{X(i)}, \ (i,r,s)\mapsto r^{e_i}.
\end{align*}
Note that $T^{\la^{(i)}}$ is in both $\Std^{X(i)}(\la^{(i)})$ and $\Std^{Y(i)}(\la^{(i)})$. 

For $T\in \Tab^{X(i)}(\la^{(i)})$, recalling the notation $N_k(\la^{(i)})$ from \S\ref{SSPar}, we denote 
\begin{eqnarray}
T_k&:=&T(N_k(\la^{(i)}))\in\Alph_{X(i)}\qquad(1\leq k\leq d_i),
\\
\label{E121116One}
\L^{T}&:=&T_1\cdots T_{d_i}\in \Alph_{X(i)}^{d_i}, 
\\
\L^{\la^{(i)}}&:=&\L^{T^{\la^{(i)}}}. 
\end{eqnarray}
Tableaux $S,T\in\Tab^{X(i)}(\la^{(i)})$ are called {\em row equivalent} if there exists $\si\in\Si_{\la^{(i)}}$ such that for all $k=1,\dots,d_i$, we have $S_k=T_{\si(k)}$. The following is clear:

\begin{Lemma} \label{LStdEquivOne} 
For every $T\in\Tab^{X(i)}(\la^{(i)})$, there exists a unique $S\in\Rst^{X(i)}(\la^{(i)})$ which is row equivalent to $T$.
\end{Lemma}


For a function $\T:\Brackets\bla\to \Alph_{X}$ and $i\in I$, we set  $T^{(i)}:=\T|_{\Brackets{\la^{(i)}}}$ to be the restriction of $\T$ to $\Brackets{\la^{(i)}}\subseteq \Brackets\bla$. We write $\T=(T^{(0)},\dots,T^{(\ell)})$, keeping in mind that the restrictions $T^{(i)}$ determine $\T$ uniquely. An {\em $X$-colored $\bla$-tableau} is a function
$\T:\Brackets\bla\to \Alph_{X}$ such that the restrictions $T^{(i)}$ are $X(i)$-colored $\la^{(i)}$-tableau for all $i\in I$. We denote the set of all $X$-colored $\bla$-tableaux by $\Tab^X(\bla)$. 

Let $\T\in \Tab^X(\bla)$. Then $\T$ is called {\em row standard}  (resp. {\em column standard}, resp. {\em standard}) if so are all the $T^{(i)}$ for $i=0,\dots,\ell$. We use the notation $\Rst^X(\bla)$, $\Cst^X(\bla)$ and $\Std^X(\bla)$ to denote the sets of all row standard, column standard and standard $X$-colored $\bla$-tableaux, respectively. For example, we have the {\em initial $\bla$-tableau} $\T^\bla=(T^{\la^{(0)}},\dots,T^{\la^{(\ell)}})\in \Std^X(\bla)\cap \Std^Y(\bla)$. We denote 
\begin{eqnarray}
\T_k&:=&\T(N_k(\bla))\in\Alph_{X}\qquad(1\leq k\leq d),
\\
\label{E121116}
\L^\T&:=&\T_1\cdots \T_d=\L^{T^{(0)}}\cdots \L^{T^{(\ell)}}\in \Alph_{X}^d, 
\\
\L^\bla&:=&\L^{\T^\bla}=\L^{\la^{(0)}}\cdots \L^{\la^{(\ell)}}. 
\end{eqnarray}

Tableaux $\Stab,\T\in\Tab^X(\bla)$ are called {\em row equivalent} if there exists $\si\in\Si_\bla$ such that for all $k=1,\dots,d$, we have $\Stab_k=\T_{\si(k)}$. The following is clear:

\begin{Lemma} \label{LStdEquiv} 
For every $\T\in\Tab^X(\bla)$, there exists a unique $\Stab\in\Rst^X(\bla)$ which is row equivalent to $\T$.
\end{Lemma}
 
The notions introduced in this section generalize the classical notion of a standard tableau which we now recall. Given $\la\in\La_+(n,d)$, a {\em classical $\la$-tableau} is a function
$T:\Brackets{\la}\to [1,n]$. 
A classical $\la$-tableau $T$ is called {\em standard} if whenever $M< N$ are nodes in the same row of $\Brackets{\la}$, then $T(M) \leq T(N)$, and whenever $M< N$ are nodes in the same column of $\Brackets{\la}$, then $T(M) < T(N)$.

\subsection{Triples}\label{SComb}
For $r,s\in\Z$ we denote $[r,s]:=\{t\in\Z\mid r\leq t\leq s\}$. 
We fix $n\in\Z_{>0}$ and $d\in\Z_{\geq 0}$. 
For a set $Z$, the elements of $Z^d$ are referred to as {\em words} (of length $d$). 
The words are usually written as 
$
z_1z_2\cdots z_d\in Z^d.
$
For $\bz\in Z^d$ and $\bz'\in Z^{d'}$ we denote by $\bz\bz'\in Z^{d+d'}$ the concatenation of $\bz$ and $\bz'$. For 
$z\in Z$, we denote $z^d:=z\cdots z\in Z^d$.

The symmetric group $\Si_d$ acts on the right on $Z^d$ by place permutations:
$$
(z_1\cdots z_d)\si=z_{\si 1}\cdots z_{\si d}.
$$
For $\bz,\bz'\in Z^d$, we write $\bz\sim\bz'$ if $\bz\si=\bz'$ for some $\si\in \Si_d$. If $Z_1,\dots,Z_N$ are sets, then $\Si_d$ acts on $Z_1^d\times\dots\times Z_N^d$ diagonally:
$$
(\bz^1,\dots,\bz^N)\si=(\bz^1\si,\dots,\bz^N\si).
$$
The set of the corresponding orbits is denoted 
$
(Z_1^d\times\dots\times Z_N^d)/\Si_d,
$ 
and the orbit of $(\bz^1,\dots,\bz^N)$ is denoted $[\bz^1,\dots,\bz^N]$. 
We write 
$(\bz^1,\dots,\bz^N)\sim (\bw^1,\dots,\bw^N)$ 
if $[\bz^1,\dots,\bz^N]= [\bw^1,\dots,\bw^N]$.  

Let $P=P_\0\sqcup P_\1$ be a set of homogeneous elements of $A$,  
and $\Seq^P (n,d)$ be the set of all triples 
$$
(\bp,\br, \bs) = ( p_1\cdots p_d,\, r_1\cdots r_d,\, s_1\cdots s_d ) \in  P^d\times[1,n]^d\times[1,n]^d
$$
such that for any $1\leq k\neq l\leq d$ we have 
$(p_k,r_k,s_k)=(p_l,r_l, s_l)$ 
only if $p_k\in P_\0$. 
For $(\bp,\br,\bs)\in\Seq^P(n,d)$, we consider the stabilizer 
$$\Si_{\bp,\br,\bs}:=\{\si\in \Si_d\mid (\bp,\br,\bs)\si=(\bp,\br,\bs)\},
$$
and denote by\, ${}^{\bp,\br,\bs}\D$ the set of the shortest coset representatives for $\Si_{\bp,\br,\bs}\backslash\Si_n$. 

We fix a total order `$<$' on 
$P\times[1,n]\times[1,n]$. 
Then we also have a total order on $\Seq^P(n,d)$ defined as follows: $(\bp,\br,  \bs)< (\bp',\br',  \bs')$ if and only if there exists $l\in[1,d]$ such that $(p_k,r_k,s_k)=(p_k',r_k',s_k')$ for all $k<l$ and $(p_l,r_l,s_l)<(p_l',r_l',s_l')$. Denote
\begin{equation}\label{ESeq0}
\Seq^P_0(n,d)=\{(\bp, \br, \bs) \in \Seq^P(n,d)\mid (\bp, \br, \bs)\leq (\bp, \br, \bs) \sigma\ \text{for all}\ \sigma \in \mathfrak{S}_d\}.
\end{equation}
For $(\bp,\br, \bs) \in \Seq^P(n,d)$, $\bp' \in P^d$ and $\si\in\Si_d$, 
we define
\begin{align*}
\lan\bp, \br,  \bs\ran
&:=\sharp\{(k,l)\in[1,d]^2\mid k<l,\ p_k,p_l\in P_\1,\ (p_k,r_k,s_k)> (p_l,r_l, s_l)\},
\\
\lan \bp, \bp'\ran
&:=\sharp\{(k,l)\in[1,d]^2\mid k>l,\  p_k,p_l'\in P_\1\}.
\\
\lan \si;\bp\ran&:=\sharp\{(k,l)\in[1,d]^2\mid k<l,\  \si^{-1}k>\si^{-1}l,\ p_k,p_l\in P_\1\}.
\end{align*}

Let $(\bb,\br, \bs)\in\Seq^B(n,d)$. 
For $b\in B$ and $r,s\in [1,n]$, we denote 
\begin{align}\label{E190617}
[\bb,\br, \bs]^b_{r,s}&:=\sharp\{k\in[1,d]\mid  (b_k,r_k,s_k)=(b,r,s)\},
\\
[\bb,\br, \bs]^! &:=\prod_{b\in B,\, r,s\in [1,n]} [\bb,\br, \bs]^b_{r,s}!=\prod_{b\in B_\0,\, r,s\in [1,n]} [\bb,\br, \bs]^b_{r,s}!
\\
\label{subalg2}
[\bb,\br, \bs]^!_{\a} &:=\prod_{ b\in B_\a,\ r,s\in [1,n]}[\bb,\br, \bs]^b_{r,s}!,\quad
[\bb,\br, \bs]^!_{\c} :=\prod_{ b\in B_\c,\ r,s\in [1,n]}[\bb,\br, \bs]^b_{r,s}!.
\end{align}

\subsection{Generalized RSK}
Recall the notation introduced in (\ref{EBIXY}) and (\ref{EB(i)}). Let 
\begin{align*}
\textup{Std}_2(I,n,d):=\{(\bla, \Stab, \T) \mid \bla\in\La^I_+(n,d),\ \Stab\in\Std^X(\bla),\ \T\in\Std^Y(\bla)\}.
\end{align*}

 \begin{Lemma}\label{TabBij}
There is a bijection between the sets \(\Seq^B(n,d)/\Si_d\) and  $\textup{Std}_2(I,n,d)$.
\end{Lemma}
\begin{proof}
We first prove a one-color version of the claim. Fix \(i \in I\), and define
\begin{align*}
\La_+^i(n,d) := \{ \bla \in  \La_+^I(n,d) \mid \la^{(j)} = \delta_{i,j}\la^{(i)} \textup{ for all }j \in I\}.
 \end{align*}
 In other words, \( \La^i_+(n,d)\) is the subset of multipartitions concentrated in the \(i\)th component. The colored alphabets \(\Alph_{X(i)}\) and \(\Alph_{Y(i)}\), with orders chosen in \S\ref{SSColLet},  are {\em alphabets} in the terminology of \cite{LNS}. The set of {\em signed two-row arrays} on \(\Alph_{X(i)}\) and \(\Alph_{Y(i)}\) described in \cite[Definition 4.1]{LNS} can be seen to be in bijection with \(\Seq^{B(i)}(n,d)/\Si_{d}\), via the assignment
\begin{align*}
\begin{bmatrix}
    r_1^{x_1}      & r_2^{x_2} & \cdots & r_d^{x_d} \\
    s_1^{y_1}      & s_2^{y_2} & \cdots & s_d^{y_d}
\end{bmatrix}
\mapsto
[b^i_{x_1,y_1} \cdots b^i_{x_d,y_d}, r_1 \cdots r_d, s_1 \cdots s_d].
\end{align*}
It is proved in \cite[Theorem 4.2]{LNS} (translating the main result of \cite{BSV} from the context of the four-fold algebra) that the set of signed two-row arrays on \(\Alph_{X(i)}\) and \(\Alph_{Y(i)}\) are in bijection with the set (in the language of \cite{LNS}) of pairs of same-shape super semistandard Young tableaux on \(\Alph_{X(i)}\) and \(\Alph_{Y(i)}\). In view of \cite[Definition 2.2]{LNS} and \S\ref{SSColLet}, this latter set is in bijection with
\begin{align*}
\textup{Std}_2(i,n,d):=\{ (\bla, S, T) \mid \bla \in \La_+^i(n,d),\, S \in \Std^{X(i)}(\la^{(i)}),\, T \in \Std^{Y(i)}(\la^{(i)})\}.
\end{align*}
Thus, for all \(i \in I\) and \(d \in \ZZ_{\geq 0}\), there is a bijection between \(\Seq^{B(i)}(n,d)/\Si_d\) and \(\textup{Std}_2(i,n,d)\).

Now note that \(\Seq^B(n,d)/\Si_d\) is in bijection with the set
\begin{align*}
\bigsqcup_{\substack{d_0, \ldots, d_\ell \in \ZZ_{\geq 0}\\ d_0 + \cdots + d_\ell =d}}\left(\prod_{i \in I}\Seq^{B(i)}(n,d_i)/\Si_{d_i}\right),
\end{align*}
and restriction of tableaux gives a bijection between \(\textup{Std}_2(I,n,d)\) and 
\begin{align*}
\bigsqcup_{\substack{d_0, \ldots, d_\ell \in \ZZ_{\geq 0}\\ d_0 + \cdots + d_\ell =d}} \left(
\prod_{i \in I}
\textup{Std}_2(i,n,d_i)
\right),
\end{align*}
so the bijection in the general case follows from the one-color case.
\end{proof}

\section{Generalized Schur algebras}\label{SecGenSch}
We continue to work with a fixed $d\in Z_{\geq 0}$, $n\in\Z_{>0}$, and based quasi-hereditary graded \(\k\)-superalgebra $A$ with $\fa$-conforming heredity data $I,X,Y$ and the corresponding heredity basis 
$$B=B_\a\sqcup B_\c\sqcup B_\1=\bigsqcup_{i\in I} B(i),$$
as in (\ref{AaBasis}), (\ref{EB(i)}).  
Define the structure constants $\kappa^b_{a,c}$ of $A$ from
\begin{equation}\label{EStrConst}
ac=\sum_{b\in B}\kappa^b_{a,c} b
\qquad(a,c\in A).
\end{equation}
More generally, for $\bb=b_1\cdots b_d\in B^d$ and $\ba=a_1\cdots a_d,\,\bc=c_1\cdots c_d\in A^d$, we define 
\begin{equation}\label{ESCB}
\ka^\bb_{\ba,\bc}:=\ka^{b_1}_{a_1,c_1}\dots \ka^{b_d}_{a_d,c_d}.
\end{equation}

\subsection{The algebras $S^A(n,d)$ and $T^A_\a(n,d)$}\label{SSA}

The matrix algebra $M_n(A)$ is naturally a superalgebra. For $r,s\in [1, n]$ and $a\in A$, we denote
\begin{equation}\label{EXirs}
\xi_{r,s}^a:=a E_{r,s}\in M_n(A). 
\end{equation}
There is a right action of $\Si_d$ on $M_n(A)^{\otimes d}$ with (super)algebra automorphisms, such that for all $a_1,\dots,a_d\in H$, $r_1,s_1,\dots,r_d,s_d\in [1,n]$ and $\si\in \Si_d$, we have 
$$(\xi_{r_1,s_1}^{a_1}\otimes\dots\otimes \xi_{r_d,s_d}^{a_d})^\si=
(-1)^{\lan\si;\ba\ran} \xi_{r_{\si1},s_{\si1}}^{a_{\si1}}\otimes\dots\otimes \xi_{r_{\si d},s_{\si d}}^{a_{\si d}}.
$$
The algebra $S^A(n,d)$ is defined as the algebra of invariants 
$$
S^A(n,d):=(M_n(A)^{\otimes d})^{\Si_d}.
$$

For $(\ba,\br,\bs)\in\Seq^H(n,d)$, we define elements 
\begin{equation}\label{EXiDef}
\xi_{\br,\bs}^\ba:= \sum_{\si\in{}^{\ba,\br,\bs}\D} 
(\xi_{r_1,s_1}^{a_1} \otimes \cdots \otimes \xi_{r_d,s_d}^{a_d})^\si\in S^A(n,d).
\end{equation}
Then 
$
\{\xi_{\br,\bs}^\bb\mid [\bb,\br,\bs]\in\Seq^B(n,d)/\Si_d\}
$ is a basis of $S^A(n,d)$. As noted in \cite[Lemma 3.3]{greenTwo}, we have:

\begin{Lemma} \label{LXiZero} 
If $(\ba',\br',\bs') \sim (\ba,\br,\bs)$ are elements of  $\Seq^H(n,d)$, then $\xi_{\br',\bs'}^{\ba'}=(-1)^{\lan\ba,\br,  \bs\ran+\lan\ba',\br',  \bs'\ran}\xi_{\br,\bs}^\ba. 
$
\end{Lemma}

For $(\ba,\bp,  \bq),\, (\bc,\bu,  \bv) \in \Seq^H(n,d)$ and $(\bb,\br,  \bs) \in \Seq^B(n,d)$, the structure constants  $f_{\ba,\bp, \bq;\bc,  \bu,\bv}^{\bb,\br,\bs}$ are defined from from
\begin{equation}\label{EFBABBBCNew}
 \xi^\ba_{\bp,  \bq}\, \xi^\bc_{\bu,  \bv}=\sum_{[\bb,\br,\bs]\in\Seq^B(n,d)/\Si_d} f_{\ba,\bp, \bq;\bc,  \bu,\bv}^{\bb,\br,\bs}\,  \xi^\bb_{\br,  \bs}.
\end{equation}

\begin{Proposition} \label{CPR} {\rm \cite[(6.14)]{EK1}} 
Let $(\ba,\bp,  \bq),\, (\bc,\bu,  \bv) \in \Seq^H(n,d)$ and $(\bb,\br,  \bs) \in \Seq^B(n,d)$. Then 
$$
f_{\ba,\bp, \bq;\bc,  \bu,\bv}^{\bb,\br,\bs}= \sum_{\ba', \bc',\bt} 
(-1)^{\lan\ba,\bp,  \bq\ran 
+\lan\bc,\bu,  \bv\ran
+ \lan\ba',\br,  \bt\ran
+\lan\bc',\bt, \bs\ran+\lan\ba',\bc'\ran} \, 
\kappa_{\ba',\bc'}^{\bb},
$$
where the sum is over all 
$\ba', \bc'\in H^d$ and $\bt\in[1,n]$  
such that $(\ba',\br,  \bt) \sim (\ba,\bp,  \bq)$ and $(\bc',\bt,  \bs)\sim (\bc,\bu,\bv)$. 
\end{Proposition}

We set 
\begin{equation}\label{E080717}
\eta^\bb_{\br,\bs}:=[\bb,\br,  \bs]^!_{\c}\, \xi^\bb_{\br,  \bs}, 
\end{equation}
and
$$
T^A_\fa(n,d):=\spa\big(\,\eta^\bb_{\br,\bs}\mid (\bb,\br,\bs)\in\Seq^B(n,d)\,\big).
$$ 
so that 
$\big\{\,\eta^\bb_{\br,\bs}\mid (\bb,\br,\bs)\in\Seq^B(n,d)/\Si_d\,\big\}$ is a basis of $T^A_\fa(n,d)$. It is proved in 
\cite[Proposition 3.11]{greenTwo} that $T^A_\fa(n,d)\subseteq S^A(n,d)$ is a $\k$-subalgebra. Moreover, it is a unital subalgebra if \((A,\a)\) is a unital pair. Sometimes we call the algebra $T^A_\a(n,d)$ a {\em generalized Schur (super)algebra}. 

\begin{Proposition}\label{AaInd} {\rm \cite[Proposition 4.11]{greenTwo}} 
The algebra \(T^A_\fa(n,d)\) depends only on the subalgebra \(\fa\), and not on the choice of the basis \(B\).
\end{Proposition}


\begin{Lemma} \label{L140117} {\rm \cite[Lemma 3.10]{greenTwo}} 
Let $a_1,\dots,a_d\in \fa\cup A_\1$ and $\br,\bs\in[1,n]^d$. Then $\xi^\ba_{\br,\bs}\in T^A_\fa(n,d)$. 
\end{Lemma}

\subsection{Coproducts}\label{SSCoproduct}
If $\Triple=(\bb,\br,\bs)\in\Seq^B(n,d)$, we write 
$$
\xi_\Triple:=\xi^\bb_{\br,\bs},\ \,\,\eta_\Triple:=\eta^\bb_{\br,\bs},\ \,\,[\Triple]^!_{\c}:=[\bb,\br,\bs]^!_{\c},\ \,\,\Triple\si:=(\bb,\br,\bs)\si,\ \,\,\text{etc.}
$$
If $d=d_1+d_2$, $\Triple^1=(\bb^1, \br^1, \bs^1)\in\Seq^B(n,d_1)$ and $\Triple^2=(\bb^2, \br^2, \bs^2)\in\Seq^B(n,d_2)$, we denote 
$$
\Triple^1\Triple^2:=(\bb^1\bb^2, \br^1\br^2, \bs^1\bs^2)\in B^d\times[1,n]^d\times [1,n]^d.
$$

Recalling the notation (\ref{ESeq0}), for \(\Triple \in \Seq^B_0(n,d)\) and \(0 \leq l \leq d\), define 
\begin{align*}
\textup{Spl}_l(\Triple)&:=\big\{(\Triple^1, \Triple^2)\in \Seq^B_0(n,l) \times \Seq^B_0(n,d-l) \mid 
\Triple^1\Triple^2\sim \Triple\big\},
\end{align*}
and set $\textup{Spl}(\Triple):=\bigsqcup_{0\leq l\leq d}\textup{Spl}_l(\Triple)$. For $(\Triple^1,\Triple^2) \in \textup{Spl}_l(\Triple)$, let \(\si^{\Triple}_{\Triple^1,\Triple^2}\) be the unique element of ${}^{\Triple}\mathscr{D}$ such that 
$$
\Triple\si^{\Triple}_{\Triple^1,\Triple^2}  = \Triple^1\Triple^2.
$$ 

It is well known that \(\bigoplus_{d\geq 0} M_n(A)^{\otimes d}\) is a supercoalgebra with the coproduct 
\begin{align*}
\nabla\,:\,\, M_n(A)^{\otimes d}\,\, &\to \,\,\bigoplus_{l=0}^d M_n(A)^{\otimes l} \otimes M_n(A)^{\otimes (d-l)}\\
\xi_1 \otimes \cdots \otimes \xi_d\,\,&\mapsto\,\, \sum_{l=0}^d (\xi_1 \otimes \cdots \otimes \xi_l) \otimes (\xi_{l+1} \otimes \cdots \otimes \xi_d), 
\end{align*}
see e.g. \cite[\S3.3]{EK1}. Let 
\begin{equation}\label{ES(n)}
S^A(n) := \bigoplus_{d\geq 0} S^A(n,d)\quad \text{and}\quad T^A_\fa(n) := \bigoplus_{d\geq 0} T^A_\fa(n,d).
\end{equation}

\begin{Lemma}\label{coprodxi} {\rm \cite[(6.12)]{EK1}\,\cite[Corollary 3.20]{greenTwo}} 
If \(\Triple=(\bb,\br,\bs) \in \Seq^B_0(n,d)\) then 
\begin{align*}
\nabla(\xi_\Triple) &=
 \sum_{(\Triple^1, \Triple^2) \in \textup{Spl}(\Triple)} 
(-1)^{\langle \si^{\Triple}_{\Triple^1,\Triple^2}; \bb \rangle}
\xi_{\Triple^1} \otimes \xi_{\Triple^2},
\\
\nabla(\eta_\Triple)&=\sum_{(\Triple^1,\Triple^2) \in \textup{Spl}(\Triple)} 
(-1)^{\langle \si^{\Triple}_{\Triple^1,\Triple^2}; \bb \rangle}
{\small \frac{[\Triple]^!_{\c}}{[\Triple^1]^!_{\c}[\Triple^2]^!_{\c}}}
\eta_{\Triple^1} \otimes \eta_{\Triple^2},
\end{align*}
with $\displaystyle{\small \frac{[\Triple]^!_{\c}}{[\Triple^1]^!_{\c}[\Triple^2]^!_{\c}}}\in\Z$. In particular, $S^A(n)$ and $T^A_\fa(n)$ are sub-supercoalgebras of \(\bigoplus_{d\geq 0} M_n(A)^{\otimes d}\). 
\end{Lemma}

\subsection{Star-product}
For $d,e\in Z_{\geq 0}$, let ${}^{(d,e)}\D$ be the set of the shortest coset representatives for $(\Si_d\times \Si_e)\backslash\Si_d$. Given $\xi_1\in M_n(A)^{\otimes d}$ and $\xi_2\in M_n(A)^{\otimes e}$, we define 
\begin{equation}\label{EStarNotation}
\xi_1* \xi_2:=\sum_{\si\in{}^{(d,e)}\D}(\xi_1\otimes \xi_2)^\si.
\end{equation}
This $*$-product  makes  $\bigoplus_{d\geq 0}M_n(A)^{\otimes d}$ into an associative supercommutative superalgebra. Moreover, 

\begin{Lemma} \label{L250217_3} {\rm \cite[Corollary 4.4]{greenTwo}} 
We have that $S^A(n)$ and $T^A_\a(n)$ are subsuperalgebras of $\bigoplus_{d\geq 0}M_n(A)^{\otimes d}$ with respect to the $*$-product. Moreover, with respect to the coproduct $\nabla$ and the product $*$, $S^A(n)$ and $T^A_\a(n)$ are superbialgebras. 
\end{Lemma}

\begin{Lemma}\label{xistarproducts} 
{\rm \cite[Lemma 4.2]{greenTwo}} 
For \((\bb, \br, \bs)\in\Seq^B(n,d)\) and \((\bc, \bt, \bu) \in \Seq^B(n,e)\), we have
\begin{enumerate}
\item \(\displaystyle\xi^{\bb}_{\br, \bs} * \xi^{\bc}_{\bt, \bu} = \frac{[\bb\bc, \br\bt, \bs\bu]^!}{[\bb, \br, \bs]^! [\bc, \bt, \bu]^!}\xi_{\br\bt, \bs\bu}^{\bb\bc}\).
\item \(\displaystyle\eta^{\bb}_{\br, \bs} * \eta^{\bc}_{\bt, \bu} = \frac{[\bb\bc, \br\bt, \bs\bu]_\a^!}{[\bb, \br, \bs]_\a^! [\bc, \bt, \bu]_\a^!}\eta_{\br\bt, \bs\bu}^{\bb\bc},\)
\end{enumerate}
where \(\frac{[\bb\bc, \br\bt, \bs\bu]^!}{[\bb, \br, \bs]^! [\bc, \bt, \bu]^!}\) and \(\frac{[\bb\bc, \br\bt, \bs\bu]_\a^!}{[\bb, \br, \bs]_\a^! [\bc, \bt, \bu]_\a^!}\) are integers, and the right hand sides of \textup{(i)} and \textup{(ii)} are taken to be zero when \((\bb\bc, \br\bt, \bs\bu) \notin \Seq^B(n,d+e)\).
\end{Lemma}

There is a special case where we can guarantee that the coefficients in the right hand sides of the expressions from Lemma~\ref{xistarproducts} are equal to $1$. To describe it, let $q\in\Z_{>0}$ and $\de=(d_1,\dots,d_q)\in\Z_{\geq 0}^q$ with $d_1+\dots+d_q=d$. Then $\Si_\de:=\Si_{d_1}\times\dots\times\Si_{d_q}\leq \Si_d$. 
Suppose that for each $m=1,\dots,q$, we are given 
$$
(\ba^{(m)},\br^{(m)}, \bs^{(m)}),\,(\bc^{(m)},\bt^{(m)}, \bu^{(m)})
\in\Seq^H(n,d_m).
$$ 
We write $\ba^{(m)}=a^{(m)}_1\cdots a^{(m)}_{d_m}$, $\br^{(m)}=r^{(m)}_1\cdots r^{(m)}_{d_m}$, etc. 
Let 
$
\ba=\ba^{(1)}\dots\ba^{(q)},\ \br=\br^{(1)}\dots\br^{(q)},
\ \text{etc.}
$ 
We write $\ba=a_1\cdots a_{d}$, $\br=r_1\cdots r_d$, etc. 
The triple $(\ba,\br, \bs)$ is called {\em $\de$-separated} if 
$1\leq m\neq l\leq q$ implies 
$(a^{(m)}_t,r^{(m)}_t,s^{(m)}_t)\neq (a^{(l)}_u,r^{(l)}_u,s^{(l)}_u)$ for all $1\leq t\leq d_m$ and $1\leq u\leq d_l$. Note that we then  automatically have  $(\ba,\br, \bs)
\in\Seq^H(n,d)$.

\begin{Lemma} \label{LSingleSep} {\rm \cite[Lemma 4.6]{greenTwo}} 
If $(\ba,\br, \bs)$ is $\de$-separated then  
$$
\xi^\ba_{\br,\bs}=\xi_{\br^{(1)},\bs^{(1)}}^{\ba^{(1)}}*\dots*\xi_{\br^{(q)},\bs^{(q)}}^{\ba^{(q)}}
\quad\text{and}\quad 
\eta^\ba_{\br,\bs}=\eta_{\br^{(1)},\bs^{(1)}}^{\ba^{(1)}}*\dots*\eta_{\br^{(q)},\bs^{(q)}}^{\ba^{(q)}}.
$$
\end{Lemma}

\begin{Lemma} \label{LSep} 
{\rm \cite[Lemma 4.7]{greenTwo}}
Let $(\ba,\br, \bs)$ and $(\bc,\bt, \bu)$ be $\de$-separated and suppose that 
$$
(\xi_{\br^{(1)},\bs^{(1)}}^{\ba^{(1)}}\otimes\dots\otimes\xi_{\br^{(q)},\bs^{(q)}}^{\ba^{(q)}})^\si(\xi_{\bt^{(1)},\bu^{(1)}}^{\bc^{(1)}}\otimes\dots\otimes\xi_{\bt^{(q)},\bu^{(q)}}^{\bc^{(q)}})^{\si'}=0
$$
whenever $\si$ and $\si'$ are distinct elements of ${}^\de\D$. Then 
$$
\xi^\ba_{\br,\bs}\xi^\bc_{\bt,\bu}=\pm(\xi_{\br^{(1)},\bs^{(1)}}^{\ba^{(1)}}\xi_{\bt^{(1)},\bu^{(1)}}^{\bc^{(1)}})*\dots*(\xi_{\br^{(q)},\bs^{(q)}}^{\ba^{(q)}}\xi_{\bt^{(q)},\bu^{(q)}}^{\bc^{(q)}}).
$$
Moreover, if $a_1,\dots,a_d$ or $c_1,\dots, c_d$ are all even, then the sign in the right hand side is $+$. 
\end{Lemma}

\subsection{Idempotents}
\label{SSIdAnt}
Let $\la\in\La(n,d)$. Set  
\begin{equation}\label{E150288}
\bl^\la:=1^{\la_1}\cdots n^{\la_n}\in [1,n]^d.
\end{equation}
If $e\in A$ be an idempotent, define 
\begin{equation}\label{E160918}
\xi(\la,e):=\xi^{e^d}_{\bl^\la,\bl^\la}=\xi^{e^{\la_1}}_{1^{\la_1},1^{\la_1}}*\cdots*\xi^{e^{\la_n}}_{n^{\la_n},n^{\la_n}}\in S^A(n,d),
\end{equation}
where the last equality follows from Lemma~\ref{LSingleSep}. 
The element $\xi(\la,e)$ was denoted $\xi_\la^e$ in \cite[\S5.1]{greenTwo}, where we have noted that it is an idempotent. 
Note using Lemma~\ref{L140117} that $\xi(\la,e)\in T^A_\a(n,d)$ if $e\in\a$. 
If $A$ is unital, we denote $\xi(\la):=\xi(\la,{1_A})$. Then
$
1_{S^A(n,d)}=\sum_{\la\in\La(n,d)}\xi(\la)
$
is an orthogonal idempotent decomposition. If the pair $(A,\a)$ is unital, then $\xi(\la)\in T^A_\a(n,d)$ for all $\la\in\La(n,d)$.

Let $\bla\in\La^I(n,d)$ with $\| \bla\|=(d_0,\dots,d_\ell)$. We 
define  
\begin{align}\label{E150288Bold}
&\bl^\bla:=\bl^{\la{(0)}}\cdots \bl^{\la^{(\ell)}}\in [1,n]^d,
\\
\label{E150218_1}
&e_\bla:=\xi_{\bl^\bla,\bl^\bla}^{e_0^{d_0}\cdots e_\ell^{d_\ell}}=\xi(\la^{(0)},e_0)*\cdots*\xi(\la^{(\ell)},e_\ell),
\end{align}
where the last equality follows from Lemma~\ref{LSingleSep}. 
We have noted in \cite[\S5.1]{greenTwo} that $e_\bla$ 
is an idempotent. 

For $\bb\in B^d$ and $\br\in[1,n]^d$, we define multicompositions 
$\bal(\bb,\br)=(\al^{(0)},\dots,\al^{(\ell)})$ and  $\bbe(\bb,\br)=(\be^{(0)},\dots,\be^{(\ell)})$ in $\La^I(n)$ via
\begin{equation}\label{E150218_5}
\begin{split}
\al^{(i)}_s&:=\sharp\{k\in[1,d]\mid r_k=s\ \text{and}\ e_ib_k=b_k\},
\\
\be^{(i)}_s&:=\sharp\{k\in[1,d]\mid r_k=s\ \text{and}\ b_ke_i=b_k\}.
\end{split}
\end{equation}
Note that $\bal(\bb,\br)\in\La^I(n,f)$ and $\bbe(\bb,\br)\in\La^I(n,f')$ for some $0\leq f,f'\leq d$. The following result follows easily from Lemma~\ref{LSep}:

\begin{Lemma} \label{L150218_6} 
Let $\bla\in\La^I(n,d)$ and $(\bb,\br,\bs)\in\Seq^B(n,d)$. Then
\begin{align*}
&e_\bla\xi^\bb_{\br,\bs}=\de_{\bla,\bal(\bb,\br)}\xi^\bb_{\br,\bs},
&\xi^\bb_{\br,\bs}e_\bla=\de_{\bla,\bbe(\bb,\bs)}\xi^\bb_{\br,\bs},
\\
&e_\bla\eta^\bb_{\br,\bs}=\de_{\bla,\bal(\bb,\br)}\eta^\bb_{\br,\bs},
&\eta^\bb_{\br,\bs}e_\bla=\de_{\bla,\bbe(\bb,\bs)}\eta^\bb_{\br,\bs}. 
\end{align*}
\end{Lemma}

\subsection{Multiplication lemma} \label{SSLeading}
Throughout the subsection we fix $i\in I$. 
If $\bu=(u_1,\dots,u_d)\in[1,n]^d$ and $\bx=x_1\cdots x_d\in X(i)^d$, we denote 
$$
\bu^\bx:=u_1^{x_1}\cdots u_d^{x_d}\in\Alph_{X(i)}^d.
$$
Recall  a total order `$<$' on $\Alph_{X(i)}$ from \S\ref{SSColLet}. This total order induces the lexicographical order `$<$' on $\Alph_{X(i)}^d$.

\begin{Lemma} \label{LTwoParts} 
Let $g,f\in\Z_{\geq 0}$ with $d=g+f$, and $r,s,t\in[1, n]$ with $s\neq t$. Set 
\begin{align*}
\by&:=e_i^d,\ \br:=r^d,\ \bs:=s^g,\   \bt:=t^f.
\end{align*} 
Let 
\begin{align*}
&\bp=p_1\cdots p_g\in[1,n]^g,\  \bq=q_1\cdots q_f\in [1,n]^f,\\ 
&\bx^1=x_1^1\cdots x_g^1\in X(i)^g,\ \bx^2=x_1^2\cdots x_f^2\in X(i)^f
\end{align*}
be such that $(\bx^1\bx^2,\bp\bq,\br)\in\Seq^{X(i)}(n,d)$ and 
$$q_1^{x_1^2}\leq\dots\leq q_f^{x_f^2}\leq p_1^{x_1^1}\leq\dots\leq p_g^{x_g^1}.$$ 
Then $(\bx^1\bx^2,\bp\bq,\bs\bt)\in\Seq^{B(i)}(n,d)$ and 
$$
\xi_{\bp\bq,\br}^{\bx^1\bx^2}\xi_{\br,\bs\bt}^\by=\xi_{\bp\bq,\bs\bt}^{\bx^1\bx^2}+(*),
$$
where {\rm (*)} is a linear combination of  $\xi_{\bu,\bs\bt}^{\bx}$ with 
$\bu^\bx\sim \bp^{\bx^1}\bq^{\bx^2}$
and 
$
\bu^\bx< \bp^{\bx^1}\bq^{\bx^2}.
$
\end{Lemma}
\begin{proof}
The property $(\bx^1\bx^2,\bp\bq,\bs\bt)\in\Seq^{B(i)}(n,d)$ easily follows from the assumption $(\bx^1\bx^2,\bp\bq,\br)\in\Seq^{X(i)}(n,d)$. 
By Proposition~\ref{CPR}, we have
\begin{equation}\label{EFirstSum}
\xi_{\bp\bq,\br}^{\bx^1\bx^2}\xi_{\br,\bs\bt}^\by=\sum_{[\bb, \bu,\bv]\in  \Seq^{B} (n,d)/\Si_d} f^{\bb,\bu,\bv} \xi_{\bu,  \bv}^{\bb},
\end{equation}
for
$$
f^{\bb,\bu,\bv}= \sum_{\ba', \bc', \bt'} 
(-1)^{
\lan \bx^1\bx^2, \bp\bq, \br\ran
+\lan \by,\br,  \bs\bt\ran
+ \lan\ba',\bu,  \bt'\ran 
+\lan\bc',\bt', \bv\ran
+\lan\ba',\bc'\ran} \, 
\kappa_{\ba',\bc'}^{\bb},
$$
where the last sum is over all triples
$(\ba', \bc', \bt') \in H^d \times H^d\times [1,n]^d$
such that $(\ba',\bu,  \bt') \sim (\bx^1\bx^2,\bp\bq,  \br)$ and $(\bc', \bt', \bv)\sim (\by,\br,\bs\bt)$. It follows that $\bt'=\br$, $\bc'=\by$ and $\ba'\in X(i)^d$. Observing that $\kappa_{\ba',\by}^{\bb}=\de_{\bb,\ba'}$, we now deduce that  
$$
f^{\bb,\bu,\bv}= 
\left\{
\begin{array}{ll}
0 &\hbox{if $\bv\not\sim \bs\bt$,}
\\
\displaystyle \sum_{\ba'} 
(-1)^{
\lan \bx^1\bx^2, \bp\bq, \br\ran
+ \lan\ba',\bu,  \br\ran 
} \, 
\de_{\bb,\ba'} &\hbox{if $\bv\sim \bs\bt$,}
\end{array}
\right.
$$
where the sum is over all $\ba'\in X(i)^d$
such that $(\ba',\bu) \sim (\bx^1\bx^2,\bp\bq)$. 

Since the sum in (\ref{EFirstSum}) is over orbit representatives, by the previous paragraph we may assume that $\bv=\bs\bt$ and $(\bb,\bu)\sim (\bx^1\bx^2,\bp\bq)$. Moreover, acting if necessary with the stabilizer $\Si_g\times\Si_f$ of $\bs\bt$, we may assume that $\bb=x_1\cdots x_d$ for some permutation $(x_1,\dots,x_d)$ of $(x_1^1,\dots,x^1_g,x^2_1,\dots,x^2_f)$ such that $u_1^{x_1}\leq\dots \leq u_g^{x_g}$ and $u_{g+1}^{x_{g+1}}\leq\dots \leq u_d^{x_d}$. It follows that $\bu^\bx\sim \bp^{\bx^1}\bq^{\bx^2}$, 
$
\bu^\bx\leq \bp^{\bx^1}\bq^{\bx^2}
$
and the equality $
\bu^\bx = \bp^{\bx^1}\bq^{\bx^2}
$ is only possible if $(\bb,\bu,\bv)=(\bx^1\bx^2,\bp\bq,\bs\bt)$. In the latter case, the formula in the previous paragraph yields $f^{\bx^1\bx^2,\bp\bq,\bs\bt}=1$, completing the proof. 
\end{proof}

\section{Codeterminants}\label{SecCod}
We continue to work with a fixed $d\in Z_{\geq 0}$, $n\in\Z_{>0}$, and based quasi-hereditary graded \(\k\)-superalgebra $A$ with $\fa$-conforming heredity data $I,X,Y$ and the corresponding heredity basis 
$B$. Fix $\bla=(\la^{(0)},\dots,\la^{(\ell)})\in\La^I(n,d)$. 

\subsection{Single colored codeterminants}\label{SSDFPOne}
Throughout the subsection we fix $i\in I$ and set 
$$
\mu:=\la^{(i)}, \quad c:=|\mu|.
$$ 
Let 
$$
T^A_\a(n,c){}_i:=\spa(\eta_{\bp,\bq}^\bb\mid (\bb,\bp,\bq)\in\Seq^{B(i)}(n,c))\subseteq T^A_\a(n,c).$$ 

Recall the combinatorial notions introduced in \S\ref{SSColLet}. Let $S\in \Tab^{X(i)}(\mu)$ and $T\in \Tab^{Y(i)}(\mu)$ with
$$
\L^S=r_1^{x_1}\cdots r_{c}^{x_{c}}\in \Alph_{X(i)}^{c}\quad\text{and}\quad \L^T=s_1^{y_1}\cdots s_{c}^{y_{c}}\in \Alph_{Y(i)}^{c},
$$ 
see (\ref{E121116One}). We define
\begin{align*}
&\bx^S:=x_1\cdots x_{c}\in X(i)^{c},  &\bl^S:=r_1\cdots r_{c}\in [1,n]^{c},
\\
&\by^T:=y_1\cdots y_{c}\in Y(i)^{c}, &\bl^T:=s_1\cdots s_{c}\in [1,n]^{c}.
\end{align*}
For the initial $\mu$-tableau $T^{\mu}$, set 
$
\bl^{\mu}:=\bl^{T^{\mu}}=1^{\mu_1}\cdots n^{\mu_n}.
$  This agrees with (\ref{E150288}). 
We now define
$$
\X_S:=\xi^{\bx^S}_{\bl^S,\bl^{\mu}},\quad  
\Y_T:=\xi^{\by^T}_{\bl^\mu,\bl^T},\quad 
\B^\mu_{S,T}:=\X_S \Y_T.
$$
We refer to the elements $\B^\mu_{S,T}$ as {\em codeterminants of color $i$}, cf. \cite{GreenCod}. Note that $X(i),Y(i)\subseteq B_\a\sqcup B_\1$ so $\X_S=\eta^{\bx^S}_{\bl^S,\bl^{\mu}}\in T^A_\a(n,c)$ and  $\Y_T=\eta^{\by^T}_{\bl^\mu,\bl^T}\in T^A_\a(n,c)$. Therefore $\B^\mu_{S,T}\in T^A_\a(n,c)$. Since $xy\in B(i)$ whenever $x\in X(i)$ and $y\in Y(i)$, it now follows that $\B^\mu_{S,T}\in T^A_\a(n,c)_i$. 

We refer to $\mu$ as the {\em shape} of the codeterminant $\B^\mu_{S,T}$. 
A codeterminant $\B^\mu_{S,T}$ is called {\em dominant} if $\mu\in\La_+(n,c)$, i.e. if its shape is a partition. A codeterminant $\B^\mu_{S,T}$ is called {\em standard} if it is dominant and $S\in\Std^{X(i)}(\mu)$ and $T\in\Std^{Y(i)}(\mu)$. 

\begin{Lemma} \label{120117One} 
If $S,S'\in \Tab^{X(i)}(\mu)$ are row equivalent, then $\X_S=\pm\X_{S'}$. If $T,T'\in \Tab^{Y(i)}(\mu)$ are row equivalent, then $\Y_T=\pm\Y_{T'}$. 
\end{Lemma}
\begin{proof}
This follows from Lemma~\ref{LXiZero}. 
\end{proof}

For $w\in \Si_n$ and $\nu=(\nu_1,\dots,\nu_n)\in\La(n)$, we define 
\begin{equation}\label{E250217_4}
w\nu:=(\nu_{w^{-1}1},\dots , \nu_{w^{-1}n})\in\La(n).
\end{equation}
Note that the rows of the Young diagram $\Brackets\nu$ are obtained by a permutation of the rows of the Young diagram $\Brackets{w\nu}$, and this permutation of rows defines a bijection $\phi_w:\Brackets{w\nu}\to \Brackets\nu$.  
If $S:\Brackets\nu\to \Alph$ is a function from the Young diagram of $\nu$ to some set $\Alph$, we denote 
\begin{equation}\label{E250217_45}wS:=S\circ \phi_w.
\end{equation} 
With this notation we have

\begin{Lemma} \label{LDomOne} 
If $w\in \Si_n$ then $\B^\mu_{S,T}= \pm\B^{w\mu}_{wS,wT}$. 
\end{Lemma}
\begin{proof}
We can write 
\begin{align*}
\bx^{S}=\bx^{S}_1\cdots\bx^{S}_n, \quad 
\bl^{S}=\bl^{S}_1\cdots\bl^{S}_n,\quad 
\by^{T}=\by^{T}_1\cdots\by^{T}_n, \quad 
\bl^{T}=\bl^{T}_1\cdots\bl^{T}_n,
\end{align*}
where the words $\bx^{S}_k,\bl^{S}_k,\by^{T}_k,\bl^{T}_k$ have length $\mu_k$ for $k=1,\dots, n$.  For $m=1,\dots,n$, denote $k_m:=w^{-1}m$. Note that $w\mu=(\mu_{k_1},\dots,\mu_{k_n})$ and 
\begin{align*}
\bx^{wS}=\bx^{S}_{k_1}\cdots\bx^{S}_{k_n}, \ 
\bl^{wS}=\bl^{S}_{k_1}\cdots\bl^{S}_{k_n},
\ 
\by^{wT}=\by^{T}_{k_1}\cdots\by^{T}_{k_n}, \ 
\bl^{T}=\bl^{T}_{k_1}\cdots\bl^{T}_{k_n},
\end{align*}
We now get 
\begin{align*}
 \B^{\mu}_{S,T}
 &=\xi^{\bx^{S}}_{\bl^{S},\bl^{\mu}}\xi^{\by^{T}}_{\bl^{\mu},\bl^{T}}
\\
&=(\xi^{\bx^{S}_1}_{\bl^{S}_1,1^{\mu_1}}\xi^{\by^{T}_1}_{1^{\mu_1},\bl^{T}_1})*\cdots*
(\xi^{\bx^{S}_n}_{\bl^{S}_n,n^{\mu_n}}\xi^{\by^{T}_n}_{n^{\mu_n},\bl^{T}_n})
\\
&=\pm (\xi^{\bx^{S}_{k_1}}_{\bl^{S}_{k_1},k_1^{\mu_{k_1}}}\xi^{\by^{T}_{k_1}}_{k_1^{\mu_{k_1}},\bl^{T}_{k_1}})*\cdots*
(\xi^{\bx^{S}_{k_n}}_{\bl^{S}_{k_n},k_n^{\mu_{k_n}}}\xi^{\by^{T}_{k_n}}_{k_n^{\mu_{k_n}},\bl^{T}_{k_n}})
\\
&=\pm (\xi^{\bx^{S}_{k_1}}_{\bl^{S}_{k_1},1^{\mu_{k_1}}}\xi^{\by^{T}_{k_1}}_{1^{\mu_{k_1}},\bl^{T}_{k_1}})*\cdots*
(\xi^{\bx^{S}_{k_n}}_{\bl^{S}_{k_n},n^{\mu_{k_n}}}\xi^{\by^{T}_{k_n}}_{n^{\mu_{k_n}},\bl^{T}_{k_n}})
\\
&=\pm\xi^{\bx^{wS}}_{\bl^{wS},\bl^{w\mu}}\xi^{\by^{wT}}_{\bl^{w\mu},\bl^{wT}}=\pm  \B^{w\mu}_{wS,wT},
\end{align*}
where the first and the last equations are  by definition, the second and the penultimate equations come from Lemma~\ref{LSep}, the third equation is by the supercommutativity of the $*$-product, and the fourth equality holds by 
Proposition~\ref{CPR}.  
\end{proof}

Note that we can always pick $w$ in the previous lemma so that $w\mu\in\La_+(n,c)$. So, when working with codeterminants $\B^\mu_{S,T}$, we can usually assume that $\mu\in\La_+(n,c)$. In addition, in view of Lemmas~\ref{120117One} and \ref{LStdEquivOne}, we can usually assume that $S$ and $T$ are row standard. 
For example, the following result shows that for $n\geq c$, $T^A_\a(n,c){}_i$ is spanned by dominant codeterminants of color $i$ corresponding to row standard tableaux. 

\begin{Proposition}\label{P260117One}
Let $n \geq c$, and $(\bb,\bp,\bq)\in\Seq^{B(i)}(n,c)$. Then $\eta_{\bp,\bq}^\bb=\pm \B^{\mu}_{S,T}$ for some 
$\mu\in\La_+(n,c)$, $S\in\Rst^{X(i)}(\mu)$ and $T\in\Rst^{Y(i)}(\mu)$. 
\end{Proposition}

\begin{proof}
In view of Lemmas~\ref{120117One} and \ref{LDomOne}, it suffices just to prove that $\eta_{\bp,\bq}^\bb=\pm \B^{\mu}_{S,T}$ for some 
$\mu\in\La(n,c)$, $S\in\Tab^{X(i)}(\mu)$ and $T\in\Tab^{Y(i)}(\mu)$. 

Let $\bb=b_1\cdots b_c$, $\bp=p_1\cdots p_c$, $\bq=q_1\cdots q_c$. For $x\in X(i)$, $y\in Y(i)$ and 
$r,s\in[1, n]$, recalling (\ref{E190617}),  denote
$$
m^{x,y}_{r,s}:=[\bb,\bp,\bq]^{b^i_{x,y}}_{r,s}.
$$
Note that $m^{x,y}_{r,s}\leq 1$ if $\bar x\neq \bar y$. 
Put  
$$
Q:=\{(x,y,r,s)\mid x\in X(i),\ y\in Y(i),\ r,s\in[1, n],\ m^{x,y}_{r,s}\neq 0\}.$$ 
Pick a total order on $Q$ and write 
$$
Q=\{(x_1,y_1,r_1,s_1)<\dots<(x_{t},y_{t},r_{t},s_{t})\}.
$$
To every $(x,y,r,s)\in Q$, we associate a composition $\nu^{x,y}_{r,s}$ of $m^{x,y}_{r,s}$ as follows:
$$
\nu^{x,y}_{r,s}=
\left\{
\begin{array}{ll}
(m^{x,y}_{r,s}) &\hbox{if $\bar x=\bar y=\0$,}\\
(1^{m^{x,y}_{r,s}}) &\hbox{otherwise.}
\end{array}
\right.
$$
Now we define $\mu\in\La(n,c)$ as the concatenation 
$$
\mu:=\nu^{x_1,y_1}_{r_1,s_1}\,\dots\,\nu^{x_{t},y_{t}}_{r_{t},s_{t}}\,0^{u}\in\La(n,c).
$$
where $u\in\Z_{\geq 0}$ is chosen so that $\mu$ has $n$ parts. 
Let $S$ be the $\mu$-tableau which associates to the nodes of each $\nu^{x,y}_{r,s}$ the value $r^x\in\Alph_{X(i)}$ and let $T$ be the $\mu$-tableau which associates to the nodes of each $\nu^{x,y}_{r,s}$ the value $s^y\in\Alph_{Y(i)}$. 

Let 
$
\de=(m^{x_1,y_1}_{r_1,s_1},\dots,m^{x_{t},y_{t}}_{r_{t},s_{t}})\in\La(t,c).
$ 
We can decompose any word $\bl$ of length $c$ as the concatenation
$
\bl=\bl_{1}\dots\bl_{t},
$ 
where, for $1\leq u\leq t$, the length of the word $\bl_{u}$ is $m^{x_u,y_u}_{r_u,s_u}$. We will apply this to the words $\bb,\bp,\bq,\bx^S,\bl^S,
\by^T,\bl^T,\bl^\mu$. 
Note that the triples $(\bx^S,\bl^S,\bl^\mu)$ and $(\by^S,\bl^\mu,\bl^T)$ satisfy the assumptions of Lemma~\ref{LSep}, so 
\begin{equation}\label{E140218}
\B^\mu_{S,T}=\xi^{\bx^S}_{\bl^S,\bl^\mu} \xi^{\by^T}_{\bl^\mu,\bl^T}
=
\pm(\xi^{\bx^S_{1}}_{\bl^S_{1},\bl^\mu_{1}}
\xi^{\by^T_{1}}_{\bl^\mu_{1},\bl^T_{1}})*\dots*
(\xi^{\bx^S_{t}}_{\bl^S_{t},\bl^\mu_{t}}
\xi^{\by^T_{t}}_{\bl^\mu_{t},\bl^T_{t}}). 
\end{equation}

Let $1\leq u\leq t$. 
Denote $m:=m^{x_u,y_u}_{r_u,s_u}$. Then 
$$\bx^{S}_{u}=x_u^m,\ \bl^S_{u}=r_u^m,\ 
\by^T_{u}=y_u^m,\ \bl^T_{u}=s_u^m.
$$ 
If $\bar x_u=\bar y_u=\0$, then $\bl^\mu_{u}=v^m$ for some $1\leq v\leq n$, and in this case, using (\ref{EXiDef}), we get 
\begin{align*}
\xi^{\bx^S_{u}}_{\bl^S_{u},\bl^\mu_{u}}
\xi^{\by^T_{u}}_{\bl^\mu_{u},\bl^T_{u}}
&=
(\xi^{x_u}_{r_u,v}\otimes \dots \otimes \xi^{x_u}_{r_u,v})
(\xi^{y_u}_{v,s_u}\otimes \dots \otimes \xi^{y_u}_{v,s_u})
\\
&=\xi^{x_uy_u}_{r_u,s_u}\otimes \dots \otimes \xi^{x_uy_u}_{r_u,s_u}=\eta^{\bb_{u}}_{\bp_{u},\bq_{u}}.
\end{align*}
If $\bar x_u$ and $\bar y_u$ are not both $\0$, then $\bl^\mu_{u}=(v,v+1,\dots,v+m-1)$ for some $v$, and in this case, using (\ref{EXiDef}) and (\ref{E080717}), we get 
\begin{align*}
\xi^{\bx^S_{u}}_{\bl^S_{u},\bl^\mu_{u}}
\xi^{\by^T_{u}}_{\bl^\mu_{u},\bl^T_{u}}
&=
(\xi^{x_u}_{r_u,v} * \dots * \xi^{x_u}_{r_u,v+m-1})
(\xi^{y_u}_{v,s_u} * \dots * \xi^{y_u}_{v+m-1,s_u})
\\
&=m!\,\xi^{x_uy_u}_{r_u,s_u}\otimes \dots \otimes \xi^{x_uy_u}_{r_u,s_u}=\eta^{\bb_{u}}_{\bp_{u},\bq_{u}}.
\end{align*}
So, using (\ref{E140218}), in all cases we have
$$
\B^\mu_{S,T}=\pm \eta^{\bb_{1}}_{\bp_{1},\bq_{1}}*\dots*
\eta^{\bb_{t}}_{\bp_{t},\bq_{t}}=\pm \eta^{\bb}_{\bp,\bq},
$$
where we have applied Lemma~\ref{LSingleSep} for the last equality, using the fact that $(\bb,\bp,\bq)$ is $\de$-separated. 
\end{proof}

\subsection{Straightening} 
We continue with the set-up of the previous subsection. In particular, $i\in I$ is fixed, $\mu:=\la^{(i)}$ and $c:=|\mu|$. 
Recall 
the lexicographical order `$<$' on $\Alph_{X(i)}^c$ from \S\ref{SSLeading}. We have a similarly defined   lexicographical order on $\Alph_{Y(i)}^c$.

The following is an analogue of the main lemma in \cite{Woodcock}. In the lemma we denote by $\prec$ the lexicographic order on partitions. 


\begin{Theorem} \label{TWoodOne} 
Suppose that $n\geq c$ and $\mu\in\La_+(n,c)$. Let  
$S\in\Rst^{X(i)}(\mu)\setminus \Std^{X(i)}(\mu)$, $T\in\Rst^{Y(i)}(\mu)\setminus \Std^{Y(i)}(\mu)$. Then:
\begin{enumerate}
\item[{\rm (i)}] there exists $\la\in\La(n,c)$ with $\la_+\succ\mu$ such that 
$$
\xi^{\bx^S}_{\bl^S,\bl^\la}\xi^{e_i^c}_{\bl^\la,\bl^\mu}
=\pm \xi_{\bl^S,\bl^\mu}^{\bx^S}\,+\,\sum_{S'\in \Rst^{X(i)}(\mu),\,\, \L^{S'}<\L^S}c_{S'}\xi_{\bl^{S'},\bl^\mu}^{\bx^{S'}}
$$
for some $c_{S'}\in\k$. 
\item[{\rm (ii)}] there exists $\nu\in\La(n,c)$ with $\nu_+\succ\mu$ such that
$$
\xi^{e_i^c}_{\bl^\mu,\bl^\nu}\xi^{\by^T}_{\bl^\nu,\bl^T}
=\pm \xi_{\bl^\mu,\bl^T}^{\by^T}\,+\,\sum_{{T'}\in \Rst^{Y(i)}(\mu),\,\, \L^{T'}<\L^T}c_{T'} \xi_{\bl^\mu,\bl^{T'}}^{\by^{T'}}
$$
for some $c_{T'}\in\k$. 
\end{enumerate} 
\end{Theorem}
\begin{proof}
By left-right symmetry it suffices to prove (i). 
In this proof, for $L,L'\in\Alph_{X(i)}$, we write $L\to L'$ if $L>L'$ or $L=L'=r^x$ with $x\in X(i)_\0$. 
Since $S$ is not column standard, there exists some $(i,a,b)\in\Brackets\mu$ such that $(i,a+1,b)\in\Brackets\mu$ and  $S(i,a,b)\to S(i,a+1,b)$. Let $(a,b)\in\Z_{>0}\times\Z_{>0}$ be the lexicographically smallest pair with such property. 

For $1\leq t\leq n$, we consider the $t$th row of the Young diagram $\Brackets{\mu}$:
$$
\Brackets{\mu}_t:=\{(i,t,s)\mid s\leq \mu_t\}.
$$
For $t=1,\dots,n$, we define set partitions    
$
\Brackets{\mu}_t=E_t\sqcup F_t
$ as follows. For $t<a$, we set 
$$
E_t:=\Brackets{\mu}_t,\quad F_t:=\varnothing.
$$
For $t=a$, we set
$$
E_a:=\{(i,a,s)\in\Brackets{\mu}_a\mid s<b-1\},\quad F_a:=\Brackets{\mu}_a\setminus E_a.
$$
For $t>a$, we set 
$$
E_t:=\{M\in \Brackets{\mu}_t\mid S(N)\to S(M) \ \text{for all $N\in F_{t-1}$}\}, \quad F_t:=\Brackets{\mu}_t\setminus E_t.
$$
For all $1\leq t\leq n$, we let
$
e_t:=|E_t|$ and $f_t:=|F_t|. 
$

To define $\la\in\La(n,c)$, we consider two cases. 

\noindent
{\em Case 1: $b=1$.} In this case, we have $E_a=\varnothing$, and we set
$$
\la_t=
\left\{
\begin{array}{ll}
\mu_t &\hbox{if $t<a$,}\\
\mu_a+e_{a+1} &\hbox{if $t=a$,}\\
e_{t+1}+f_t &\hbox{if $a<t\leq n$,}
\end{array}
\right.
$$
where $e_{n+1}$ is interpreted as $0$. 

\noindent
{\em Case 2: $b>1$.} In this case we set
$$
\la_t=
\left\{
\begin{array}{ll}
\mu_t &\hbox{if $t<a$,}\\
e_{t}+f_{t-1} &\hbox{if $a\leq t\leq n$,}
\end{array}
\right.
$$
where $f_{0}$ is interpreted as $0$.

Note that $|\la|=|\mu|=c$. This is clear for the case $b=1$. If  $b>1$, then the assumptions $c\leq n$ and $\mu\in\La_+(n,c)$ imply $\mu_n=0$, so $f_n=0$, giving the claim. 

We next claim that $\la_+\succ \mu$. Indeed, we have $\la_t=\mu_t$ for $t<a$. If $b=1$ then $\la_a>\mu_a$. If $b>1$ then $\la_{a+1}>\mu_a$. So the claim follows in all cases. 


Given any word $\br$ of length $c$, we can write it as concatenation $\br=\br_1\dots\br_n$ such that the length of the word $\br_k$ is $\la_k$ for all $k=1,\dots,n$. 
For $1\leq k\leq n$, we have  $\bl^\la_k=k^{\la_k}$ and 
$$
\bl^\mu_k=
\left\{
\begin{array}{ll}
k^{\mu_k} &\hbox{if $k<a$,}
\\
(k-1)^{f_{k-1}} k^{e_k} &\hbox{if $k\geq a$ and $b>1$,}
\\
k^{f_{k}} (k+1)^{e_{k+1}} &\hbox{if $k\geq a$ and $b=1$.}
\end{array}
\right.
$$
So, by Lemma~\ref{LSep}, we have
$$
\xi^{\bx^S}_{\bl^S,\bl^\la}\xi^{e_i^c}_{\bl^\la,\bl^\mu}=
\pm (\xi^{\bx^S_1}_{\bl^S_1,1^{\la_1}}\xi^{e_i^{\la_1}}_{1^{\la_1},\bl^\mu_1})*\dots * (\xi^{\bx^S_n}_{\bl^S_n,n^{\la_n}}\xi^{e_i^{\la_n}}_{n^{\la_n},\bl^\mu_n}).
$$

Let $1\leq k\leq n$. By Lemma~\ref{LTwoParts}, 
$$
\xi^{\bx^S_k}_{\bl^S_k,k^{\la_k}}\xi^{e_i^{\la_k}}_{k^{\la_k},\bl^\mu_k}
=\xi^{\bx^S_k}_{\bl^S_k,\bl^\mu_k}+X_k,
$$
where $X_k$ is a linear combination of some $\xi_{\bu,\bl^\mu_k}^{\bx}$ 
with $\bu^{\bx}\sim (\bl^S_k)^{\bx^S_k}$ and 
$
\bu^{\bx}< (\bl^S_k)^{\bx^S_k}.
$ 

By the definition of $\la$, the triple $(\bx^S,\bl^S,\bl^\mu)\in\Seq^{B(i)}(n,c)$ is $\la$-separated. So by Lemma~\ref{LSingleSep}, we have 
$$
\xi^{\bx^S_1}_{\bl^S_1,\bl^\mu_1}*\cdots * \xi^{\bx^S_n}_{\bl^S_n,\bl^\mu_n}=\pm \xi_{\bl^S,\bl^\mu}^{\bx^S},
$$
which gives us the required leading term. On the other hand, if 
we have $\bu_k^{\bx_k}\leq (\bl^S_k)^{\bx^S_k}$ for all $k=1,\dots,n$, with at least one inequality being strict, then by Lemma~\ref{xistarproducts}, 
$$
\xi^{\bx_1}_{\bu_1,\bl^\mu_1}*\cdots * \xi^{\bx_n}_{\bu_n,\bl^\mu_n}=\pm c\, \xi_{\bu_1\cdots \bu_n,\bl^\mu}^{\bx_1\cdots \bx_n}
$$
for some $c\in \k$. Note that $\bu_1^{\bx_1}\cdots \bu_1^{\bx_1}<(\bl^S)^{\bx^S}$. Moreover, there is a tableau $S'\in \Rst^{X(i)}(\mu)$ such that $(\bx^{S'},\bl^{S'},\bl^\mu)\sim(\bx_1\cdots \bx_n, \bu_1\cdots\bu_n, \bl^\mu)$. Then  
$\xi_{\bu_1\cdots \bu_n,\bl^\mu}^{\bx_1\cdots \bx_n}=\pm c\  \xi_{\bl^{S'},\bl^\mu}^{\bx^{S'}}$
and
$$
\L^{S'}\leq \bu_1^{\bx_1}\cdots \bu_1^{\bx_1}<(\bl^S)^{\bx^S}=\L^{S},
$$ 
as required.
\end{proof}

\begin{Lemma} \label{L250217} 
For $\la\in\La(n,c)$ and $T\in\Rst^{Y(i)}(\mu)$, we have
$$\xi^{e_i^c}_{\bl^\la,\bl^\mu}\xi_{\bl^\mu,\bl^T}^{\by^T}=\sum_{T'\in\Rst^{Y(i)}(\la)}c_{T'}\xi_{\bl^\la,\bl^{T'}}^{\by^{T'}}$$ 
for some $c_{T'}\in\k$. 
\end{Lemma}
\begin{proof}
By Proposition~\ref{CPR}, we have that $\xi^{e_i^c}_{\bl^\la,\bl^\mu}\xi_{\bl^\mu,\bl^T}^{\by^T}$ is a linear combination of terms of the form  $\xi_{\bl^\la,\bl}^{\by}$ for some $\bl\in[1,n]^c$ and $\by\in Y(i)^c$ with $(\by,\bl^\la,\bl)\in\Seq^{B(i)}(n,c)$. Each of these terms  equals $\pm\xi_{\bl^\la,\bl^{T'}}^{\by^{T'}}$ for some $T'\in\Rst^{Y(i)}(\la)$.
\end{proof}

Let $\la,\la'\in\La(n,c)$ and $S\in\Rst^{X(i)}(\la)$, $S'\in\Rst^{X(i)}(\la')$, 
$T\in\Rst^{Y(i)}(\la)$, and $T'\in\Rst^{Y(i)}(\la')$. We write 
\begin{equation}\label{EPOOne}
(\la,S,T)\geq (\la',S',T')
\end{equation} 
if $\la\rhd\la'$, or $\la=\la'$, $\L^{S}\leq \L^{S'}$, $\L^{T}\leq \L^{T'}$.

\begin{Theorem} \label{TStrOne} 
Let $n\geq c$, $S\in\Rst^{X(i)}(\mu)$ and  $T\in\Rst^{Y(i)}(\mu)$. Then $\cod^\mu_{S,T}$ is a linear combination of standard codeterminants $\cod^\la_{S',T'}$ such that 
$(\la,S',T')\geq (\mu,S,T)$. 
\end{Theorem}
\begin{proof}
By Lemmas~\ref{LDomOne},~\ref{120117One} and \ref{LStdEquivOne}, we may assume that $\mu\in\La_+(n,c)$. If $S$ and $T$ are standard, we are done. Otherwise we may assume by symmetry that $S$ is not standard. 
Let $U$ be the set of all triples $(\la,S',T')$ such that $\la\in\La(n,c)$, $S'\in \Rst^{X(i)}(\la)$, $T'\in\Rst^{Y(i)}(\la)$, and either $\la_+\rhd\mu$ or $\la=\mu$, $T'=T$, $\L^{S'}< \L^S$. Using induction on the partial order (\ref{EPOOne}) and Lemma~\ref{LDomOne}, we see that it suffices to prove 
\begin{equation}\label{E160217}
\cod^\mu_{S,T}= \sum_{(\la,S',T')\in U}c_{\la,S',T'}\cod^\la_{S',T'}
\end{equation}
for some $c_{\la,S',T'}\in \k$.

By Theorem~\ref{TWoodOne}, there exists $\la\in\La^I(n,c)$ with $\la_+\succ\mu$ such that
$$
\xi_{\bl^S,\bl^\mu}^{\bx^S}
=\pm \xi^{\bx^S}_{\bl^S,\bl^\la}\xi^{e_i^c}_{\bl^\la,\bl^\mu}\,+\,\sum_{S'\in \Rst^{X(i)}(\mu),\, \L^{S'}<\L^S}c_{S'}\xi_{\bl^{S'},\bl^\mu}^{\bx^{S'}}.
$$
Multiplying on the right with $\xi_{\bl^\mu,\bl^T}^{\by^T}$ yields
$$
\cod^\mu_{S,T}
=\pm \xi^{\bx^S}_{\bl^S,\bl^\la}\xi^{e_i^c}_{\bl^\la,\bl^\mu}\xi_{\bl^\mu,\bl^T}^{\by^T}\,+\,\sum_{S'\in \Rst^{X(i)}(\mu),\, \L^{S'}<\L^S}c_{S'}\cod^\mu_{S',T}.
$$
It remains to note, using Lemma~\ref{L250217}, that we can write $\xi^{\bx^S}_{\bl^S,\bl^\la}\xi^{e_i^c}_{\bl^\la,\bl^\mu}\xi_{\bl^\mu,\bl^T}^{\by^T}$ as a linear combination of codeterminants of shape $\la$. 
\end{proof}

\subsection{Multicolored codeterminants}\label{SSNN}
Recall that in the beginning of the section, we have fixed $\bla\in\La^I(n,d)$. Let 
$$\|\bla\|=(d_0,\dots,d_\ell).$$ 
Recalling the notation of \S\ref{SSDFPOne}, for $\Stab=(S^{(0)},\dots,S^{(\ell)})\in \Tab^X(\bla)$ and $\T=(T^{(0)},\dots,T^{(\ell)})\in \Tab^Y(\bla)$, we define 
\begin{align*}
&\bx^\Stab=\bx^{S^{(0)}}\cdots \bx^{S^{(\ell)}}\in X^d,
&\bl^\Stab=\bl^{S^{(0)}}\cdots \bl^{S^{(\ell)}}\in[1,n]^d,
\\
&\by^\T=\by^{T^{(0)}}\cdots \by^{T^{(\ell)}}\in Y^d,
&\bl^\T=\bl^{T^{(0)}}\cdots \bl^{T^{(\ell)}}\in [1,n]^d,
\end{align*}
and
$$
\X_\Stab:=\xi^{\bx^\Stab}_{\bl^\Stab,\bl^\bla},\quad
\Y_\T:=\xi^{\by^\T}_{\bl^\bla,\bl^\T},\quad
\B^\bla_{\Stab,\T}:=\X_\Stab \Y_\T,
$$
where, in agreement with (\ref{E150288Bold}), we set 
$$\bl^\bla:=\bl^{\T^\bla}:=\bl^{\la^{(0)}}\cdots \bl^{\la^{(\ell)}}.
$$ 
We refer to the elements $\B^\bla_{\Stab,\T}$ as {\em codeterminants}. As for singled colored codeterminants, it is easy to see that $\X_\Stab, 
\Y_\T, 
\B^\bla_{\Stab,\T}\in T^A_\a(n,d)$. 

We refer to $\bla$ as the {\em shape} of the codeterminant $\B^\bla_{\Stab,\T}$. 
A codeterminant $\B^\bla_{\Stab,\T}$ is called {\em dominant} if $\bla\in\La_+^I(n,d)$, i.e. if its shape is a multipartition. A codeterminant $\B^\bla_{\Stab,\T}$ is called {\em standard} if it is dominant and $\Stab\in\Std^X(\bla)$ and $\T\in\Std^Y(\bla)$. 

The following lemma will allow us to use the theory of single-colored codeterminants developed in the previous subsections:

\begin{Lemma} \label{L290417_2New} 
We have 
$$\B^{\bla}_{\Stab,\T}=\pm\,\cod^{\la^{(0)}}_{S^{(0)},T^{(0)}}*\cdots * \cod^{\la^{(\ell)}}_{S^{(\ell)},T^{(\ell)}},$$ and the sign is $+$ if all $\bx^{\Stab}_k$ or all $\by^\T_k$ are even. 
\end{Lemma}
\begin{proof}
By Lemmas~\ref{LSingleSep} and \ref{LSep}, we have 
\begin{align*}
\B^{\bla}_{\Stab,\T}&=\X_\Stab \Y_\T=(\X_{S^{(0)}} *\cdots * \X_{S^{(\ell)}})(\Y_{T^{(0)}} *\cdots * \Y_{T^{(\ell)}})
\\
&=\pm(\X_{S^{(0)}}\Y_{T^{(0)}}) *\cdots * (\X_{S^{(\ell)}}\Y_{T^{(\ell)}})=\pm\,\cod^{\la^{(0)}}_{S^{(0)},T^{(0)}}*\cdots * \cod^{\la^{(\ell)}}_{S^{(\ell)},T^{(\ell)}}
\end{align*}
and the sign claim also follows from Lemma~\ref{LSep}. 
\end{proof}

For $\bb=b_1\cdots b_d\in B^d$, we write
\begin{equation}\label{E140218_1}
\|\bb\|=\mu
\end{equation}
if $\mu_i=\sharp\{k\in[1,d]\mid b_k\in B(i)\}$ for all $i\in I$. 


\begin{Proposition}\label{P260117}
Let $n \geq d$, and $(\bb,\bp,\bq)\in\Seq^B(n,d)$. Then $\eta_{\bp,\bq}^\bb=\pm \B ^{\bmu}_{\Stab,\T}$ for some 
$\bmu\in\La_+^I(n,d)$ with $\|\bmu\|=\|\bb\|$,  $\Stab\in\Rst^{X}(\bmu)$ and $\T\in\Rst^{Y}(\bmu)$. 
\end{Proposition}
\begin{proof}
This follows from Proposition~\ref{P260117One} and Lemmas~\ref{L290417_2New},\,\ref{LSingleSep}. 
\end{proof}

Let $\bla,\bmu\in\La(n,d)$, $\Stab\in\Rst^X(\bla)$, $\T\in\Rst^Y(\bla)$, $\Stab'\in\Rst^X(\mu)$ and $\T'\in\Rst^Y(\bmu)$. Recalling (\ref{EPOOne}), we write $(\bla,\Stab,\T)\geq (\bmu,\Stab',\T')$ 
if $(\la^{(i)},S^{(i)},T^{(i)})\geq (\mu^{(i)},(S')^{(i)},(T')^{(i)})$ for all $i\in I$.  

\begin{Theorem} \label{TStr} 
Let $n\geq d$, $\bla\in\La^I(n,d)$, $\Stab\in\Rst^X(\bla)$ and $\T\in\Rst^Y(\bla)$. Then $\B ^\bla_{\Stab,\T}$ is a linear combination of  $\B ^\bmu_{\Stab',\T'}$ such that $\bmu\in\La_+^I(n,d)$, $\Stab'\in\Std^X(\bmu)$, $\T'\in\Std^Y(\bmu)$ and 
$(\bmu,\Stab',\T')\geq (\bla,\Stab,\T)$. 
\end{Theorem}
\begin{proof}
This follows from Lemma~\ref{L290417_2New} and Theorem~\ref{TStrOne}.
\end{proof}

\begin{Corollary} \label{CStSpan} 
Let $n\geq d$. Then the standard codeterminants 
$$\{\cod^\bla_{\Stab,\T}\mid \bla\in\La^I_+(n,d),\ \Stab\in\Std^X(\bla),\ \T\in\Std^Y(\bla)\}$$ 
span $T^A_\a(n,d)$.
\end{Corollary}
\begin{proof}
The result follows from Proposition~\ref{P260117} and  Theorem~\ref{TStr}. 
\end{proof}

\begin{Theorem} \label{TStBasis} 
Let $n\geq d$. Then the standard codeterminants  $$\{\cod^\bla_{\Stab,\T}\mid \bla\in\La^I_+(n,d),\ \Stab\in\Std^X(\bla),\ \T\in\Std^Y(\bla)\}$$ 
form a $\k$-basis of $T^A_\a(n,d)$. 
\end{Theorem}
\begin{proof}
By Lemma \ref{TabBij}, there exists a bijection between the indexing set for the standard codeterminants and the indexing set for the standard basis of \(T^A_\a(n,d)\). Since the standard codeterminants span \(T^A_\a(n,d)\) by Corollary \ref{CStSpan}, the result follows since $\k$ is a domain.
\end{proof}

\section{Quasi-hereditary structure on $T^A_\a(n,d)$}\label{SecQHstruct}
We continue working with a fixed $d\in \Z_{\geq 0}$, $n\in\Z_{>0}$, and based quasi-hereditary graded \(\k\)-superalgebra $A$ with $\fa$-conforming heredity data $I,X,Y$. 
Recall the order $\leq$ on $\La^I(n,d)$ defined in (\ref{E140218_2}). 

\subsection{Heredity basis}
\label{SSMain}
Throughout the subsection we fix $\bla\in\La^I(n,d)$ with $\|\bla\|=(d_0,\dots,d_\ell)$. Recall the idmepotent $e_\bla$ defined in (\ref{E150218_1}). It is easy to see that 
\begin{equation}\label{E230617}
e_\bla=\X_{\T^\bla}=\Y_{\T^\bla}=\cod^\bla_{\T^\bla,\T^\bla}.
\end{equation}
Let $\Stab\in\Tab^X(\bla)$ and $\T\in \Tab^Y(\bla)$. Recalling (\ref{E150218_5}), define 
\begin{equation}\label{E210218_7}
\bal^\Stab:=\bal(\bx^\Stab,\bl^\Stab),\quad
\bbe^\T:=\bbe(\by^\T,\bl^\T). 
\end{equation}
The following results should be compared to Definition~\ref{DCC}(c). 

\begin{Lemma} \label{L130517} 
Let  $\bla\in\La_+^I(n,d)$, $\bmu\in\La^I(n,d)$, 
$\Stab\in\Std^X(\bla)$ and $\T\in \Std^Y(\bla)$. Then:
\begin{enumerate}
\item[{\rm (i)}] $\X_\Stab e_\bmu = \de_{\bla,\bmu}\X_{\Stab}$ and $e_\bmu\Y_{\T} =\de_{\bla,\bmu}\Y_{\T}$;
\item[{\rm (ii)}] $e_\bla \X_\Stab = \de_{\Stab,\T^\bla}\X_\Stab$ and  $ \Y_\T e_\bla = \de_{\T,\T^\bla}\Y_\T$.
\item[{\rm (iii)}] $e_\bmu \X_\Stab = \de_{\bmu,\bal^\Stab}\X_\Stab$ and  $ \Y_\T e_\bmu = \de_{\bmu,\bbe^\T}\Y_\T$.

\end{enumerate}
\end{Lemma}
\begin{proof} (i) and (iii) follow easily from Lemma~\ref{L150218_6}.

(ii) We prove the first equality, the second one being similar. By Lemma~\ref{L150218_6}, we have $e_\bla \X_\Stab = \de_{\bla,\bal^\Stab}\X_\Stab$, so it suffices to prove that $\bla=\bal^\Stab$ if and only if $\Stab=\T^\bla$. If $\Stab=\T^\bla$ then $\bal^\Stab=\bal^{\T^\bla}=\bla$. 

Conversely, if $\bla=\bal^\Stab$ then $\|\bla\|=\|\bal^\Stab\|$ and it follows using Definition~\ref{DCC}(c) that for all $i\in I$ every entry of $S^{(i)}$ is of the form $r^{e_i}$ for some $r\in [1,d]$. 
Fix $i\in I$. For $r=1,\dots, n$, let $\nu_r:=\sharp\{a\in[1,d_i]\mid \letter(S^{(i)}_a)=r\}$. Let $S$ be the $\la^{(i)}$-tableaux obtained from $S^{(i)}$ by replacing each entry $r^{e_i}$ with $r$.  Then $S$ is a classical standard $\la^{(i)}$-tableau of weight $\nu$. So $\nu\unlhd\la^{(i)}$, and $\nu=\la^{(i)}$ if and only if $S^{(i)}=T^{\la^{(i)}}$. Since $i$ is arbitrary, we have proved that $\Stab=\T^\bla$. 
\end{proof}

If $\bla\in\La_+^I(n,d)$, we denote 
\begin{align*}
T^A_\a(n,d)^{\geq \bla}&:=\spa\{\cod^\bmu_{\Stab,\T}\mid \bmu\in\La_+^I(n,d),\, \bmu\geq \bla,\, \Stab\in\Std^X(\bmu),\, \T\in\Std^Y(\bmu)\},
\\
T^A_\a(n,d)^{>\bla}&:=\spa\{\cod^\bmu_{\Stab,\T}\mid \bmu\in\La_+^I(n,d),\, \bmu> \bla,\, \Stab\in\Std^X(\bmu),\, \T\in\Std^Y(\bmu)\}.
\end{align*}

\begin{Proposition} \label{P190517_3} 
Let $n\geq d$ and $\bla\in\La_+^I(n,d)$. Then $T^A_\a(n,d)^{\geq \bla}$ is the two-sided ideal of $T^A_\a(n,d)$ generated by $\{e_\bmu\mid \bmu\in\La_+^I(n,d),\, \bmu \geq \bla\}$. 
\end{Proposition}
\begin{proof}
Let $J$ be the two-sided ideal of $T^A_\a(n,d)$ generated by $\{e_\bmu\mid \bmu\in\La_+^I(n,d),\, \bmu \geq \bla\}$. 
If $\bmu\in\La_+^I(n,d)$, $\bmu\geq \bla$,  $\Stab\in\Std^X(\bmu)$ and $\T\in\Std^Y(\bmu)$, then by Lemma~\ref{L130517}(i), we have $\cod^\bmu_{\Stab,\T}=\X_\Stab\Y_\T=\X_\Stab e_\bmu \Y_\T\in J$, so $T^A_\a(n,d)^{\geq \bla}\subseteq J$.  

We prove the converse inclusion by downward induction on $\leq$. Suppose the result has been proved for all $\bnu> \bla$, and let $\eta\in J$. By the inductive assumption, we may assume that $\eta=\eta_1 e_\bla\eta_2$ for some $\eta_1,\eta_2\in T^A_\a(n,d)$. By Lemma~\ref{L150218_6}, we may assume that $\eta_1$ is of the form $\eta^\bb_{\br,\bs}$ for $(\bb,\br,\bs)\in \Seq^B(n,d)$ with $\bbe(\bb,\bs)=\bla$, and $\eta_2$ is of the form $\eta^\bc_{\bt,\bu}$ for $(\bc,\bt,\bu)\in \Seq^B(n,d)$  with $\bal(\bc,\bt)=\bla$. 

Recalling the notation (\ref{E140218_1}), Definition~\ref{DCC} and Proposition~\ref{CPR}, we now deduce that either $\eta\in T^A_\a(n,d)^{>\|\bla\|}$ or $\|\bb\|=\|\bla\|=\|\bc\|$. 
In the latter case, using Definition~\ref{DCC}(c) and Proposition~\ref{CPR}, we see that $\eta_1 e_\bla\eta_2\neq0$ only if $\eta_1$ is of the form $\X_\Stab$ for some $\Stab\in\Tab^X(\bla)$ and $\eta_2$ is of the form $\Y_\T$ for some $\T\in\Tab^Y(\bla)$, i.e. we  may assume that $\eta=\cod^\bla_{\Stab,\T}$. 
By Theorem~\ref{TStr} and Lemma~\ref{L290417_2New}, $\cod^\bla_{\Stab,\T}$ is a linear combination of standard codeterminants $\cod^\bmu_{\Stab',\T'}$ with $\bmu\geq \bla$. Thus $\eta\in T^A_\a(n,d)^{\geq\bla}$. 
\end{proof}


\begin{Proposition}\label{P290517} 
Let $n\geq d$, $\bla\in\La^I_+(n,d)$, $\Stab\in\Std^X(\bla)$, $\T\in\Std^Y(\bla)$, and $\eta\in T^A_\a(n,d)$. Then
\begin{align*}
\eta\X_\Stab &\equiv\sum_{\Stab'\in\Std^X(\bla)}l^\Stab_{\Stab'}(\eta)\X_{\Stab'}\pmod{T^A_\a(n,d)^{>\bla}},
\\
\Y_\T\eta &\equiv\sum_{\T'\in\Std^Y(\bla)}r^\T_{\T'}(\eta)\Y_{\T'}\pmod{T^A_\a(n,d)^{>\bla}}
\end{align*}
for some $l^\Stab_{\Stab'}(\eta),\,r^\T_{\T'}(\eta)\in\k$.
\end{Proposition}
\begin{proof}
We prove the first equality, the second one being similar. 
By Proposition~\ref{P190517_3}, $\X_\Stab=\X_\Stab e_\bla$ belongs to the ideal $T^A_\a(n,d)^{\geq \bla}$. So we can write
$$
\eta\X_\Stab \equiv\sum_{\Stab'\in\Std^X(\bla),\ \T\in\Std^Y(\bla)}l^\Stab_{\Stab',\T}(\eta)\cod^\bla_{\Stab',\T}\pmod{T^A_\a(n,d)^{>\bla}}
$$
for some $l^\Stab_{\Stab',\T}(\eta)\in\k$. Multiplying on the right by $e_\bla$ and using Lemma~\ref{L130517}(ii), we see that $l^\Stab_{\Stab',\T}(\eta)=0$ unless $\T=\T^\bla$, in which case $\cod^\bla_{\Stab',\T}=\X_{\Stab'}$.
\end{proof}



The partial order `$\leq$' on $\La^I(n,d)$ restricts to a partial order on the subset $\La_+^I(n,d)\subseteq \La^I(n,d)$. For each $\bla\in\La^I_+(n,d)$, set 
\begin{align*}
\X(\bla)=\{\X_\Stab\mid \Stab\in\Std^X(\bla)\},\quad
\Y(\bla)=\{\Y_\T\mid \T\in\Std^Y(\bla)\}
\end{align*}
Define 
$$
\X:=\bigsqcup_{\bla\in\La^I_+(n,d)}\X(\bla),\quad \Y:=\bigsqcup_{\bla\in\La^I_+(n,d)}\Y(\bla).
$$

\begin{Theorem} \label{T290517_2} 
Let $n\geq d$ and $A$ be a based quasi-hereditary $\k$-superalgebra with $\a$-conforming heredity data $I,X,Y$. 
Then $T^A_\a(n,d)$ is a  based quasi-hereditary $\k$-superalgebra 
with heredity data $\La_+^I(n,d), \X, \Y$ and initial elements $e_\bla=\X_{\T^\bla}=\Y_{\T^\bla}$ for all $\bla\in \La^I_+(n,d)$.
\end{Theorem}
\begin{proof}
The property (a) of Definition~\ref{DCC} follows from Theorem~\ref{TStBasis}, the property (b) of Definition~\ref{DCC} follows from Proposition~\ref{P290517} and the property (c) of Definition~\ref{DCC} follows from Lemma~\ref{L130517}. 
\end{proof}

\begin{Remark} 
{\rm 
Let $(Z,\z)$ be as in \S\ref{SSZ}. It is easy to check explicitly that $T^Z_{\z}(1,2)$ is not quasihereditary when $\ell=1$. This shows that the assumption $n\geq d$ in Theorem~\ref{T290517_2} is necessary. In its absence, we can only sometimes  guarantee cellularity, see Lemma~\ref{LCellular}. 
}
\end{Remark}

\begin{Remark} 
{\rm 
In view of \cite[Remark 5.17]{greenTwo}, Theorem~\ref{T290517_2}  should be compared to the main result of \cite{GG}, which claims that the wreath product algebra $A\wr\Si_d$ is cellular of $A$ is {\em cyclic} cellular.  
}
\end{Remark}

\begin{Remark} \label{RConf} 
{\rm 
While Theorem~\ref{T290517_2} claims that $T^A_\a(n,d)$ is a based quasi-hereditary superalgebra with heredity data $\La_+^I(n,d), 
\X, 
\Y$, it does {\em not} claim that in general this heredity data is conforming. However, $\La_+^I(n,d), 
\X, 
\Y$ would be conforming under some natural additional assumptions on the heredity data $I,X,Y$ of $A$ which hold in most interesting examples, see \cite[\S4.4]{greenOne}. 
We consider such a situation in the following lemma.
}
\end{Remark}

\begin{Lemma} \label{Z2Z2}
Suppose that $A$ possesses a $(\Z/2\times \Z/2)$-grading 
$
A=\bigoplus_{\eps,\de\in\Z/2} A_{\eps,\de}
$
and heredity data \(I,X,Y\) such that the following conditions hold:
\begin{enumerate}
\item[{\rm (1)}] $A_{\eps,\de}A_{\eps',\de'}\subseteq A_{\eps+\eps',\de+\de'}$ for all $\eps,\de,\eps',\de'\in\Z/2$;
\item[{\rm (2)}] For all $\eps\in\Z/2$, we have $A_\eps=\bigoplus_{\eps'+\eps''=\eps} A_{\eps',\eps''}$. 
\item[{\rm (3)}] $X_\eps\subseteq A_{\eps,\0}$ and $Y_\eps\subseteq A_{\0,\eps}$ for all $\eps\in\Z/2$. 
\end{enumerate}
Then we have that:
\begin{enumerate}
\item The heredity data $I,X,Y$ is $\fa$-conforming for $\fa=A_{\0,\0}$. 
\item The $(\Z/2\times \Z/2)$-grading on \(A\) induces a $(\Z/2\times \Z/2)$-grading on \(T^A_\a(n,d)\) which, with the heredity data \(\La^I_+(n,d),\X,\Y\), satisfies axioms \textup{(1)--(3)}.
\item The heredity data \(\La^I_+(n,d),\X,\Y\) is conforming.
\end{enumerate}
\end{Lemma}

\begin{proof}
Claim (i) is easy to see, and (iii) will likewise follow from (ii). For the proof of (ii), write the $\Z/2 \times \Z/2$-degree of a homogeneous element \(a \in A\) as \((a^{(1)},a^{(2)}) \in \Z/2 \times \Z/2\). The $(\Z/2\times \Z/2)$-grading on \(A\) induces such a grading on $T^A_\a(n,d)$, where
\begin{align}\label{Z2Z2T}
((\xi^{\bb}_{\br, \bs})^{(1)},(\xi^{\bb}_{\br, \bs})^{(2)} ) = (b_1^{(1)} + \cdots + b_d^{(1)},b_1^{(2)}+\cdots +b_d^{(2)}) \in \Z_2 \times \Z_2,
\end{align}
for all \((\bb, \br, \bs) \in \Seq^B(n,d)\). We have by condition (2) on \(A\) that
\begin{align*}
\bar{\xi}^{\bb}_{\br, \bs} =\bar{b}_1 + \cdots + \bar{b}_d = (b_1^{(1)}+b_2^{(2)}) + \cdots + (b_d^{(1)}+b_d^{(2)}) = (\xi^{\bb}_{\br, \bs})^{(1)} + (\xi^{\bb}_{\br, \bs})^{(2)},
\end{align*}
so the \(\Z/2 \times \Z/2\)-grading on \(T^A_\a(n,d)\) satisfies condition (2) as well.

Elements of \(\X\) are of the form \(\xi^{\bx}_{\br, \bs}\) for some \((\bx, \br, \bs) \in \Seq^X(n,d)\), so by (\ref{Z2Z2T}) and condition (3) on \(A\), we have that \((\X^{(1)},\X^{(2)}) = (\X^{(1)},\0)=(\overline{\X},\0)\). Thus \(\X_{\varepsilon} \subseteq T^A_\a(n,d)_{\varepsilon,\0}\). We similarly have \(\Y_{\varepsilon} \subseteq T^A_\a(n,d)_{\0,\varepsilon}\), so the \(\Z/2 \times \Z/2\)-grading on \(T^A_\a(n,d)\) satisfies condition (3) as well, which completes the proof of (ii).
\end{proof}

\begin{Remark} 
{\rm We sometimes refer to the process of passing from $A$ to  $T^A_\a(n,d)$ as {\em schurifying} $A$. If there is no problem with conformity, as discussed in Remark~\ref{RConf}, one can schurify iteratively. For example, starting with $\k$ this produces interesting quasi-hereditary algebras which could be considered as Schur algebra analogues of iterated wreath products of symmetric groups. 
}
\end{Remark}

We complete this subsection with a technical result needed for future reference. Given $\bmu=(\mu^{(0)},\dots,\mu^{(\ell)})\in\La^I(n,d)$ and recalling the notation $\mu_+$ from \S\ref{SSPar}, let $\bla_+\in\La^I_+(n,d)$ be defined from $\la^{(i)}:=\mu^{(i)}_+$ for all $i\in I$. Recall the notation (\ref{E150218_5}). 


\begin{Lemma} \label{L260617} 
Let $n\geq d$ and $(\bx,\br,\bs)\in\Seq^X(n,d)$. Then $\bbe(\bx,\bs)\in\La^I(n,d)$, and $\xi^\bx_{\br,\bs}\in T^A_\a(n,d)^{\geq \bbe(\bx,\bs)_+}$. 
\end{Lemma}
\begin{proof}
We denote $\bmu:=\bbe(\bx,\bs)$ and $\bla:=\bbe(\bx,\bs)_+$. 
We can write $\xi^\bx_{\br,\bs}=\pm \X_\Stab=\pm\cod^\bmu_{\Stab,\T^\bmu}$ for some $\Stab\in\Tab^X(\bmu)$. 
By Lemmas~\ref{L290417_2New} and \ref{LDomOne}, we now deduce that $\xi^\bx_{\br,\bs}=\pm\cod^\bla_{\Stab',\T'}$ for some $\Stab',\T'$. The result now follows from Theorem~\ref{TStr}. 
\end{proof}

\subsection{Standard modules over generalized Schur algebras}
Recall the coproduct $\nabla:T^A_\a(n)\to T^A_\a(n)\otimes T^A_\a(n)$ defined in \S\ref{SSCoproduct}. 
In view of coassociativity of $\nabla$, we also have a well-defined homomorphism $\nabla^m:T^A_\a(n)\to T^A_\a(n)^{\otimes m}$ for any $m\geq 2$, with $\nabla^2=\nabla$. 
Restricting $\nabla^{\ell+1}$ from $T^A_\a(n)$ to $T^A_\a(n,d)\subset T^A_\a(n)$ gives a map:
$$
\nabla^{\ell+1}:T^A_\a(n,d)\to \bigoplus_{(d_0,\dots,d_\ell)\in\La(I,d)}  T^A_\a(n,d_0)\otimes\dots\otimes T^A_\a(n,d_\ell).
$$
Let $\de:=(d_0,\dots,d_\ell)\in\La(I,d)$. The natural projections $T^A_\a(n)\to T^A_\a(n,d_i)$ for all $i\in I$ induce the natural projection  
$$ \pi_\de:T^A_\a(n)^{\otimes (\ell+1)}\to T^A_\a(n,d_0)\otimes\dots\otimes T^A_\a(n,d_\ell).$$
Then we have an algebra homomorphism 
$$\nabla_\de:=\pi_\de\circ \nabla^{\ell+1}:T^A_\a(n,d)\to T^A_\a(n,d_0)\otimes\dots\otimes T^A_\a(n,d_\ell).$$
If $V_i\in\mod{T^A_\a(n,d_i)}$ for all $i\in I$, we use $\nabla_\de$ to define a structure of $T^A_\a(n,d)$-modules on 
$$
\bigotimes_{i\in I}V_i=V_0\otimes\dots\otimes V_\ell.
$$

Let now $\bla=(\la^{(0)},\dots,\la^{(\ell)})\in\La^I_+(n,d)$ with 
$$\de:=\|\bla\|=(d_0,\dots,d_\ell)\in \La(I,d).$$ For $i\in I$, we let  
\begin{equation}\label{E100717}
\text{\boldmath$\lambda^{(i)}$}:=(0,\dots,0,\la^{(i)},0,\dots,0)\in\La^I(n,d_i). 
\end{equation}
Recalling (\ref{E160918}), we have $e_{\text{\boldmath$\lambda^{(i)}$}}=\xi(\la^{(i)},e_i)\in T^A_\a(n,d_i)$. 
Denote 
\begin{align*}
\bar T^A_\a(n,d_i)&:=T^A_\a(n,d_i)/T^A_\a(n,d_i)^{>\text{\boldmath$\lambda^{(i)}$}}\qquad(i\in I),
\\
T^A_\a(n,\de)&:=T^A_\a(n,d_0)\otimes\dots\otimes T^A_\a(n,d_\ell),
\\
\bar T^A_\a(n,\de)&:=\bar T^A_\a(n,d_0)\otimes\dots\otimes \bar T^A_\a(n,d_\ell),
\\
T^A_\a(n,\de)^{>\bla}&:=\sum_{i=0}^\ell T^A_\a(n,d_0)\otimes\dots \otimes T^A_\a(n,d_i)^{>\text{\boldmath$\lambda^{(i)}$}}\otimes \dots\otimes T^A_\a(n,d_\ell),
\end{align*}
so that $T^A_\a(n,\de)/T^A_\a(n,\de)^{>\bla}\cong \bar T^A_\a(n,\de)$. With this notation we have:

\begin{Lemma} \label{L230617_6} 
If 
$\Stab=(S^{(0)},\dots,S^{(\ell)})\in\Std^X(\bla)$ and $\T=(T^{(0)},\dots,T^{(\ell)})\in\Std^Y(\bla)$ then
\begin{align*}
\nabla_\de(\X_\Stab)&\equiv \X_{S^{(0)}}\otimes\dots\otimes \X_{S^{(\ell)}}\pmod{T^A_\a(n,\de)^{>\bla}},
\\
\nabla_\de(\Y_\T)&\equiv \Y_{T^{(0)}}\otimes\dots\otimes \Y_{T^{(\ell)}}\pmod{T^A_\a(n,\de)^{>\bla}}.
\end{align*}
In particular, 
$$\nabla_\de(e_\bla)\equiv e_{\text{\boldmath$\lambda^{(0)}$}}\otimes\dots\otimes e_{\text{\boldmath$\lambda^{(\ell)}$}}\pmod{T^A_\a(n,\de)^{>\bla}}.
$$ 
\end{Lemma}
\begin{proof}
We prove the result for $\X$, the proof for $\Y$ being similar. 
Recall the set $\Seq_0^X(n,d)$ from \ref{ESeq0}, which depends on the choice of a total order on $X\times[1,n]\times[1,n]$. Let $(x,r,s),(x',r',s')\in X\times[1,n]\times[1,n]$ with $x\in X(i),x'\in X(i')$. 
Recalling (\ref{E190218}), we choose the total order on $X\times[1,n]\times[1,n]$ such that $(x,r,s)<(x',r',s')$ if and only if one of the following holds:
(a) $i\succ i'$; (b) $i=i'$ and $s<s'$; (c) $i=i'$, $s=s'$ and $r^x<(r')^{x'}$ in our fixed total order on $\Alph^{X(i)}$, cf. \S\ref{SSColLet}. 


Since $\Stab\in\Std^X(\bla)$, we have that $(\bx^\Stab,\bl^\Stab,\bl^\bla)\in \Seq_0^X(n,d)$. 
Applying Lemma~\ref{coprodxi} to $\X_\Stab=\xi^{\bx^\Stab}_{\bl^\Stab\hspace{-.7mm},\,\bl^\bla}$, we get 
$$\nabla_\de(\X_\Stab)= 
\xi^{\bx^{S^{(0)}}}_{\bl^{S^{(0)}}\hspace{-1mm},\,\bl^{\la^{(0)}}}\otimes\dots\otimes
\xi^{\bx^{S^{(\ell)}}}_{\bl^{S^{(\ell)}}\hspace{-1mm},\,\bl^{\la^{(\ell)}}}+(*),
$$
where (*) is a linear combination of terms of the form 
$\xi^{\bx^0}_{\br^0,\bs^0}\otimes\dots\otimes
\xi^{\bx^\ell}_{\br^\ell,\bs^\ell}$ such that $(\bx^i,\br^i,\bs^i)\in\Seq^X(n,d_i)$ for all $i\in I$, and for at least one $i\in I$ we have that not all entries of $\bx^i$ belong to $X(i)$. By choosing the smallest such $i$ (with respect to $\prec$) for each  $\xi^{\bx^0}_{\br^0,\bs^0}\otimes\dots\otimes
\xi^{\bx^\ell}_{\br^\ell,\bs^\ell}$ and using Lemma~\ref{L260617}, we deduce that 
$$\xi^{\bx^0}_{\br^0,\bs^0}\otimes\dots\otimes
\xi^{\bx^\ell}_{\br^\ell,\bs^\ell}\in T^A_\a(n,d_0)\otimes\dots \otimes T^A_\a(n,d_i)^{>\text{\boldmath$\lambda^{(i)}$}}\otimes \dots\otimes T^A_\a(n,d_\ell).$$
It remains to note that $\xi^{\bx^{S^{(i)}}}_{\bl^{S^{(i)}}\hspace{-1mm},\,\bl^{\la^{(i)}}}=\X_{S^{(i)}}$ for all $i\in I$. 
\end{proof}

\begin{Lemma} \label{L230617_5} 
We have 
$\nabla_\de(T^A_\a(n,d)^{>\bla})\subseteq T^A_\a(n,\de)^{>\bla}.
$ 
\end{Lemma}
\begin{proof}
By Proposition~\ref{P190517_3}, $T^A_\a(n,d)^{>\bla}$ is the two-sided ideal of $T^A_\a(n,d)$ generated by all $e_\bmu$ with $\bmu>\bla$. So it suffices to prove that for all $\nabla_\de(e_\bmu)\in T^A_\a(n,\de)^{>\bla}$ for all $\bmu>\bla$. By the second statement of Lemma~\ref{L230617_6}, we have $\nabla_\de(e_\bmu)\equiv e_{\text{\boldmath$\mu^{(0)}$}}\otimes\dots\otimes  e_{\text{\boldmath$\mu^{(\ell)}$}}$ modulo $T^A_\a(n,\de)^{>\bmu}$ and hence modulo $T^A_\a(n,\de)^{>\bla}$. It remains to observe that $e_{\text{\boldmath$\mu^{(0)}$}}\otimes\dots\otimes  e_{\text{\boldmath$\mu^{(\ell)}$}}\in T^A_\a(n,\de)^{>\bla}$.
\end{proof}

Let $n\geq d$. In Theorem~\ref{T290517_2}, we have established that $T^A_\a(n,d)$ is a based quasi-hereditary $\k$-superalgebra. By the general theory of \S\ref{SSBQHA}, we have standard modules $\{\De(\bla)\mid \bla\in \La^I_+(n,d)\}$. Each  standard module $\De(\bla)$ has basis $\{v_\Stab\mid \Stab\in\Std^{X}(\bla)\}$ such that, denoting $v_\bla:=v_{\T^\bla}$, we have $\X_\Stab v_\bla=v_\Stab$. In particular $e_\bla v_\bla=v_\bla$. We also have a bilinear pairing 
$$
(\cdot,\cdot)_\bla:\De(\bla)\times\De^\op(\bla)\to \k.
$$ 
If $\k$ is a field, the quotient $L(\bla)$ of $\De(\bla)$ by the radical of $(\cdot,\cdot)_\bla$ is an irreducible $T^A_\a(n,d)$-module, and $\{L(\bla)\mid \bla\in \La^I_+(n,d)\}$ is a complete and irredundant set of irreducible $T^A_\a(n,d)$-modules. 

\begin{Theorem} \label{T080717_2} 
Let $\bla=(\la^{(0)},\dots,\la^{(\ell)})\in \La^I_+(n,d)$. Then:
\begin{enumerate}
\item[{\rm (i)}] $\De(\bla)\cong \bigotimes_{i\in I}\De(\text{\boldmath$\lambda^{(i)}$})$ and $\De^\op(\bla)\cong \bigotimes_{i\in I}\De^\op(\text{\boldmath$\lambda^{(i)}$})$.
\item[{\rm (ii)}] Under the isomorphisms of {\rm (i)}, the pairing $(\cdot,\cdot)_\bla$ corresponds to the tensor product of the pairings $(\cdot,\cdot)_{\text{\boldmath$\lambda^{(i)}$}}$ over $i\in I$.
\end{enumerate}

\end{Theorem}
\begin{proof}
(i) We prove (i) for $\De(\bla)$, the proof for $\De^\op(\bla)$ being similar. Denote $\bar T^A_\a(n,d):=T^A_\a(n,d)/T^A_\a(n,d)^{>\bla}$ and write $\bar \eta:=\eta+T^A_\a(n,d)^{>\bla}$ for $\eta\in T^A_\a(n,d)$. Then $v_\Stab=\bar \X_\Stab$. Moreover, for all $i\in I$, denote $\bar T^A_\a(n,d_i):=T^A_\a(n,d_i)/T^A_\a(n,d_i)^{>\text{\boldmath$\lambda^{(i)}$}}$ and $\bar \eta:=\eta+T^A_\a(n,d_i)^{>\text{\boldmath$\lambda^{(i)}$}}$ for $\eta\in T^A_\a(n,d_i)$. Then $v_{S^{(i)}}=\bar \X_{S^{(i)}}$. 

Recall from\S\ref{SSBQHA} that $\De(\bla)=\bar T^A_\a(n,d)\bar e_\bla$ and $\De(\text{\boldmath$\lambda^{(i)}$})=\bar T^A_\a(n,d_i)\bar e_{\text{\boldmath$\lambda^{(i)}$}}$ for all $i\in I$. By the second statement of Lemma~\ref{L230617_6}, we now have that 
$$
e_\bla(v_{\text{\boldmath$\lambda^{(0)}$}}\otimes\dots\otimes v_{\text{\boldmath$\lambda^{(\ell)}$}})
=(e_{\text{\boldmath$\lambda^{(0)}$}}v_{\text{\boldmath$\lambda^{(0)}$}})\otimes\dots\otimes (e_{\text{\boldmath$\lambda^{(\ell)}$}}v_{\text{\boldmath$\lambda^{(i)}$}})=v_{\text{\boldmath$\lambda^{(0)}$}}\otimes\dots\otimes v_{\text{\boldmath$\lambda^{(\ell)}$}}.
$$
So it follows from 
Lemma~\ref{L230617_5} that there is a $T^A_\a(n,d)$-module homomorphism $\phi:\De(\bla)\to  \De(\text{\boldmath$\lambda^{(0)}$})\otimes\dots\otimes  \De(\text{\boldmath$\lambda^{(\ell)}$})$ which maps $v_\bla$ onto $v_{\text{\boldmath$\lambda^{(0)}$}}\otimes\dots\otimes v_{\text{\boldmath$\lambda^{(\ell)}$}}$. 
Moreover, by the first statement of Lemma~\ref{L230617_6}, for $\Stab=(S^{(0)},\dots,S^{(\ell)})\in\Std^{X}(\bla)$, we have 
$$
\phi\big(\X_\Stab(v_{\text{\boldmath$\lambda^{(0)}$}}\otimes\dots\otimes v_{\text{\boldmath$\lambda^{(\ell)}$}})\big)
=(\X_{S^{(0)}}v_{\text{\boldmath$\lambda^{(0)}$}})\otimes\dots\otimes (\X_{S^{(\ell)}}v_{\text{\boldmath$\lambda^{(\ell)}$}})=v_{S^{(0)}}\otimes\dots\otimes v_{S^{(\ell)}},
$$
so $\phi$ is surjective. Now, $\phi$ is injective by dimensions. 

(ii) It suffices to prove that for any $\Stab=(S^{(0)},\dots,S^{(\ell)})\in\Std^{X}(\bla)$ and $\T=(T^{(0)},\dots,T^{(\ell)})\in\Std^{Y}(\bla)$, we have 
$$(v_\Stab,v_\T)_\bla=\prod_{i\in I}(v_{S^{(i)}},v_{T^{(i)}})_{\text{\boldmath$\lambda^{(i)}$}}.$$
By definition, 
$$
\Y_\T\X_\Stab\equiv (v_\Stab,v_\T)e_\bla\pmod{T^A_\a(n,d)^{>\bla}}.
$$
Applying the isomorphism $\phi$ and using Lemma~\ref{L230617_6}, we get 
$$
\Y_{T^{(0)}}\X_{S^{(0)}}\otimes \dots\otimes 
\Y_{T^{(\ell)}}\X_{S^{(\ell)}}\equiv (v_\Stab,v_\T)(e_{\text{\boldmath$\lambda^{(0)}$}}\otimes \dots\otimes e_{\text{\boldmath$\lambda^{(\ell)}$}})\pmod{T^A_\a(n,\de)^{>\bla}}
$$
But the left hand side is congruent to 
$$
\prod_{i\in I}(v_{S^{(i)}},v_{T^{(i)}})_{\text{\boldmath$\lambda^{(i)}$}}
(e_{\text{\boldmath$\lambda^{(0)}$}}\otimes \dots\otimes e_{\text{\boldmath$\lambda^{(\ell)}$}})\pmod{T^A_\a(n,\de)^{>\bla}},
$$
so indeed $(v_\Stab,v_\T)_\bla=\prod_{i\in I}(v_{S^{(i)}},v_{T^{(i)}})_{\text{\boldmath$\lambda^{(i)}$}}$.
\end{proof}

\subsection{Anti-involution}
Let $\tau$ be a homogeneous anti-involution on $A$.   
Then $\tau$ induces a homogeneous anti-involution 
$\tau_n:M_n(A)\to M_n(A),\ \xi^a_{r,s}\mapsto \xi^{\tau(a)}_{s,r},
$
which in turn induces an anti-involution
\begin{equation}\label{TauND}
\tau_{n,d}:S^A(n,d)\to S^A(n,d),\ \xi^{\ba}_{\br,\bs}\mapsto \xi^{\ba^\tau}_{\bs,\br},
\end{equation}
where for a tuple $\ba=a_1\cdots a_d$ of homogeneous elements, we have denoted
$
\ba^\tau:=\tau(a_1)\cdots \tau(a_d).
$
If $\tau(\a)=\a$, then $
\tau_{n,d}$ restricts to the involution of $T^A_\a(n,d)$. Moreover, if $\tau(B_\a)=B_\a$, 
$\tau(B_\c)=B_c$ and $\tau(B_\1)=B_\1$, then we have 
\begin{equation}\label{EAntiinv}
\tau_{n,d}:T^A_\a(n,d)\to T^A_\a(n,d),\ \eta^{\bb}_{\br,\bs}\mapsto \eta^{\bb^\tau}_{\bs,\br}.
\end{equation}

Now, suppose in addition that $\tau$ is a standard anti-involution on $A$ with $y(x)=\tau(x)$, see \S\ref{SQHAReg}. 
In \S\ref{SSColLet}, to define standard tableaux we have fixed arbitrary total orders on all $\Alph_{X(i)}$ and $\Alph_{Y(i)}$. Note that the map $r^x\mapsto r^{y(x)}$ induces a bijection between $\Alph_{X(i)}$ and $\Alph_{Y(i)}$. Let us choose the total order on 
$\Alph_{X(i)}$ and $\Alph_{Y(i)}$ so that this bijection is an isomorphism of totally ordered sets. 

Let $\bla\in\La^I(n,d)$. 
Given a tableau $\Stab\in\Tab^X(\bla)$ we define a  tableau $\Stab^\tau\in \Tab^Y(\bla)$ via $\Stab^\tau_k=r^{y(x)}$ if $\Stab_k=r^{x}$ for all $k=1,\dots,d$. Due to the choice of total orders made in the previous paragraph, we have that $\Stab\mapsto \Stab^\tau$ is 
a bijection between $\Tab^X(\bla)$ and $\Tab^Y(\bla)$, which restricts to 
a bijection between $\Std^X(\bla)$ and $\Std^Y(\bla)$. 

\begin{Proposition} \label{PTau} 
Let $n\geq d$ and $A$ be a based quasi-hereditary $\k$-superalgebra with standard anti-involution $\tau$. Then, 
considering $T^A_\a(n,d)$ as a based quasi-hereditary algebra as in Theorem~\ref{T290517_2}, the involution $\tau_{n,d}$ of $T^A_\a(n,d)$ is standard, with $\tau(\X_\Stab)=\Y_{\Stab^\tau}$ for all admissible $\Stab$.
\end{Proposition}
\begin{proof}
The first statement follows from the second one, which in turn follows using  (\ref{EAntiinv}).
\end{proof}

\subsection{Idempotent truncation}
Let $e\in \a$ be an idempotent. Set $\bar A:=eAe$ and $\bar \a:=e\a e$.  Recalling (\ref{E160918}), we can associate to $e$ the idempotent
$$
\xi^e:=\sum_{\la\in\La(n,d)}\xi(\la,e)\in T^A_\a(n,d).
$$

\begin{Lemma} \label{LTrunc} %
{\rm \cite[Lemma 5.12]{greenTwo}} 
We have: 
\begin{enumerate}
\item[{\rm (i)}] \(S^{\bar A}(n,d) = \xi^e{}S^A(n,d)\xi^e\).
\item[{\rm (ii)}] \(T^{\bar A}_{\bar\a}(n,d) = \xi^e{}T^A_\a(n,d)\xi^e\).
\end{enumerate}
\end{Lemma}

Suppose now that $e$ is adapted with respect to $I,X,Y$, with the corresponding $e$-truncation $\bar I, \bar X,\bar Y$, see  \S\ref{SQHAReg}. 
For a subset \(J \subset I\), we consider $\Lambda^{I}_+(n,d)$ as the subset  of $\Lambda^{J}_+(n,d)$ as follows:
\begin{align*}\Lambda^{J}_+(n,d) = \{ \bla \in \La^I_+(n,d) \mid \la^{(i)} = \varnothing \textup{ if } i \notin J\}.
\end{align*}
In particular, we have $\La^{\bar I}_+(n,d)\subseteq \La^I_+(n,d)$ . 

\begin{Lemma} \label{L210218_1} 
Let \(n \geq d\) and \(e \in \a\) be an idempotent. 
\begin{enumerate}
\item[{\rm (i)}] If $e$ is adapted with respect to the heredity data $I,X,Y$ of $A$ with $e$-truncation $\bar I,\bar X,\bar Y$, then  \(\xi^e\) is adapted with respect to the heredity data $\La^I_+(n,d),\X,\Y$ of $T^A_\a(n,d)$, with $\xi^e$-truncation $\La^{\bar I}_+(n,d),\bar\X,\bar \Y$, where 
\begin{align*}
\bar\X&
=\bigsqcup_{\bla\in \La^{\bar I}_+(n,d)}\{\X_\Stab\mid \Stab\in \Std^{\overline{X}}(\bla)\},
\\
\bar\Y&
=\bigsqcup_{\bla\in \La^{\bar I}_+(n,d)}\{\Y_\T\mid \T\in \Std^{\overline{Y}}(\bla)\}.
\end{align*}

\item[{\rm (ii)}] If $e$ is strongly adapted, then so is $\xi^e$.
\end{enumerate}
\end{Lemma}
\begin{proof}
(i) 
Let \(\bx  = x_1 \cdots x_d \in X^d\). For \(k \in [1,d]\), we have \(ex_k = \de_{\{x_k\in \overline{X}\}}x_k\), whence 
$
\xi^e \xi^{\bx}_{\br, \bs} = 
\xi_{\br, \bs}^{(ex_1) \cdots (ex_d)} = 
\de_{\{\bx\in \overline{X}^d\}} \xi^{\bx}_{\br, \bs}.
$
In particular, for $\Stab \in \Std^{X}(\bla)$, we have 
$\xi^e \X_\Stab = \de_{\{\Stab \in \Std^{\overline{X}}(\bla)\}} \X_\Stab
$. 
Similarly for \(\T \in \Std^Y(\bla)\), we have 
$
\Y_\T \xi^e = \de_{\{\T\in \Std^{\overline{Y}}(\bla)\}}\Y_\T$. Thus \(\xi^e\) is adapted with $\bar\X,\bar \Y$ as in the statement of the lemma. 

It remains to  show that $\overline{\La^I_+(n,d)}=\La^{\bar I}(n,d)$.  Note that $\overline{\La^I_+(n,d)}\subseteq \La^{\bar I}(n,d)$ since if $\bla\in \La^{I}(n,d)\setminus \La^{\bar I}(n,d)$ and $\Stab\in\Std^X(\bla)$, then $\bx^\Stab\not\in \bar X^d$. To prove the converse inclusion we just need to observe, using the fact that $n\geq d$, that for every $\bla\in \La^{\bar I}(n,d)$ there exist $\Stab\in\Std^{\bar X}(\bla)$ and $\T\in\Std^{\bar Y}(\bla)$. 

(ii) If $e$ is strongly adapted then for any $\bla\in\La^{\bar I}_+(n,d)$, we have $e e_\bla=e_\bla=e_\bla e$, which implies the claim.  
\end{proof}

In the following result we consider $T^{\bar A}_{\bar\a}(n,d)$ as the subalgebra of $\xi^eT^A_\a(n,d)\xi^e\subseteq T^A_\a(n,d)$ using Lemma~\ref{LTrunc}. Recall the anti-involution $\tau_{n,d}$ from Proposition~\ref{PTau}. 

\begin{Proposition}\label{SchCell}
Let $\tau$ be a standard anti-involution on $A$ and $e\in\a$ be a  adapted $\tau$-invariant idempotent. Then $T^{\bar A}_{\bar\a}(n,d)$ is a cellular algebra with cellular basis \begin{align*}
\{C^{\bla}_{\Stab, \T} \mid \bla \in \La^{\bar I}_+(n,d), \Stab,\T \in \Std^{\bar X}(\bla)\},
\end{align*}
where we have set $C^\bla_{\Stab, \T}:=\B^\bla_{\Stab,\T^\tau}$. 
\end{Proposition}
\begin{proof}
We use Lemma~\ref{L210218_1}(i), which shows that $\xi^e$ is an adapted idempotent with $\xi^e$-truncation $\La^{\bar I}_+(n,d),\bar\X,\bar \Y$. By Proposition~\ref{PTau}, $\tau_{n,d}$ is a standard anti-involution on $T^A_\a(n,d)$ with $\tau(\B^\bla_{\Stab,\T^\tau})=\B^\bla_{\T,\Stab^\tau}$. Since $\xi^e$ is obviously $\tau_{n,d}$-invariant, \cite[Lemma 4.4(ii)]{greenOne} implies the result.
\end{proof}

\begin{Remark} 
{\rm 
Let $\bar A$ be a cellular algebra with cellular basis $\bar B$ and a subalgebra $\bar \a\subseteq \bar A_\0$. The following question was raised in \cite[Remark 4.6]{greenOne}: is there a based quasi-hereditary algebra $A$ with heredity basis $B$, a standard anti-involution $\tau$ and $\tau$-invariant adapted idempotent $e$ such that $\bar A=eAe$, $\bar\a=e\a e$, and $\bar B$ is the $e$-truncation of $B$? If such $A$ exists, which at least happens in many examples, it follows from the previous proposition that $T^{\bar A}_{\bar \a}(n,d)$ is cellular. 
}
\end{Remark}

Note that in the following lemma we do not require that $n\geq d$. While it is not in general true that $T^A_\a(n,d)$ is quasi-hereditary, the lemma shows that at least it is cellular under some natural assumptions. 

\begin{Lemma} \label{LCellular} 
Suppose that $(A,\a)$ is a unital pair. If $A$ possesses a standard anti-involution, then the algebra $T^A_\a(n,d)$ is cellular.\end{Lemma}
\begin{proof}
We may assume that $n<d$, in which case, by \cite[Lemma 5.13(ii)]{greenTwo}, we can realize $T^A_\a(n,d)$ as the idempotent truncation $\xi^d_n T^A_\a(d,d) \xi^d_n$, where $\xi^d_n$ is a $\tau_{n,d}$-invariant idempotent. Now we apply  \cite[Lemma 4.4(ii)]{greenOne}. 
\end{proof}

If $e\in A$ is an adapted idempotent then by Lemma~\ref{L210218}, there is a subset $\bar I'\subseteq \bar I$ with $eL(i)\neq 0$ if and only if $i\in\bar I'$. 

\begin{Lemma}\label{adaptSchur}
Let \(n \geq d\). Suppose that \(\k\) is a field, \(A_{\bar 1} \subseteq J(A)\), and 
 \(e \in \a\) is an idempotent adapted with respect to the heredity data $I,X,Y$ of $A$. Then $\xi^eL(\bla)\neq 0$ if and only if $\bla \in \La^{\bar I'}(n,d)$. In particular, $\{\xi^eL(\bla) \mid \bla \in \La^{\bar I'}(n,d)\}$ is a complete and irredundant set of irreducible \( \xi^e{}T^A_\a(n,d)\xi^e\)-modules up to isomorphism.
\end{Lemma}

\begin{proof}
By Lemma~\ref{L210218_1}, $\xi^e$ is an adapted idempotent, which will be used repeatedly without further reference. 

If  $\bla\in\La^I_+(n,d)\setminus \La^{\bar I'}_+(n,d)$ then there exists \(i \in \bar I \backslash \bar I'\) such that \(\bla^{(i)} \neq \varnothing\). By Lemma~\ref{L210218},  the condition $i\not\in \bar I'$ implies that 
$ye x \in {A^{>i}}$ for all $x\in X(i)$ and $y \in Y(i)$. So \(\Y_\T \xi^e \X_\Stab \in T^{A}_\a(n,d)^{> \bla}\) for all \(\Stab \in \Std^X(\bla)\) and \(\T \in \Std^Y(\bla)\). Hence $\xi^eL(\bla)=0$ by  Lemma~\ref{L210218} again. 

In the other direction, assume that \(\bla \in \La^{\bar I'}(n,d)\) with $d_i:=|\la^{(i)}|$. 
By Lemmas~\ref{L210218} and \ref{idemactionNew}(iii), for every \(i \in \bar{I}'\) there exists \(x_i \in \bar X(i)\) and \(y _i\in \bar Y(i)\) such that \(y_ix_i \equiv \kappa_i e_i \pmod{A^{>i}}\), for some \(\kappa_i \neq 0 \in \k\). Then \(y_i L(i) \neq 0\), and since \(A_{\bar 1} \subseteq J(A)\), we have that \(x_i, y_i \in A_{\bar 0}\). Then there exists \(\Stab \in \Std^{\bar X}(\bla)\), \(\T \in \Std^{\bar Y}(\bla)\) such that
$\letter(S^{(i)}_k) =\letter(T^{(i)}_k)=\letter(T^{\la^{(i)}}_k)$, 
$\col(S^{(i)}_k) =x_i\), and \(\col(T^{(i)}_k) =y_i\), for all \(i \in \bar I'\) and \(k \in [1,d_i]\). Applying Theorem~\ref{T080717_2}(ii), we have that \(\Y_\T \X_\Stab  \equiv \kappa_0^{d_0}  \cdots \kappa_\ell^{d_\ell} e_\bla \not \equiv 0 \pmod{T^A_\a(n,d)^{>\bla}}\), so \(\xi^eL(\bla) \neq 0\) by Lemma~\ref{L210218}.
\end{proof}

\section{Decomposition numbers}\label{SecDecomp}
Let again $\k$ be a commutative integral domain of characteristic zero, and suppose that we are given a ring homomorphism $\k\to \F$, where $\F$ is a field of characteristic $p\geq 0$. An important example is when $({\mathbb L},\k,{\mathbb K})$ is a modular system and $\F={\mathbb K}$ or ${\mathbb L}$, i.e. $\k$ is a complete discrete valuation ring of characteristic $0$, and either $\F$ is the field of fractions of $\k$ or $\F$ is the residue class field of $\k$. 

Throughout the section we assume that $d\leq n$. 
Recall from Theorem~\ref{T290517_2} that, starting with a based quasi-hereditary $\k$-superalgebra $A$ with $\a$-conforming heredity data, we have constructed a based quasi-hereditary $\k$-superalgebra $T^A_\a(n,d)$. 
 
\subsection{Set-up}
Define 
$$A_\F:=A\otimes_\k \F\quad \text{and}\quad T^A_\a(n,d)_\F:=T^A_\a(n,d)\otimes_\k \F.$$
As $T^A_\a(n,d)$ is a based quasi-hereditary $\k$-superalgebra with heredity data $\La^I_+(n,d),\X,\Y$, it is immediate that $T^A_\a(n,d)_\F$ is a based quasi-hereditary $\F$-superalgebra with heredity data $\La^I_+(n,d),\X_\F,\Y_\F$, where $\X_\F=\{x\otimes 1\mid x\in\X\}, \Y_\F=\{y\otimes 1\mid x\in\X\}\subseteq T^A_\a(n,d)_\F$. Normally we will drop indices and for example write $\X$ for $\X_\F$, etc.

\begin{Remark} 
{\rm 
If $\F$ has characteristic $p=0$, then  $T^A_\a(n,d)_\F\cong T^{A_\F}_{\a_\F}(n,d)$. However, if $p>0$, then in general $T^{A_\F}_{\a_\F}(n,d)$ is a `wrong object'; in particular, it does not have to be quasi-hereditary, and we might have $\dim T^{A_\F}_{\a_\F}(n,d)<\dim T^A_\a(n,d)_\F$ due to the presence of factorials in (\ref{E080717}). 
}
\end{Remark}

Let $\bla\in\La^I_+(n,d)$. The standard objects $\De_\F(\bla)$ and $\De^\op_\F(\bla)$ over $T^A_\a(n,d)_\F$ are obtained by extending scalars from $\k$ to $\F$ from the standard objects $\De(\bla)$ and $\De^\op(\bla)$ over $T^A_\a(n,d)$, and the pairing $(\cdot,\cdot)_{\bla,\F}:\De(\bla)_\F\times\De^\op(\bla)_\F\to \F$ is obtained from the pairing $(\cdot,\cdot)_\bla$ also by extending scalars. So for the irreducible $T^A_\a(n,d)_\F$-modules $L(\bla)_\F:=\De(\bla)_\F/\rad (\cdot,\cdot)_{\bla,\F}$ we have from Theorem~\ref{T080717_2}:

\begin{Lemma} \label{LIrrTens} 
If $\bla=(\la^{(0)},\dots,\la^{(\ell)})\in \La^I_+(n,d)$, then $L(\bla)_\F\cong \bigotimes_{i\in I}L(\text{\boldmath$\lambda^{(i)}$})_\F$. 
\end{Lemma}

We would like to describe the decomposition numbers 
$$d_{\bla,\bmu}^\F:=[\De(\bla)_\F:L(\bmu)_\F]
$$ 
of the quasi-hereditary algebra $T^A_\a(n,d)_\F$ in terms of the decomposition numbers $d_{i,j}^\F$ of $A_\F$ and the decomposition numbers $d_{\la,\mu}^{\operatorname{cl},\F}$ of the classical Schur algebra $S_\F(n,d)$. 

From now on, until the end of the section, since $\F$ is fixed, we drop `$\F$' from all the indices and write $T^A_\a(n,d)=T^A_\a(n,d)_\F$, $\De(\bla)=\De(\bla)_\F$, $d_{\la,\mu}^{\operatorname{cl}}=d_{\la,\mu}^{\operatorname{cl},\F}$, etc.

\subsection{Classical characters}\label{SSClChar}
Let us write the operation $(\mu,\nu)\mapsto \mu+\nu$ on the monoid $\La(n)$ multiplicatively, and let $\Z\La(n)$ be the corresponding monoid algebra with coefficients in $\Z$. 
So the product ${\tt s}{\tt t}$ is defined for any ${\tt s},{\tt t}\in\Z\La(n)$. 
Moreover, if ${\tt s}\in\Z\La(n,d)$ and ${\tt t}\in\Z\La(n,e)$ then ${\tt s}{\tt t}\in\Z\La(n,d+e)$. We identify 
\begin{equation}\label{EIdentify}
\Z(\La(n)^{\times m})=(\Z\La(n))^{\otimes m}\qquad(m\in\Z_{>0}). 
\end{equation}
with $(\mu^1,\dots,\mu^m)$ on the left corresponding to $\mu^1\otimes\dots\otimes\mu^m$ on the right.

Let $\la\in\La_+(n,d)$ and $S$ be a classical standard $\la$-tableau as defined in \S\ref{SSColLet}. The weight of $T$ is $\omega^T\in\La(n,d)$ defined from
$$
\omega^T_r:=\sharp\{N\in[\la]\mid T(N)=r\}\qquad(1\leq r\leq n).
$$
For 
$\mu\in\La(n,d)$, the {\em Kostka number} $k_{\la,\mu}$ is then the number of the classical standard $\la$-tableaux of weight $\mu$. 

For $\la\in\La_+(n,d)$, let $\De^{\operatorname{cl}}(\la)$ and $L^{\operatorname{cl}}(\la)$ be the standard and the irreducible module with high weight $\la$ over the classical Schur algebra $S(n,d)$, see \cite{Green}. We have the classical decomposition numbers
\begin{equation}\label{EClDecN}
d_{\la,\mu}^{\operatorname{cl}}:=[\De^{\operatorname{cl}}(\la):L^{\operatorname{cl}}(\mu)]\qquad (\la,\mu\in\La_+(n)),
\end{equation}
which is interpreted as zero if $|\la|\neq|\mu|$. 
For $\bla,\bmu\in\La^I_+(n)$, we will also need the products of the classical decomposition numbers: 
\begin{equation}\label{EClDecNBold}
d_{\bla,\bmu}^{\,\operatorname{cl}}:=\prod_{i\in I}d^{\,\operatorname{cl}}_{\la^{(i)},\mu^{(i)}}.
\end{equation}

For the classical  formal characters we have 
\begin{align}\label{E100717_15}
{\tt s}_\la:=\ch \De^{\operatorname{cl}}(\la)&=\sum_{\mu\in \La(n,d)}k_{\la,\mu }\cdot \mu\in\Z\La(n,d),
\\
\label{E100717_2}
\bar{\tt s}_\la:=\ch L^{\operatorname{cl}}(\la)&=\sum_{\mu\in \La(n,d)}\bar k_{\la,\mu }\cdot \mu\in\Z\La(n,d),
\end{align}
where $k_{\la,\mu}$ are the Kostka numbers and the weight multiplicities $\bar k_{\la,\mu}:=\dim L(\la)_\mu$ are not known in general if $p>0$. 
Note that  
\begin{equation}\label{EMatrix}
{\tt s}_\la=\sum_{\mu\in\La_+(n)}d_{\la,\mu}^{\operatorname{cl}}
\bar{\tt s}_\mu. 
\end{equation}

For $\mu^1\in\La_+(n,d_1),\dots,\mu^m\in\La_+(n,d_m)$, we can write
\begin{align}\label{E100717_6}
{\tt s}_{\mu^1}\cdots {\tt s}_{\mu^m}&=\sum_{\la\in \La_+(n)}c^{\,\la}_{\mu^1,\dots,\mu^m} {\tt s}_\la,
\end{align}
where $c^{\,\la}_{\mu^1,\dots,\mu^m}$ are the classical {\em Littlewood-Richardson coefficients}. 
We have $c^{\,\la}_{\mu^1,\dots,\mu^m}=0$ unless $d=d_1+\dots+d_m$. 
Denoting by $\umu=(\mu^k)_{k\in[1,m]}$ the unordered tuple of the partitions $\mu^1,\dots, \mu^m$, we define  
\begin{equation}\label{E170717_2}
c^{\la}_\umu:=c^{\,\la}_{\mu_1,\dots,\mu_m},
\end{equation}
which makes sense by the commutativity of tensor product.
The following follows easily from the definitions:

\begin{Lemma} \label{LLRTrans} 
Let $\la\in\La_+(n)$, and $u_1,\dots,u_m\in\Z_{>0}$ for some $m\geq 1$. Suppose that we are given a tuple $\umu=(\mu^{r}_s\mid 1\leq r\leq m,\ 1\leq s\leq u_r)$ of partitions in $\La_+(n)$. 
Then
$$
c^\la_\umu=\sum_{\unu=(\nu^1,\dots,\nu^m)\in\La_+(n)^m}c^\la_{\unu}\,\prod_{r=1}^m
c^{\,\nu^r}_{\mu^r_1,\dots,\mu^r_{u_r}}.
$$
\end{Lemma}

Let $\mu\in\La_+(n,e)$ and $\la\in\La_+(n,e+d)$ for some $d,e\in\Z_{\geq 0}$ be such that $[\mu]\subseteq [\la]$. 
Denote 
$$[\la/\mu]:=[\la]\setminus[\mu].
$$
The {\em even}  (resp. {\em odd}) {\em standard $\la/\mu$-tableau} is a function $T:[\la/\mu]\to[1,n]$ such that  whenever $M< N$ are nodes in the same row of $[\la/\mu]$, then $T(M) \leq T(N)$ (resp. $T(M) < T(N)$), and whenever $M< N$ are nodes in the same column of $[\la/\mu]$, then $T(M) < T(N)$ (resp. $T(M) \leq T(N)$). Let $\operatorname{St}_\0(\la/\mu)$ (resp. $\operatorname{St}_\1(\la/\mu)$) denote the set of all even (resp. odd) standard $\la/\mu$-tableaux. 

Let $\eps\in\Z/2$ and $T\in \operatorname{St}_\eps(\la/\mu)$.  
The weight of $T$  is the composition  
$$
\omega^T=(\omega^T_1,\dots,\omega^T_n)\in \La(n,d)
$$
with
$$
\omega^T_r:=\sharp\{N\in [\la/\mu]\mid T(N)=r\}\qquad(r=1,\dots,n). 
$$
Define
$$
{\tt s}^\eps_{\la/\mu}:=\sum_{T\in\operatorname{St}_\eps(\la/\mu)} \om^T \in\Z\La(n,d).
$$
For $\nu\in\La_+(n,d)$ we denote 
$$
\nu^\eps:=
\left\{
\begin{array}{ll}
\nu &\hbox{if $\eps=\0$,}\\
\nu^{\operatorname{transpose}} &\hbox{if $\eps=\1$.}
\end{array}
\right.
$$
As $d\leq n$, we can interpret $\nu^\eps$ as an element of $\La_+(n,d)$ again. 

\begin{Lemma} \label{LMac} 
Let 
$\mu\in\La_+(n,e)$ and $\la\in\La_+(n,e+d)$ for some $d,e\in\Z_{\geq 0}$ be such that $[\mu]\subseteq [\la]$. Then
$$
{\tt s}^\eps_{\la/\mu}=\sum_{\nu\in \La_+(n,d)} c^{\la}_{\mu,\nu^\eps} {\tt s}_\nu.
$$
\end{Lemma}
\begin{proof}
Follows from \cite[(3.8),(5.2),(5.3),(5.6),(5.13)]{MacD}.
\end{proof}

We need the following generalization of the above. Let $\la\in \La_+(n,d)$ and $\ud=(d_1,\dots,d_m)\in\La(m,d)$ for some $m\geq 1$. We denote by $\La_+^{\la,\ud}$ the set of all tuples $\tilde\mu=(\mu^1,\dots,\mu^m)$ such that 
$$\mu^t\in \La_+(n,d_1+\dots+d_t)\ \text{for $t=1,\dots,m$,\quad  and\quad }  
[\mu^1]\subseteq \dots\subseteq [\mu^m]=[\la].
$$
We will usually write $\tilde\mu=\mu^m/\cdots/\mu^1$ instead of  $\tilde\mu=(\mu^1,\dots,\mu^m)$, and  interpret $[\mu^0]$ as $\emptyset$. Denote 
\begin{equation}\label{E130717}
\La_+^{\la,m}:=\bigsqcup_{\ud\in\La(m,d)}\La_+^{\la,\ud}.
\end{equation}

Let $\tilde\mu=\mu^m/\cdots/\mu^1\in\La_+^{\la,m}$ and $\ueps=(\eps_1,\dots,\eps_m)\in (\Z/2)^m$. 
A {\em standard $\tilde\mu$-tableau of parity $\ueps$} is a function $\tilde T:[\la]\to[1,n]$ such that $\tilde T|_{\Brackets{\mu^t/\mu^{t-1}}}\in \operatorname{St}_{\eps_t}(\mu^t/\mu^{t-1})$ for all $t=1,\dots,m$. Let $\operatorname{St}_{\ueps}(\tilde\mu)$ denote the set of all standard $\tilde\mu$-tableaux of parity $\ueps$. 
The weight of a tableau $\tilde T\in \operatorname{St}_{\ueps}(\tilde\mu)$ is 
\begin{equation}\label{EUOm}
\uom^{\tilde\mu,\tilde T}:=(\omega^{\tilde T|_{\Brackets{\mu^1/\mu^{0}}}},\dots,\omega^{\tilde T|_{\Brackets{\mu^m/\mu^{m-1}}}})\in \La(n)^{\times m}.
\end{equation}
Now, define
\begin{equation}\label{150717_10}
{\tt s}^\ueps_{\la,m}:=
\sum_{\tilde\mu\in \La_+^{\la,m},\, \tilde T\in \operatorname{St}_\ueps(\tilde\mu)} \uom^{\tilde\mu,\tilde T}
 \in\Z(\La(n)^{\times m}).
\end{equation}

If $\unu=(\nu^1,\dots,\nu^m)\in\La(n)^n$ with $|\nu^k|<n$ for all $k=1,\dots,m$, and $\ueps=(\eps_1,\dots,\eps_m)\in (\Z/2)^m$, we denote
$$
\unu^{\ueps}:=\big((\nu^1)^{\eps_1},\dots,(\nu^m)^{\eps_m}\big)\in\La(n)^m
$$
Recall (\ref{EIdentify}). Using Lemma~\ref{LMac}, we have:

\begin{Corollary} \label{CMac} 
Let $\la\in \La_+(n,d)$ and $\ueps=(\eps_1,\dots,\eps_m)\in (\Z/2)^m$. Then
$$
{\tt s}^\ueps_{\la,m}=\sum_{\unu=(\nu^1,\dots,\nu^m)\in \La_+(n)^m} c^{\la}_{\unu^{\ueps}}\, {\tt s}_{\nu^1}\otimes\dots\otimes {\tt s}_{\nu^m}.
$$
\end{Corollary}

\subsection{Characters}
Let $V\in\mod{A}$ and $W\in\mod{T^A_\a(n,d)}$. 
If $e_i V$  is 
free of finite rank as a graded $\k$-supermodule for all $i\in I$, we say that $V$ has {\em free weight spaces}. Similarly, if $e_\bla W$  is free of finite rank as a graded $\k$-supermodule for all $\bla\in\La^I(n,d)$, we say that $W$ has {\em free weight spaces}.  
Suppose that $V$ and $W$ have free weight spaces. 
Although in general $\bigoplus_{i\in I} e_i V\subsetneq V$ and $\bigoplus_{\bmu\in \La^I(n,d)} e_\bmu W\subsetneq W$, we define the {\em (bigraded) characters}: 
\begin{align*}
\ch^q_\pi  V&:=\sum_{i\in I}(\dim^q_\pi e_iV)\, i\in R I,
\\
\ch^q_\pi  W&:=\sum_{\bmu\in \La^I(n,d)}(\dim^q_\pi e_\bmu W)\, \bmu \in R \La^I(n,d).
\end{align*}

\begin{Lemma}\label{L070318_4} \cite[Lemma 5.10]{greenTwo}
Suppose that $W_1\in\mod{T^A_\a(n,d_1)}$ and $W_2\in \mod{T^A_\a(n,d_2)}$ have free weight spaces. We consider $W_1\otimes W_2$ as a $T^A_\a(n,d_1+d_2)$-supermodule via the coproduct $\nabla$. Then $W_1\otimes W_2$ has free weight spaces, and 
$
\ch_\pi(W_1 \otimes W_2) = \ch_\pi(W_1) \, \ch_\pi(W_2).
$
\end{Lemma}

As in (\ref{EIdentify}), we always identify 
\begin{equation}\label{EIdentifyR}
R\La^I(n)=(R\La(n))^{\otimes I},
\end{equation}
with $\bmu=(\mu^{(0)},\dots,\mu^{(\ell)})$ on the left corresponding to $\mu^{(0)}\otimes\dots\otimes\mu^{(\ell)}$ on the right.

Recall that $e_\bla\De(\bla)=\F v_\bla$ and $e_\bmu\De(\bla)\neq 0$ implies $\bmu\leq \bla$. Similar properties hold for $L(\bla)$. Therefore:

\begin{Lemma} \label{LIndep} 
The subsets  $\{\ch^q_\pi \De(\bla)\mid \bla\in \La^I_+(n,d)\}$ and $\{\ch^q_\pi L(\bla)\mid \bla\in \La^I_+(n,d)\}$ are $R$-linearly independent in $R \La^I(n,d)$.
\end{Lemma}


Let $A$ be basic, i.e. all irreducible $A$-modules $L(i)$ are $1$-dimensional. In this case we have $1_A=\sum_{i\in I}e_i$ 
and $1_{T^A_\a(n,d)}=\sum_{\bmu\in\La^I(n,d)}e_\bmu$. 
Moreover, for all $i\in I$,  we have 
\begin{equation}\label{E110717}
\ch^q_\pi  L(i)=i\quad\text{and}\quad\ch^q_\pi \De(i)=\sum_{j\leq i}d_{i,j}(q,\pi)\cdot j.
\end{equation}
For $i,j\in I$, we set 
\begin{equation}
\label{E170717_9}
{}_jX(i):=\{x\in X(i)\mid e_jx=x\}\quad\text{and}\quad 
{}_jX:=\bigsqcup_{i\in I}{}_jX(i).
\end{equation}
We have ${}_iX(i)=\{e_i\}$ and ${}_jX(i)=\emptyset$ unless $j\leq i$. Each ${}_jX(i)$ splits as unions 
\begin{equation}\label{E180717}
{}_jX(i)={}_jX(i)_\0\sqcup {}_jX(i)_\1=\bigsqcup_{\eps\in\Z/2,\,n\in\Z}{}_jX(i)_\eps^n,
\end{equation}
see (\ref{E210218_2}). 
In view of (\ref{E110717}), we have 
$$|{}_jX(i)_\eps|=d_{i,j}^{\,\eps}:=\sum_{n\in \Z}d_{i,j}^{n,\eps}\quad\text{and}\quad 
|{}_jX(i)_\eps^n|=d_{i,j}^{n,\eps}.$$ 

\begin{Lemma} \label{L140717_5} 
Let $A$ be basic and $i\in I$. Then $\rad\De(i) =\spa\{v_x\mid x\in X(i)\setminus\{e_i\}\}$. 
\end{Lemma}
\begin{proof}
As the heredity data is basic, it follows that the codimension of $\rad\De(i)$ in $\De(i)$ is $1$, which implies the lemma since $v_i\not\in \rad\De(i)$.  
\end{proof}

\subsection{Characters of standard and irreducible modules for  basic  $A$}\label{SSIrrChar}
Throughout the subsection, we assume that $A$ is basic.
We first study the characters of standard $T^A_\a(n,d)$-modules. 
Let $\bla\in\La^I_+(n,d).$ 
For $\Stab\in \Std^X(\bla)$, recall $\bal^\Stab\in \La^I(n,d)$ from (\ref{E210218_7}) so that $\bal^\Stab=(\al^{(0)},\dots,\al^{(\ell)})$ with
\begin{equation}\label{E150717_1}
\al^{(i)}_r:=\sharp\{k\in[1,d]\mid \letter(\Stab_k)=r\ \text{and}\ \col(\Stab_k)\in{}_iX\}.
\end{equation}
Recalling (\ref{EDeg}), the {\em bi-degree} of $\Stab$ is defined to be 
\begin{equation}\label{E150717_2}
\deg(\Stab):=\deg(x_1)\dots\deg(x_d)\in R. 
\end{equation}
Using Lemma~\ref{L130517}(iii), we deduce: for any $\bmu\in\La^I(n,d)$:
$$
\dim^q_\pi  e(\bmu)\De(\bla)=\sum_{\Stab\in \Std^X(\bla),\,\bal^\Stab=\bmu}\deg(\Stab).
$$
Therefore:

\begin{Lemma} \label{L130717_3} 
We have 
$$
\ch^q_\pi \De(\bla)=\sum_{\Stab\in \Std^X(\bla)}\deg(\Stab)\cdot\bal^\Stab.
$$
\end{Lemma}

Now we turn to characters of irreducible $T^A_\a(n,d)$-modules. 
Let $\Std^X_0(\bla)\subseteq \Std^{X}(\bla)$ be the set of all standard $\bla$-tableaux $\Stab=(S^{(0)},\dots,S^{(\ell)})$ such that for all $i\in I$, the entries of $S^{(i)}$ are of the form $r^{e_i}$. Since $e_i$'s are  even, replacing every entry $r^{e_i}$ with $r$ yields a bijection between $\Std^X_{0}(\bla)$ and 
$\operatorname{St}_\0(\la^{(0)})\times\dots\times \operatorname{St}_\0(\la^{(\ell)})$. 
Recalling (\ref{EIdentifyR}) and (\ref{E100717_2}), we 
define
\begin{eqnarray}
\label{E210218_8}
{\tt s}_\bla&:=& {\tt s}_{\la^{(0)}}\otimes\dots\otimes  {\tt s}_{\la^{(\ell)}}\in R\La^I(n),
\\
\label{E210218_9}
\bar {\tt s}_\bla&:=&\bar {\tt s}_{\la^{(0)}}\otimes\dots\otimes \bar {\tt s}_{\la^{(\ell)}}\in R\La^I(n).
\end{eqnarray}

\begin{Lemma} \label{L210218_5} 
We have 
$\ch^q_{\pi}L(\bla)=\bar {\tt s}_\bla.$
\end{Lemma}
\begin{proof}
In view of Lemma~\ref{LIrrTens}, we may assume that 
$\bla=(0,\dots,0,\la^{(i)},0,\dots,0)$ with $\la^{(i)}\in\La(n,d)$ in the $i$th component for some $i\in I$. 
Recalling the elements $\X_\Stab=\xi^{\bx^\Stab}_{\bl^\Stab,\bl^\bla}$ from \S\ref{SSNN}. 
By definition, the module $\De(\bla)$ is an $\F$-subspace of $T^A_\a(n,d)/T^A_\a(n,d)^{>\bla}$ with basis
$
\{\bar\xi^{\bx^\Stab}_{\bl^\Stab,\bl^\bla}\mid \Stab\in\Std^X(\bla)\},
$
where we write $\bar \xi:=\xi+T^A_\a(n,d)^{>\bla}$ for $\xi\in T^A_\a(n,d)$. 

If $\bx^\Stab=x_1\cdots x_d$ and $x_k\neq e_i$ for some $k\in[1,d]$, then $x_k\in\rad \De(i)$ by Lemma~\ref{L140717_5}, and it follows from the definition of the pairing  $(\cdot,\cdot)_\bla$ that $\bar \xi^{\bx^\Stab}_{\bl^\Stab,\bl^\bla}\in\rad\De(\bla)$. Let now $\bx^\Stab=e_i^d$, i.e. $\Stab\in\Std_0^X(\bla)$. Since 
$
(v_\Stab,w_\T)
$
is the coefficient of $e_\bla=\bar\xi^{\,e_i^d}_{\bl^\bla,\bl^\bla}$ in
$
\bar\Y_\T\bar\X_\Stab=\bar\xi^{\,\by^\T}_{\bl^\bla,\bl^\T}
\bar\xi^{\,\bx^\Stab}_{\bl^\Stab,\bl^\bla},
$
we have that $(v_\Stab,w_\T)=0$ unless $\by^T=e_i^d$. But then by Proposition~\ref{CPR}, using $e_i^2=e_i$, we obtain 
$$
\xi^{e_i^d}_{\bl^\bla,\bl^\T}
\xi^{e_i^d}_{\bl^\Stab,\bl^\bla}=\sum_{\br,\bs}c_{\br,\bs}\xi_{\br,\bs}^{e_i^d},
$$
where the coefficients $c_{\br,\bs}$ are determined from
$$
\xi_{\bl^\bla,\bl^\T}
\xi_{\bl^\Stab,\bl^\bla}=\sum_{\br,\bs}c_{\br,\bs}\xi_{\br,\bs}
$$
in the classical Schur algebra $S(n,d)$. The result follows. 
\end{proof}

\begin{Corollary} \label{C160717}
We have  
$
{\tt s}_\bla=\sum_{\bmu\in\La^I_+(n)}d_{\bla,\bmu}^{\,\operatorname{cl}}\ch^q_\pi  L(\bmu). 
$
\end{Corollary}
\begin{proof}
This follows immediately from Lemma~\ref{L210218_5} and (\ref{EMatrix}). 
\end{proof}

\subsection{One-colored standard characters}\label{SSOneColStdChar}
We continue with the set-up of \S\ref{SSIrrChar}, in particular we assume that $A$ is basic. 
Throughout the subsection, we fix $i\in I$ and $\bla\in\La^I(n,d)$ of the form $\bla=\text{\boldmath$\lambda^{(i)}$}$, see (\ref{E100717}). 
We identify the Young diagram of $\la^{(i)}$ and the Young diagram of $\text{\boldmath$\lambda^{(i)}$}$ via the map $(r,s)\mapsto (i,r,s)$ on the nodes. In the same way we also identify the sets $\Std^{X(i)}(\la^{(i)})$ and $\Std^X(\text{\boldmath$\lambda^{(i)}$})$ so that we have a function $\deg:\Std^{X(i)}(\la^{(i)})\to R$ defined in (\ref{E150717_2}). 

Recalling (\ref{E170717_9}) and (\ref{E180717}), observe that we have 
$$
X(i)={}_iX(i)_\0\sqcup{}_iX(i)_\1\sqcup {}_{i+1}X(i)_\0\sqcup {}_{i+1}X(i)_\1\sqcup\dots\sqcup {}_\ell X(i)_\0\sqcup {}_\ell X(i)_\1.
$$
List the elements of $X(i)$ 
\begin{equation}\label{E170717_10}
X(i)=\{x_{i,1}=e_i,x_{i,2},\dots,x_{i,t_i}\}
\end{equation}
so that the elements of ${}_iX(i)_\0$ precede the  elements of ${}_iX(i)_\1$ precede the elements of ${}_{i+1}X(i)_\0$, \dots, precede  the elements of ${}_\ell X(i)_\1$. 
Define the segments 
$$
{}_j\Om^{(i)}:=\{u\in [1,t_i]\mid x_{i,u}\in {}_jX(i)\},
$$
so that $[1,t_i]={}_i\Om^{(i)}\sqcup {}_{i+1}\Om^{(i)}\sqcup \dots\sqcup {}_\ell\Om^{(i)}$. 
We order $X(i)$ so that $x_{i,1}< x_{i,2}<\dots<x_{i,t_i}$. 
Now pick an order on $\Alph^{X(i)}$ with $r^{x}<s^{x'}$ if and only if $x<x'$ or $x=x'$ and $r<s$. 

Let $S\in\Std^{X(i)}(\la^{(i)})$. For $ t\in[1, t_i]$, let $\la_{S,\leq t}\in\La_+(n)$ be the partition such that 
$$[\la_{S,\leq t}]=\{N\in[\la^{(i)}]\mid  \col(S(N))=x_{i,u}\text{ for some $u\leq t$}\}.
$$ 
Recalling (\ref{E130717}), 
$$
\tilde\la_{S}:=\big(\la_{S,\leq t_i}/\la_{S,\leq t_i-1}/\cdots/\la_{S,\leq 1}\big)\in\La_+^{\la^{(i)},t_i}.
$$
Denote
\begin{equation}\label{E130717_2}
\ueps_i:=(\bar x_{i,1},\bar x_{i,2},\dots,\bar x_{i,t_i})\in(\Z/2)^{t_i}.
\end{equation}
Then the map 
$$
\letter\circ S:[\la^{(i)}]\to[1,n],\ N\mapsto \letter(S(N))
$$
is a standard $\tilde\la_{S}$-tableau of parity $\ueps_i$, see \S\ref{SSClChar}. The map 
$$
f:S\mapsto (\tilde\la_{S},\letter\circ S)
$$
is easily seen to be a bijection between $\Std^{X(i)}(\lambda^{(i)})$ and the set 
\begin{equation}\label{EP}
P:=\{(\tilde\la,\tilde T)\mid \tilde\la\in \La_+^{\la^{(i)},t_i},\ \tilde T\in\operatorname{St}_{\ueps_i}(\tilde\la)\}.
\end{equation}

For every $(\tilde\la=\la^{t_i}/\cdots/\la^1,\tilde T)\in P$, we define
\begin{align}
\deg_i(\tilde\la)&:=\prod_{u=1}^{t_i}\deg(x_{i,u})^{|\la^u|-|\la^{u-1}|}
\label{E150717_9}
\\
\label{E220218}
\bal^{(\tilde\la,\tilde T)}&:=(\al^{(0)},\dots,\al^{(\ell)})\in\La^I(n,d),
\end{align}
where
$$
\al^{(j)}_r:=\sharp\{N\in[\la^{(i)}]\mid \tilde T(N)=r\ \text{and}\ N\in [\la^u]\setminus [\la^{u-1}]\ \text{for}\ u\in {}_j\Om(i)\}.
$$
Note that $\al^{(0)}=\dots=\al^{(i-1)}=0$. It follows from the definitions that $\deg(S)=\deg_i(\tilde\la_{S})\ \text{and}\ 
\bal^S=\bal^{f(S)}$. Thus, we have:

\begin{Lemma} \label{L150717_6} 
The map
$$
f:
\Std^{X(i)}(\lambda^{(i)})\to P,\ 
S\mapsto (\tilde\la_{S},\letter\circ S)
$$
is a bijection such that $\deg(S)=\deg_i(\tilde\la_{S})\ \text{and}\ 
\bal^S=\bal^{f(S)}$. 
\end{Lemma}

For $\unu=(\nu^1,\dots,\nu^{t_i})\in\La(n)^{t_i}$, we define \begin{align}
\deg^{(i)}(\unu)&:=\prod_{u=1}^{t_i}\deg(x_{i,u})^{|\nu^u|}\in R,
\\
\label{E170717}
{}_{j}\underline{\nu}&:=(\nu^t)_{t\in{}_j\Om(i)}\in\La(n)^{|{}_j\Om(i)|} \qquad(j\in I),
\\
\label{E170717_1}
\bchi(\unu)&:=(\chi^{(0)},\dots,\chi^{(\ell)})\in\La^I(n),
\end{align}
where we have set 
$
\chi^{(j)}:=\sum_{t\in{}_j\Om(i)}\nu^t
$ 
for all $j\in I$. 
We extend $\bchi$ by linearity to a function  
$$\bchi: \Z(\La(n)^{t_i})=(\Z\La(n))^{\otimes t_i} \to \Z\La^I(n).$$ 

From the Littlewood-Richardson rule (\ref{E100717_6}), we get:


\begin{Lemma} \label{L160717} 
Let $\unu=(\nu^1,\dots,\nu^{t_i})\in\La_+(n)^{t_i}$. 
 Then
$$
\bchi({\tt s}_{\nu^1}\otimes\dots\otimes {\tt s}_{\nu^{t_i}})=\sum_{\bga=(\ga^{(0)},\dots,\ga^{(\ell)})\in\La^I_+(n)}\left(\prod_{j\in I}c^{\ga^{(j)}}_{\,{}_{j}\unu}\right)
\tts_{\bga}.
$$ 
\end{Lemma}

Let $(\tilde\la=\la^{t_i}/\dots/\la^{1},\tilde T)\in P$. If $\uom^{\tilde\la,\tilde T}=(\om^1,\dots,\om^{t_i})$ then for all $u\in[1,t_i]$ we have 
$|\la^u|-|\la^{u-1}|=|\om^u|.$ So, comparing with (\ref{E150717_9}), we deduce
\begin{equation}\label{E150717_8}
\deg_i(\tilde\la)=\deg^{(i)}(\uom^{\tilde\la,\tilde T}).
\end{equation}
Taking into account (\ref{E220218}), we also get:

\begin{Lemma} \label{LUomBom} 
For $(\tilde\la,\tilde T)\in P$, we have $\bal^{(\tilde\la,\tilde T)}=\bchi(\uom^{\tilde\la,\tilde T})$. 
\end{Lemma}

\begin{Proposition} \label{P170717} 
We have 
$$
\ch^q_\pi \De(\text{\boldmath$\lambda^{(i)}$})=
\sum_{\bga\in\La^I_+(n)}\, \sum_{\unu\in\La_+(n)^{t_i}}
c^{\la^{(i)}}_{\unu^{\ueps_i}}
\deg^{(i)}(\unu)
\left(\prod_{j\in I}c^{\ga^{(j)}}_{\,{}_{j}\unu}\right)
\tts_{\bga}.
$$
\end{Proposition}
\begin{proof}
We have 
\begin{align*}
\ch^q_\pi \De(\text{\boldmath$\lambda^{(i)}$})
=&\sum_{S\in\Std^{X(i)}(\lambda^{(i)})}
\deg(S)\,\bal^S
\\
=&\sum_{(\tilde\la,\tilde T)\in P}
\deg_i(\tilde\la)\,\bal^{(\tilde\la,\tilde T)}
\\
=&\sum_{\uom\in\La(n)^{t_i}} \left(\sum_{(\tilde\la,\tilde T)\in P,\,\uom^{\tilde\la,T}=\uom}
\deg^{(i)}(\uom)\right)\bchi(\uom)
\\
=&\sum_{\unu=(\nu^1,\dots,\nu^{t_i})\in\La_+(n)^{t_i}}c^{\la^{(i)}}_{\unu^{\ueps_i}}
\deg^{(i)}(\unu)\bchi(\tts_{\nu^1}\otimes\dots\otimes \tts_{\nu^t} )
\\
=&
\sum_{\unu\in\La_+(n)^{t_i}}c^{\la^{(i)}}_{\unu^{\ueps_i}}
\deg^{(i)}(\unu)
\sum_{\bga\in\La^I_+(n)}\left(\prod_{j\in I}c^{\ga^{(j)}}_{\,{}_{j}\unu}\right)
\tts_{\bga},
\end{align*}
where we have used Lemma~\ref{L130717_3} for the first equality, 
Lemma~\ref{L150717_6} for the second equality, Lemma~\ref{LUomBom} and (\ref{E150717_8}) for the third equality, 
(\ref{EP}), (\ref{150717_10}) and Corollary~\ref{CMac} for the fourth equality, and Lemma~\ref{L160717} for the fifth equality.
\end{proof}

\subsection{Standard characters and decomposition numbers}
We continue with the assumption that $A$ is basic and use the notation of \S\ref{SSOneColStdChar}. 
In addition, we denote 
$$
\La^{I\times I}_+(n):=\{\underline{\bga}=(\bga_0,\dots,\bga_\ell)\mid \bga_0,\dots,\bga_\ell\in\La^I_+(n)\}.
$$ 
Thus for $\underline{\bga}=(\bga_0,\dots,\bga_\ell)\in \La^{I\times I}_+(n)$, each $\bga_i$ looks like $\bga_i=(\ga_i^{(0)},\dots,\ga_i^{(\ell)})$ with $\ga_i^{(j)}\in\La_+(n)$. 

Given a multipartition $\bnu=(\nu_x)_{x\in X}\in \La^{X}_+(n)$ and recalling (\ref{E170717_10}), we associate to $\bnu$ the tuple
$$
\underline{\bnu}=(\unu_0,\dots,\unu_\ell)\in \La_+(n)^{t_0}\times \La_+(n)^{t_1}\times\cdots\times \La_+(n)^{t_\ell}
$$
of multipartitions $\unu_i:=(\nu_{x_{i,1}},\dots,\nu_{x_{i,t_i}})\in\La_+(n)^{t_i}$. 
We denote
\begin{equation}
\label{E170717_6}
\deg(\bnu):=\prod_{x\in X}\deg(x)^{|\nu_x|}=\prod_{i\in I}\deg^{(i)}(\unu_i).
\end{equation}

For $i,j\in I$, $\bla=(\la^{(0)},\dots, \la^{(\ell)})\in\La^I_+(n)$ and $\underline{\bga}=(\bga_0,\dots,\bga_\ell)\in\La^{I\times I}_+(n)$ we observe: 
\begin{equation}
\label{E170717_7}
c^{\la^{(i)}}_{\unu_i^{\ueps_i}}
=
c^{\la^{(i)}}_{(\nu_x^{|x|})_{x\in X(i)}}
\qquad\text{and}\qquad
c^{\ga^{(j)}_i}_{\,{}_j\unu_i}
=
c^{\ga^{(j)}_i}_{(\nu_x)_{x\in {}_jX(i)}}.
\end{equation}

\begin{Lemma} \label{L170717_14} 
For 
$\bnu\in \La^X_+(n)$, we have 
$$
\sum_{\underline{\bga}=(\bga_0,\dots,\bga_\ell)\in\La^{I\times I}_+(n)} \prod_{i\in I}\left(\left[\prod_{j\in I}c^{\ga^{(j)}_i}_{\,{}_j\unu_i}\right]\tts_{\bga_i}\right)=
\sum_{\bmu\in\La^I_+(n)}\left(\prod_{j\in I}c^{\,\mu^{(j)}}_{(\nu_x)_{x\in {}_jX}}\right)\tts_\bmu.
$$
\end{Lemma}
\begin{proof}
In view of (\ref{E170717_7}) and  (\ref{E100717_6}), the left hand side equals 
\begin{align*}
&\sum_{\underline{\bga}\in\La^{I\times I}_+(n)} \left(\prod_{i,j\in I}c^{\ga^{(j)}_i}_{\,{}_j\unu_i}\right)\left(\prod_{i\in I}\tts_{\bga_i}\right)
\\
=&
\sum_{\underline{\bga}\in\La^{I\times I}_+(n)} \left(\prod_{i,j\in I}c^{\ga^{(j)}_i}_{(\nu_x)_{x\in {}_jX(i)}}\right)\left(\sum_{\bmu\in\La^I_+(n)} \left[\prod_{i\in I}c^{\mu^{(i)}}_{\ga_0^{(i)},\dots,\ga_\ell^{(i)}}\right] \tts_{\bmu}\right)
\\
=&
\sum_{\bmu\in\La^I_+(n)}\sum_{\underline{\bga}\in\La^{I\times I}_+(n)} \left(\prod_{i,j\in I}c^{\ga^{(j)}_i}_{(\nu_x)_{x\in {}_jX(i)}}\right)\left( \prod_{j\in I}c^{\mu^{(j)}}_{\ga_0^{(j)},\dots,\ga_\ell^{(j)}}\right) \tts_{\bmu},
\end{align*}
which equals the right hand side thanks to Lemma~\ref{LLRTrans}. 
\end{proof}

Now we can get a general formula for decomposition numbers:

\begin{Proposition} \label{C160717_3}
For $\bla\in\La^I_+(n,d)$, we have 
$$
\ch^q_\pi \De(\bla)=
\sum_{\bmu\in\La^I_+(n)}
\sum_{\bnu\in\La^{X}_+(n)}
\deg(\bnu)\left(\prod_{i\in I}
c^{\,\la^{(i)}}_{(\nu_x^{|x|})_{x\in X(i)}}c^{\,\mu^{(i)}}_{(\nu_x)_{x\in {}_iX}}\right)\tts_\bmu.
$$
\end{Proposition}
\begin{proof}
The result follows from the following computation
\begin{align*}
\ch^q_\pi \De(\bla)
=&\ch^q_\pi \bigg(\bigotimes_{i\in I}\De(\text{\boldmath$\lambda^{(i)}$})\bigg)
\\
=&
\prod_{i\in I}\,\ch^q_\pi  \De(\text{\boldmath$\lambda^{(i)}$})
\\
=&
\prod_{i\in I}\left(\sum_{\bga_i\in\La^I_+(n)} \sum_{\unu_i\in\La_+(n)^{t_i}}
c^{\la^{(i)}}_{\unu_i^{\ueps_i}}
\deg^{(i)}(\unu_i)
\left[\prod_{j\in I}c^{\ga^{(j)}_i}_{\,{}_j\unu_i}\right]
\tts_{\bga_i}\right)
\\
=&
\sum_{\underline{\bga}\in\La^{I\times I}_+(n)}
\sum_{\bnu\in\La^{X}_+(n)}
\prod_{i\in I}\left(
c^{\la^{(i)}}_{\unu_i^{\ueps_i}}
\deg^{(i)}(\unu_i)
\left[\prod_{j\in I}c^{\ga^{(j)}_i}_{\,{}_j\unu_i}\right]
\tts_{\bga_i}\right)
\\
=&
\sum_{\bnu\in\La^{X}_+(n)}
\deg(\bnu)\left(\prod_{i\in I}
c^{\la^{(i)}}_{\unu_i^{\ueps_i}}\right)
\sum_{\underline{\bga}\in\La^{I\times I}_+(n)}
\prod_{i\in I}\left(
\left[\prod_{j\in I}c^{\ga^{(j)}_i}_{\,{}_j\unu_i}\right]
\tts_{\bga_i}\right)
\\
=&
\sum_{\bnu\in\La^{X}_+(n)}
\deg(\bnu)\left(\prod_{i\in I}
c^{\,\la^{(i)}}_{(\nu_x^{|x|})_{x\in X(i)}}\right)
\sum_{\bmu\in\La^I_+(n)}\left(\prod_{j\in I}c^{\,\mu^{(j)}}_{(\nu_x)_{x\in {}_jX}}\right)\tts_\bmu,
\end{align*}
where we have used Theorem~\ref{T080717_2} for the first equality, Lemma~\ref{L070318_4} for the second equality, 
Proposition~\ref{P170717} for the third equality, (\ref{E170717_6}) for the fifth equality, (\ref{E170717_7}) and Lemma~\ref{L170717_14} for the last equality.  
\end{proof}

\begin{Corollary} \label{C160717_23}
Let $\bla,\bmu\in\La^I_+(n,d)$. Then 
$$d_{\bla,\bmu}
=\sum_{\bga\in\La^I_+(n)}
\sum_{\bnu\in\La^{X}_+(n)}
d_{\bga,\bmu}^{\operatorname{cl}}
\deg(\bnu)\left(\prod_{i\in I}
c^{\la^{(i)}}_{(\nu_x^{|x|})_{x\in X(i)}}c^{\ga^{(i)}}_{(\nu_x)_{x\in {}_iX}}\right).
$$
\end{Corollary}
\begin{proof}
This follows from Proposition~\ref{C160717_3}, Corollary~\ref{C160717} and Lemma~\ref{LIndep}. 
\end{proof}


\subsection{Decomposition numbers for non-basic algebras}\label{nonbasicdecomp}
Recall from (\ref{E080318}) the decomposition numbers $d_{i,j}^{n,\eps}$ for the algebra $A$. 
Define
\begin{align*}
\La_+^D(n) = \prod_{i,j \in I,\, m \in \Z,\, \varepsilon \in \Z_2} \La_+(n)^{d^{m,\varepsilon}_{i,j}}.
\end{align*}
We consider elements \(\bnu \in \La_+^D(n)\) as multipartitions $(\nu^{(i,j,m,\varepsilon,t)})$ with components \(\nu^{(i,j,m,\varepsilon,t)}\in \La_+(n)\) indexed by \(i,j \in I\), \(m \in \Z\), \(\varepsilon \in \Z_2\), and \(t \in [1,d_{i,j}^{m, \varepsilon}]\). 

Let $\bnu\in\La_+^D(n)$. We define $\bar\bnu=(\bar\nu^{(i,j,m,\varepsilon,t)})\in\La_+^D(n)$ via $\bar\nu^{(i,j,m,\varepsilon,t)}:=(\nu^{(i,j,m,\varepsilon,t)})^\varepsilon$; i.e. the conjugate partition if \(\varepsilon = \overline 1\), or the unchanged partition if \(\varepsilon = \overline 0\).
For \(i \in I\), we define the multipartitions
\begin{align*}
\bnu_i&:= (\nu^{(j,i,m,\varepsilon,t)})_{j \in I, m \in \Z, \varepsilon \in \Z_2, t \in [1,d_{j,i}^{m,\varepsilon}]},
\\
{}_i\bar \bnu&:= (\bar \nu^{(i,j,m,\varepsilon,t)})_{j \in I, m \in \Z, \varepsilon \in \Z_2, t \in [1,d_{i,j}^{m,\varepsilon}]}.
\end{align*}
Finally, define
\begin{align*}
\deg(\bnu) = \prod_{i,j \in I,\, m \in \Z,\, \varepsilon \in \Z_2,\, t \in [1,d_{i,j}^{m, \varepsilon}]} (q^m \pi^\varepsilon)^{|\nu^{(i,j,m,\varepsilon,t)}|}.
\end{align*}
Then we have the following.

\begin{Theorem}\label{NonBasicDecomp}
Suppose that $(A,\a)$ is a unital pair and assume that \(A_{\overline 1} \subset J(A)\). Then for \(\bla, \bmu \in \La_+^I(n,d)\), the corresponding decomposition number is given by
\begin{align*}
d_{\bla, \bmu}= \sum_{\bga \in \La_+^I(n)} \sum_{\bnu \in \La_+^D(n)} d^{\,\textup{cl}}_{\bga,\bmu} \deg(\bnu)\left( \prod_{i \in I } c^{\la^{(i)}}_{{}_i\overline{\bnu}} c^{\ga^{(i)}}_{\bnu_i}\right).
\end{align*}
\end{Theorem}
\begin{proof}
By Theorem~\ref{MorBasic}, there exists an $\a$-conforming heredity data $I,X',Y'$ with the same ideals $A(\Om)$ and such that the new initial elements $\{e'_i \mid i \in I\}$ are primitive idempotents in \(\mathfrak{a}\). Note that $T^A_\a(n,d)$ depends only on $\a$, see Proposition~\ref{AaInd}. Moreover, it is easy to see that the standard modules defined using the new heredity data are the same as the ones defined using the original heredity data since the ideals $A(\Om)$ have not changed. We we may and will assume that the idempotents \(\{e_i \mid i \in I\}\) are primitive in \(\a\). 

We now set $f:=\sum_{i\in I} e_i$, $\bar A:=fAf$ and $\bar\a:=f\a f$. Then by Theorem~\ref{MorBasic}(i), $f$ is strongly adapted so that $\bar A$ is based quasihereditary with $\bar \a$-conforming heredity data $I,\bar X,\bar Y$ which is the $f$-truncation of $I,X,Y$. Moreover, by Theorem~\ref{MorBasic}(iii), $\bar A:=fAf$ and $\bar\a:=f\a f$ are basic. By Theorem~\ref{MorBasic}(iv), $\funF:\mod{A}\to\mod{\bar A}, \ V\mapsto fV$ is an equivalence of categories with $\funF(L_A(i))\cong L_{\bar A}(i)$ and $\funF(\De_A(i))\cong \De_{\bar A}(i)$. In particular we have the equality of decomposition numbers $[\De_A(i):q^m\pi^\eps L_A(j)]=[\De_{\bar A}(i):q^m\pi^\eps L_{\bar A}(j)]$ for all $i,j,m,\eps$. 

By Lemma~\ref{L210218_1}(ii), we have a strongly adapted idempotent $\xi^f\in T^A_\a(n,d)$, and taking into account  Lemma~\ref{LTrunc}(ii), we have that $\xi^f T^A_\a(n,d)\xi^f= T^{\bar A}_{\bar \a}(n,d)$ is a quasihereditary algebra with heredity data $\La^{I}_+(n,d),\bar \X,\bar\Y$ for $$\bar\X
=\bigsqcup_{\bla\in \La^{\bar I}_+(n,d)}\{\X_\Stab\mid \Stab\in \Std^{\overline{X}}(\bla)\}\ \text{and}\ \bar\Y
=\bigsqcup_{\bla\in \La^{\bar I}_+(n,d)}\{\Y_\T\mid \T\in \Std^{\overline{Y}}(\bla)\}.$$ 
It follows that the functor $\funG: \mod{T^A_\a(n,d)}\to \mod{T^{\bar A}_{\bar \a}(n,d)},\ W\mapsto \xi^f W$ is an equivalence of categories. It is clear that 
$$\funG(\De_{T^A_\a(n,d)}(\bla))=\De_{T^{\bar A}_{\bar \a}(n,d)}(\bla)\ \text{and}\ 
\funG(L_{T^A_\a(n,d)}(\bla))=L_{T^{\bar A}_{\bar \a}(n,d)}(\bla).
$$
for all $\bla\in\La^I_+(n,d)$, and so we have the equality of decomposition numbers
\begin{equation}\label{E080318_3}
[\De_{T^A_\a(n,d)}(\bla):q^m\pi^\eps L_{T^A_\a(n,d)}(\bmu)]
=[\De_{T^{\bar A}_{\bar \a}(n,d)}(\bla):q^m\pi^\eps L_{T^{\bar A}_{\bar \a}(n,d)}(\bmu)]
\end{equation}
for all $\bla,\bmu,m,\eps$. 

From Corollary~\ref{C160717_23}, we can compute the decomposition numbers in the right hand of (\ref{E080318_3}) since $\bar A$ is basic. Moreover, it is easy to see that for basic algebras, the formula claimed in the theorem is equivalent to the formula of Corollary~\ref{C160717_23}. Now  the theorem follows from (\ref{E080318_3}). 
\end{proof}

\subsection{Blocks of Schurifications}\label{SchurBlocks}
The following result, proved in \cite[Lemma 4.8]{greenTwo}, shows that any algebra decomposition of \(A\) yields an associated decomposition 
 of $T^A_\a(n,d)$.
 
\begin{Lemma}
\label{BlockDecomp}
Let \(m \in \ZZ_{>0}\). For \(t \in [1,m]\) assume that \((A_t, \fa_t)\) is a good pair. Write \(A:= \bigoplus_{t=1}^m A_t\) and \(\fa := \bigoplus_{t=1}^m \fa_t\). Then we have
\begin{align*}
S^A(n,d) \cong \bigoplus_{\nu \in \La(m,d)} \bigotimes_{t=1}^m S^{A_t}(n,\nu_t)
\qquad
\textup{and}
\qquad
T^A_\fa(n,d) \cong \bigoplus_{\nu \in \La(m,d)} \bigotimes_{t=1}^m T^{A_t}_{\fa_t}(n,\nu_t)
\end{align*}
as \(\k\)-superalgebras.
\end{Lemma}

We now examine conditions under which \(T^{A}_{\fa}(n,d)\) is known to be indecomposable, thus showing that Lemma~\ref{BlockDecomp} describes a {\em block} decomposition of \(T^{A}_{\fa}(n,d)\) in terms of the blocks of \(A\) in many important cases.

Throughout this subsection, let \(e\) be an adapted idempotent for \(A\) (\(e=1\) is allowed). Recall that \(\{\bar L(i) := eL(i) \mid i \in \bar I\}\) is a complete set of simple modules for \(\bar A:= eAe\). We also have that \(\bar P(i) := eP(i)\) is the projective cover of \(\bar L(i)\) for all \(i \in \bar I\). 

For \(i, j \in \bar I\), we will write \(i \sim j\) if there exists a sequence \(i=i_0, i_1, \ldots, i_m = j \in \bar I\) such that \(\bar P(i_{t-1})\) and \(\bar P(i_t)\) share a common composition factor, for every \(t \in [1,m]\). Then \(i \sim j\) if and only if \(\bar L(i)\) and \(\bar L(j)\) belong to the same block of \(\bar A\), and thus \(\bar A\) is indecomposable if and only if \(i \sim j\) for all \(i,j \in \bar I\).

For all \(i,j \in \bar I\), we have, using BGG reciprocity:
\begin{align*}
[\bar P(i): \bar L(j)] \neq 0 & \iff 
[P(i) : L(j)] \neq 0 \\
&\iff (P(i): \Delta(k))\cdot [\Delta(k):L(j)] \neq 0 \textup{ for some }k \in I\\
&\iff [\nabla(k): L(i)] \cdot [\Delta(k):L(j)] \neq 0 \textup{ for some }k \in I \\
&\iff [\Delta^\op(k): L^\op (i)] \cdot [\Delta(k):L(j)] \neq 0 \textup{ for some }k \in I \\
&\iff d^{\,\op}_{k,i} d_{k,j} \neq 0 \textup{ for some }k \in I.
\end{align*}
Then, for \(i,j \in \bar I\), we have that \(\bar P(i)\), \(\bar P(j)\) share a common composition factor if and only if there exist \(k,k' \in I\), \(l \in \bar I\) such that
\(
d_{k,i}^{\, \op} d_{k,l} d_{k',j}^{\, \op} d_{k',l} \neq 0.
\)

Then \(i \sim j\) if and only if there exist sequences
\begin{align*}
i=i_0, \ldots, i_m = j \in \bar I, \qquad l_1, \ldots, l_m \in \bar I
\qquad
k_1, \ldots, k_m \in I, \qquad k'_1, \ldots, k'_m \in I
\end{align*}
such that 
\(
d_{k_t,i_{t-1}}^{\, \op} d_{k_t,l_t} d_{k'_t,i_t}^{\, \op} d_{k'_t,l_t} \neq 0
\)
for all \(t \in [1,m]\).

\begin{Theorem}\label{indec}
Suppose that $(A,\a)$ is a unital pair and \(A_{\overline 1} \subset J(A)\). Moreover, suppose that \(e\) is an adapted idempotent such that \(\bar A = eAe\) is indecomposable, and \(|\bar I| >1\). Then \(T^{\bar A}_{\bar \fa}(n,d)\) is indecomposable.
\end{Theorem}
\begin{proof}
By the indecomposability assumption we have \(i \sim j\) for all \(i,j  \in \bar I\). We aim to show that \(\bla \sim \bmu\) for all \(\bla, \bmu \in \La^{\bar I}_+(n,d)\). We first prove a claim.

{\em Claim 1.} For \(\bla, \balpha \in \La^I_+(n)\), write \(\balpha \subseteq \bla\) to indicate that \(\alpha^{(i)}_r \leq \la^{(i)}_r\) for all \(i \in I\), \(r \in \Z_{>0}\). Let \(\balpha \in \La^I_+(n,d-1)\), \(\bla,\bmu \in \La^I_+(n,d)\), with \(\balpha \subseteq \bla,\bmu\). Moreover, assume that \(|\la^{(i)}| = |\alpha^{(i)}| +1\) and \(|\mu^{(j)}| = |\alpha^{(j)}|+1\), for some \(i \neq j \in I\) such that \(d_{ij} \neq 0\) (resp. \(d_{ij}^{\, \op}\)). Then \(d_{\bla,\bmu} \neq 0\) (resp. \( d^{\,\op}_{\bla,\bmu}  \neq 0\)).

{\em Proof of Claim 1.}
Assume that \(d_{ij} \neq 0\); the proof that \(d_{ij}^{\, \op} \neq 0\) implies \(d^{\,\op}_{\bla,\bmu} \neq 0\) is similar. We have that \(d_{ij}^{n, \varepsilon} \neq 0\) for some \(n \in \Z\), \(\varepsilon \in \Z_2\).
 Let \(\bnu\) be the element of \(\La^D_+(n)\) such that \(\nu^{(k,k,0,0,1)} = \al^{(k)}\) for all \(k \in I\), \(\nu^{(i,j,n, \varepsilon,1)}= (1)\), and all other components of \(\bnu\) are empty. Then, noting that \(d_{\bmu, \bmu}^{\,\textup{cl}} = 1\) and \(\prod_{k \in I} c^{\la^{(k)}}_{{}_i \bar{\bnu}} c^{\mu^{(k)}}_{\bnu_i} = 1\), Theorem \ref{NonBasicDecomp} gives us that \(d_{\bla, \bmu}\neq 0\), as desired.

Now, for \(\bla, \bmu \in \La^{\bar I}_+(n,d)\), define
\begin{align*}
\textup{diff}(\bla, \bmu):= \sum_{k \in I} \sum_{r >0} |\la^{(i)}_k - \mu^{(i)}_k|.
\end{align*}
We will prove that \(\bla \sim \bmu\) by induction on \(\textup{diff}(\bla, \bmu)\).

Assume \(\textup{diff}(\bla,\bmu)=2\) (the smallest non-trivial case). Then for some \(\balpha \in \La^{\bar I}_+(n,d-1)\) we have \(\balpha \subseteq \bla,\bmu\), with \(|\la^{(i)}| = |\alpha^{(i)}| +1\) and \(|\mu^{(j)}| = |\alpha^{(j)}|+1\), for some \(i, j \in \bar I\). Since \(i \sim j\) by assumption, there exist sequences
\begin{align*}
i=i_0, \ldots, i_m = j \in \bar I, \qquad l_1, \ldots, l_m \in \bar I
\qquad
k_1, \ldots, k_m \in I, \qquad k'_1, \ldots, k'_m \in I
\end{align*}
such that 
\(
d_{k_t,i_{t-1}}^{\, \op} d_{k_t,l_t} d_{k'_t,i_t}^{\, \op} d_{k'_t,l_t} \neq 0
\)
for all \(t \in [1,m]\). Since \(|\bar I| >1 \), we may moreover assume we have chosen sequences such that \(i_{t-1} \neq i_{t}\) for all \(t \in [1,m]\). 

For \(h \in I\), define \({}_h \bbeta \in \La^I_+(n,d)\) via
\begin{align*}
{}_h\beta^{(h')}_r :=
\begin{cases}
\alpha^{(h)}_1 +1 & \textup{if } h'=h, r=1\\
\alpha^{(h')}_r & \textup{otherwise}.
\end{cases}
\end{align*}
Define \({}^0\btau, \ldots, {}^m\btau \in \La^{\bar I}_+(n,d)\) by \({}^0 \btau := \bla\), \({}^m\btau:= \bmu\), and \({}^t \btau := {}_{i_t}\bbeta\) for \(t \in [1,m-1]\). For \(t \in [1,m]\), define \({}^t\bta \in \La^{\bar I}_+(n,d)\) via
\begin{align*}
{}^t\bta :=
\begin{cases}
{}^{t-1}\btau & \textup{if }l_t=i_{t-1}\\
{}^t\btau & \textup{if }l_t = i_t\\
{}_{l_t}\bbeta & \textup{otherwise}.
\end{cases}
\end{align*}
Furthermore, define \({}^t\bkappa, {}^t\bkappa' \in \La^I_+(n,d)\) via
\begin{align*}
{}^t\bkappa :=
\begin{cases}
{}^{t-1}\btau & \textup{if }k_t=i_{t-1}\\
{}^{t}\bta & \textup{if }k_t = l_t\\
{}_{k_t}\bbeta & \textup{otherwise}.
\end{cases}
\qquad
{}^t\bkappa' :=
\begin{cases}
{}^t\btau & \textup{if }k'_t=i_{t}\\
{}^t\bta & \textup{if }k'_t = l_t\\
{}_{k'_t}\bbeta & \textup{otherwise}.
\end{cases}
\end{align*}

Let \(t \in [1,m]\). By construction, we have that either \({}^t \bkappa = {}^{t-1} \bkappa\), or else \({}^t \bkappa, {}^{t-1} \btau\) satisfy the assumptions of Claim 1. Then, in either case we have \(d_{{}^t \bkappa, {}^{t-1} \btau}^{\, \op} \neq 0\). Similarly, we have that the pairs \(({}^t \bkappa, {}^t \bta)\), \(({}^t \bkappa', {}^t \btau)\), \(({}^t \bkappa', {}^t \bta)\) are either equal or satisfy the assumptions of Claim 1, so \(d_{{}^t \bkappa, {}^t \bta}\), \(d_{{}^t \bkappa', {}^t \btau}^{\, \op}\), \(d_{{}^t \bkappa', {}^t \bta} \neq 0\) as well. Therefore we have sequences
\[
\bla={}^0 \btau, \ldots, {}^m \btau= \bmu \in \La^{\bar I}_+(n,d),
\qquad
{}^0 \bta, \ldots, {}^m \bta \in \La^{\bar I}_+(n,d),
\]
\[
{}^0 \bkappa, \ldots, {}^m \bkappa \in \La^{ I}_+(n,d),
\qquad
{}^0 \bkappa', \ldots, {}^m \bkappa' \in \La^{ I}_+(n,d),
\]
such that 
\(
d_{{}^t \bkappa, {}^{t-1} \btau}^{\, \op} d_{{}^t \bkappa, {}^t \bta}d_{{}^t \bkappa', {}^t \btau}^{\, \op}d_{{}^t \bkappa', {}^t \bta} \neq 0
\)
for all \(t \in [1,m]\). This proves that \(\bla \sim \bmu\).

Now for the induction step, assume \(\textup{diff}(\bla,\bmu) = D>2\). There exists some \(i,j \in \bar I\), \(r,s \in \Z_{>0}\)  such that \(\la_r^{(i)} > \mu^{(i)}_r\) and \(\la_s^{(j)} < \mu_s^{(j)}\). Assume that \(r\) is maximal, and \(s\) is minimal such that these inequalities hold. Then define \(\brho \in \La_+^{\bar I}(n,d)\) via
\begin{align*}
\rho^{(k)}_t = \begin{cases}
\la^{(i)}_r-1 & \textup{if }t=r, k=i\\
\la^{(j)}_s+1 & \textup{if }t=s, k=j\\
\la^{(k)}_t & \textup{otherwise}.
\end{cases}
\end{align*}
The maximality/minimality assumptions guarantee that \(\brho\) is in fact a multipartition. We have then that \(\textup{diff}(\bla, \brho) = 2\) and \(\textup{diff}(\brho, \bmu) = D-2\). So by induction we have \(\bla \sim \brho \sim \bmu\), as desired.
\end{proof}

\begin{Remark}
The condition \(|\bar{I}|>1\) implies that \(\bar{A}\) is non-simple. If \(\bar{A}\) is simple, then \(T^{\bar{A}}_{\bar{\fa}}(n,d)\) is isomorphic to a classical Schur algebra, which is decomposable in general.
\end{Remark}

\subsection{Decomposition numbers for the zigzag algebra}
\label{SSZ}

Fix $\ell\geq 1$ and set 
$$
I:=\{0,1,\dots,\ell\},\quad J:=\{0,\dots,\ell-1\}.
$$ 
Let $\Gamma$ be the quiver with vertex set $I$ and arrows $\{a_{j,j+1},a_{j+1,j}\mid j\in J\}$ as in the picture: 
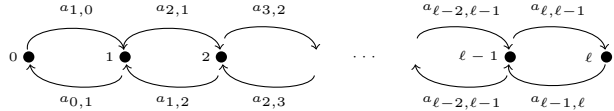
\begin{align*}
\begin{braid}\tikzset{baseline=3mm}
\coordinate (0) at (-4,0);
\coordinate (1) at (0,0);
\coordinate (2) at (4,0);
\coordinate (3) at (8,0);
\coordinate (6) at (12,0);
\coordinate (L1) at (16,0);
\coordinate (L) at (20,0);
\draw [thin, black,->,shorten <= 0.1cm, shorten >= 0.1cm]   (0) to[distance=1.5cm,out=100, in=100] (1);
\draw [thin,black,->,shorten <= 0.25cm, shorten >= 0.1cm]   (1) to[distance=1.5cm,out=-100, in=-80] (0);
\draw [thin, black,->,shorten <= 0.1cm, shorten >= 0.1cm]   (1) to[distance=1.5cm,out=100, in=100] (2);
\draw [thin,black,->,shorten <= 0.25cm, shorten >= 0.1cm]   (2) to[distance=1.5cm,out=-100, in=-80] (1);
\draw [thin,black,->,shorten <= 0.25cm, shorten >= 0.1cm]   (2) to[distance=1.5cm,out=80, in=100] (3);
\draw [thin,black,->,shorten <= 0.25cm, shorten >= 0.1cm]   (3) to[distance=1.5cm,out=-100, in=-80] (2);
\draw [thin,black,->,shorten <= 0.25cm, shorten >= 0.1cm]   (6) to[distance=1.5cm,out=80, in=100] (L1);
\draw [thin,black,->,shorten <= 0.25cm, shorten >= 0.1cm]   (L1) to[distance=1.5cm,out=-100, in=-80] (6);
\draw [thin,black,->,shorten <= 0.25cm, shorten >= 0.1cm]   (L1) to[distance=1.5cm,out=80, in=100] (L);
\draw [thin,black,->,shorten <= 0.1cm, shorten >= 0.1cm]   (L) to[distance=1.5cm,out=-100, in=-100] (L1);
\blackdot(-4,0);
\blackdot(0,0);
\blackdot(4,0);
\blackdot(16,0);
\blackdot(20,0);
\draw(-4,0) node[left]{$0$};
\draw(0,0) node[left]{$1$};
\draw(4,0) node[left]{$2$};
\draw(10,0) node {$\cdots$};
\draw(13.4,0) node[right]{$\ell-1$};
\draw(18.65,0) node[right]{$\ell$};
\draw(-2,1.2) node[above]{$ a_{1,0}$};
\draw(2,1.2) node[above]{$ a_{2,1}$};
\draw(6,1.2) node[above]{$ a_{3,2}$};
\draw(14,1.2) node[above]{$ a_{\ell-2,\ell-1}$};
\draw(18,1.2) node[above]{$ a_{\ell,\ell-1}$};
\draw(-2,-1.2) node[below]{$ a_{0,1}$};
\draw(2,-1.2) node[below]{$ a_{1,2}$};
\draw(6,-1.2) node[below]{$ a_{2,3}$};
\draw(14,-1.2) node[below]{$ a_{\ell-2,\ell-1}$};
\draw(18,-1.2) node[below]{$ a_{\ell-1,\ell}$};
\end{braid}
\end{align*}

The {\em extended zigzag algebra $\EZig$} is the path algebra $\k\Gamma$ modulo the following relations:
\begin{enumerate}
\item All paths of length three or greater are zero.
\item All paths of length two that are not cycles are zero.
\item All length-two cycles based at the same vertex are equivalent.
\item $ a_{\ell,\ell-1} a_{\ell-1,\ell}=0$.
\end{enumerate}
Length zero paths yield the standard idempotents $\{ e_i\mid i\in I\}$ with $ e_i  a_{i,j} e_j= a_{i,j}$ for all admissible $i,j$. The algebra $\EZig$ is graded by the path length: 
$\EZig=\EZig^0\oplus \EZig^1\oplus \EZig^2.
$ 
We also consider $\EZig$ as a superalgebra with 
$\EZig_\0=\EZig^0\oplus \EZig^2$ and $\EZig_\1=\EZig^1.
$ 

Define
$$
 c_j:= a_{j,j+1} a_{j+1,j} \qquad (j\in J).
$$
The algebra $\EZig$ has an anti-involution $\tau$ with
$$
\tau( e_i)=  e_i,\quad  \tau(a_{ij})=  a_{ji},\quad  \tau(c_j) = c_j.
$$

We consider the total order on $I$
 given by $0<1<\dots<\ell$. For $i\in I$, we set 
 $$X(i):=
 \left\{
\begin{array}{ll}
\{e_i,a_{i-1,i}\}   &\hbox{if $i\neq 0$,}\\
\{e_0\}  &\hbox{if $i=0$,}
 \end{array}
 \right.
 \quad
 Y(i):=
 \left\{
\begin{array}{ll}
\{e_i,a_{i,i-1}\}   &\hbox{if $i\neq 0$,}\\
\{e_0\}  &\hbox{if $i=0$.}
 \end{array}
 \right.
 $$
 Finally, define \(\z:=\spa (e_i \mid i \in I )\). 

\begin{Lemma} \label{LAQH} 
The graded superalgebra $\EZig$ is a based quasi-hereditary algebra with \(\z\)-conforming heredity data \(I,X,Y\) and standard anti-involution $\tau$. If $\k$ is local, $\EZig$ is basic.
\end{Lemma}
\begin{proof}
Follows immediately from definitions.
\end{proof}

Let \(e:= e_0 + \cdots + e_{\ell-1} \in Z\). Note that \(e\) is an adapted idempotent, and \(\tau(e) = e\), so the {\em zigzag algebra} \(\overline{Z}:=e Z e \subset Z\) is a cellular algebra with involution \(\tau|_{\overline{Z}}\), and cellular basis
\begin{align*}
\overline{B} = \{xy\mid i \in I, x \in \overline{X}(i), y \in \overline{Y}(i)\},
\end{align*}
where \(\overline{X}(0) = \{a_{\ell-1,\ell}\}\), \(\overline{Y}(0)=\{a_{\ell,\ell-1}\}\), and \(\overline{X}(i) = X(i)\), \(\overline{Y}(i) = Y(i)\) for all \(i \in J\).  
The cell modules are \(\{\bar{\Delta}(i) = e\Delta(i) \mid i \in I\}\). 
Note that \(\bar\z:=e\z e=\spa (e_i \mid i \in J )\).

From now on let $d\leq n$. By Theorem~\ref{T290517_2}, we have a based quasi-hereditary $\k$-superalgebra \(T^\EZig_\z(n,d)\) with heredity data $\La^I_+(n,d), \X,\Y$ 
 and standard modules $\{\Delta(\bla) \mid \bla \in \La^I_+(n,d)\}$. 
By Lemmas~\ref{LTrunc}(ii) and \ref{SchCell}, \(T^{\overline{\EZig}}_{\bar{\z}}(n,d)=\xi^e T^\EZig_\z(n,d)\xi^e\) is a cellular algebra with involution induced by \(\tau\), and cell modules
$
\{\bar{\Delta}(\bla) = \xi^{e}\Delta(\bla) \mid \bla \in \La^I_+(n,d)\}.
$

Recall that we have fixed a field $\F$ and a homomorphism $\k\to \F$. In particular, we have the algebras $\EZig_\F$, $\Zig_\F$, $T^\EZig_\z(n,d)_\F$, $T^\Zig_{\bar \z}(n,d)_\F$, modules $\De(i)_\F$, $L(i)_\F$, $\Delta(\bla)_\F$, $\bar \Delta(\bla)_\F$, etc. 
Since \(eL(0)_\F = 0\) and \(eL(j)_\F = L(j)_\F\) for all \(j \in J\), the simple \(\Zig_\F\)-modules are \(\{\bar{L}(j)_\F = eL(j)_\F \mid j \in J\}\). The following lemma is easily checked.

\begin{Lemma}\label{zigdecomp} Let \(i,j \in I\). Then the graded decomposition numbers for  \(\EZig_\F\)  are given by
$
d_{i,j}^\F(q,\pi) = 
\delta_{i,j} +\delta_{i-1,j} q \pi.
$
\end{Lemma}

In view of Lemma~\ref{adaptSchur}, the irreducible $T^\EZig_\z(n,d)_\F$-modules $\{L(\bla)_\F \mid \bla \in \La^I_+(n,d)\}$ give rise to the irreducible $T^\Zig_{\bar \z}(n,d)_\F$-modules
$$
\{\bar{L}(\bla)_\F = \xi^{e}L(\bla)_\F \mid \bla \in \La^J_+(n,d)\}.
$$


Recalling the notation from \S\ref{nonbasicdecomp} and Lemma \ref{zigdecomp}, for \(\bnu \in \Delta^D_+(n,d)\) we may write
\begin{align*}
\bnu= (\beta^{(0)}, \ldots, \beta^{(\ell)}, \alpha^{(0)}, \ldots, \alpha^{(\ell-1)}),
\end{align*}
where \(\beta^{(i)} = \nu^{(i,i,0,\bar 0,1)}\) and \(\alpha^{(j)}=\nu^{(j+1,j,1,\bar 1, 1)}\), for all  \(i \in I\) and \(j \in J\). 
Setting $\bbe:=(\be^{(i)})_{i\in I}\in \La^I_+(n)$ and 
$\bal:=(\al^{(j)})_{j\in J}\in \La^J_+(n)$, we identify \(\La^D_+(n)\) with \(\La^I_+(n) \times \La^J_+(n)\) via \(\bnu \mapsto (\bbeta, \balpha)\). For all \(\bla, \bmu \in \La^I_+(n,d)\), define 
\begin{align*}
\delta(\bla,\bmu) := \sum_{j \in J}j(|\la^{(j)}| - |\mu^{(j)}|).
\end{align*}

\begin{Lemma}\label{SchurZigDecomp}
Let \(n \geq d\). Then, for \(\bla,\bmu \in \La^I_+(n,d)\), the graded decomposition numbers for  \(T^{\EZig}_{\z}(n,d)_\F\) are given by the formula
\begin{align*}
d_{\bla, \bmu}^\F= (q \pi)^{\delta(\bla,\bmu)}\sum_{\bga, \bbeta \in \La_+^I(n)}
 \sum_{\balpha \in \La_+^J(n)} d^{\,\textup{cl},\F}_{\bga,\bmu} \left( \prod_{i \in I } c^{\la^{(i)}}_{\beta^{(i)}, (\alpha^{(i-1)})'}c^{\ga^{(i)}}_{\beta^{(i)},\alpha^{(i)}}\right),
\end{align*}
where we formally impose that \(\alpha^{(-1)} = \alpha^{(\ell)} = \varnothing\) for \(\balpha = (\alpha^{(0)}, \ldots, \alpha^{(\ell-1)}) \in \La^J_+(n)\). 
Moreover, if  \( \bmu \in \La^J_+(n,d)\) then for the algebra 
\(T^{\Zig}_{\bar\z}(n,d)_\F\) we have 
$$[\bar\De(\bla)_\F : \bar L(\bmu)_\F]_{q,\pi}=d_{\bla,\bmu}^\F.$$
\end{Lemma}
\begin{proof}
The second statement follows from the first since $T^{\overline{\EZig}}_{\bar{\z}}(n,d)=\xi^e T^\EZig_\z(n,d)\xi^e$. 
For the first statement, by Theorem \ref{NonBasicDecomp} we have
\begin{align*}
d_{\bla, \bmu}^\F= \sum_{\bga, \bbeta \in \La_+^I(n)}
 \sum_{\balpha \in \La_+^J(n)} d^{\,\textup{cl},\F}_{\bga,\bmu} \cdot
 (q\pi)^{|\balpha|}
 \left( \prod_{i \in I } c^{\la^{(i)}}_{\beta^{(i)}, (\alpha^{(i-1)})'}c^{\ga^{(i)}}_{\beta^{(i)},\alpha^{(i)}}\right),
\end{align*}
For \(\bga, \bbeta \in \La^I_+(n)\), \(\balpha \in \La^J_+(n)\), and \(i \in I\), we have \(c^{\la^{(i)}}_{\beta^{(i)}, (\alpha^{(i-1)})'}c^{\ga^{(i)}}_{\beta^{(i)},\alpha^{(i)}} =0\) unless \(|\beta^{(i)}| + |\alpha^{(i-1)}| = |\ga^{(i)}|\) and \(|\beta^{(i)}| + |\alpha^{(i)}| = |\la^{(i)}|\). But then this implies that 
\begin{align*}
|\balpha| = \sum_{j \in J}j(|\la^{(j)}| - |\ga^{(j)}|).
\end{align*}
for all \(\balpha \in \Lambda_+^J(n)\) which contribute to the sum. Now, noting that \(d^{\,\textup{cl}, \F}_{\bga, \bmu} = 0\) unless \(|\ga^{(i)}| = |\mu^{(i)}|\) for all \(i \in I\) gives the result.
\end{proof}

\begin{Remark}\label{R140318}
The generalized Schur algebra $T^\Zig_{\bar \z}(n,d)$  is Morita equivalent to weight $d$ RoCK blocks of symmetric groups and the corresponding Hecke algebras, as conjectured by Turner \cite{T} and proved in \cite{EK2}. We conjecture that the Morita equivalence constructed in \cite{EK2} sends cell modules to cell modules and behaves well on combinatorial labels. The evidence for the conjecture comes from the fact that 
the formula for $d_{\bla,\bmu}$ in Lemma \ref{SchurZigDecomp} is equivalent to the formula obtained by Turner \cite[Corollary 134]{T} for decomposition numbers of Specht modules in RoCK blocks.

Note that when \(d < \textup{char} \,\F\) or \(\textup{char} \,\F=0\), we have \(d^{\,\textup{cl}, \F}_{\bla, \bmu} = \delta_{\bla, \bmu}\), so the formula in Lemma \ref{SchurZigDecomp} may be simplified to
\begin{align*}
d_{\bla, \bmu}^\F= (q \pi)^{\delta(\bla,\bmu)}\sum_{\substack{\alpha^{(-1)}, \ldots, \alpha^{(\ell)} \\\beta^{(0)}, \ldots, \beta^{(\ell)} }}  \left( \prod_{i \in I } c^{\la^{(i)}}_{\beta^{(i)}, (\alpha^{(i-1)})'}c^{\mu^{(i)}}_{\beta^{(i)},\alpha^{(i)}}\right),
\end{align*}
where the sum is over partitions \(\alpha^{(-1)}, \ldots, \alpha^{(\ell)}, \beta^{(0)}, \ldots, \beta^{(\ell)}\), with
 \begin{align*}
 |\alpha^{(i)}| = \sum_{j=i+1}^{\ell} |\la^{(j)}| - |\mu^{(j)}|
 \qquad
 \textup{and}
 \qquad
  |\beta^{(i)}| = |\mu^{(i)}| + \sum_{j=i+1}^{\ell} |\mu^{(j)}| - |\la^{(j)}|.
 \end{align*}

On the other hand, we have the formula obtained by Chuang-Tan \cite[Theorem 6.2]{CT} and Leclerc-Miyachi \cite[Corollary 10]{LM} (see also \cite[Theorem 4.1]{JLM}), for RoCK blocks of weight \(d < \textup{char}\,\F\):
\begin{align*}
d_{\bla, \bmu}^{\F,\textup{RoCK}}= 
(q \pi)^{\delta^{\textup{RoCK}}(\bla,\bmu)}
 \sum_{\substack{\alpha^{(0)}, \ldots, \alpha^{(\ell+1)} \\\beta^{(0)}, \ldots, \beta^{(\ell)} }}   \left( \prod_{0\leq i\leq \ell} c^{\la^{(i)}}_{\beta^{(i)},(\alpha^{(i+1)})'} c^{\mu^{(i)}}_{\beta^{(i)}, \alpha^{(i)}}\right).
 \end{align*}
 where the sum is over partitions \(\alpha^{(0)}, \ldots, \alpha^{(\ell+1)}, \beta^{(0)}, \ldots, \beta^{(\ell)}\),
 \begin{align*}
 |\alpha^{(i)}| = \sum_{j=0}^{i-1} |\la^{(j)}| - |\mu^{(j)}|
 \qquad
 \textup{and}
 \qquad
  |\beta^{(i)}| = |\mu^{(i)}| + \sum_{j=0}^{i-1} |\mu^{(j)}| - |\la^{(j)}|,
 \end{align*}
 and
 \begin{align*}
 \delta^{\textup{RoCK}}(\bla,\bmu):=\sum_{j =1}^\ell(\ell - j + 1)(|\la^{(j-1)}| - |\mu^{(j-1)}|).
 \end{align*}
 After some manipulation and reindexing, we get
 \begin{align*}
 d^\F_{\bla,\bmu} = d^{\F, \textup{RoCK}}_{\bla', \bmu'},
 \end{align*}
where \(\bla':= ((\bla^{(\ell)})', \ldots, (\bla^{(0)})')\) for all \(\bla \in \La^I_+(n)\). This notation coincides with the fact that if \(\la\) is a partition in a RoCK block with \((\ell+1)\)-quotient \(\bla\), then \(\la'\) has \((\ell+1)\)-quotient \(\bla'\), see \cite[Lemma 1.1(2)]{Pag}.
\end{Remark}

For the following conjecture, we now consider the usual $q$-Schur algebra $S_q(N,f)_\F$. Let $e$ be the corresponding {\em quantum characteristic}, see e.g. \cite[\S2.1]{K}. Note that in the case $q=1$, we have $e=\cha \F$. To avoid trivial cases we assume that $e>0$. 

\begin{Conjecture} \label{Conj} 
{\rm 
Let $\mathcal B$ be a weight $d$ RoCK block of $S_q(N,f)$. Let $\ell=e-1$ and $Z$ be the corresponding extended zigzag algebra. For $n\geq d$, the generalized Schur algebra $T^{\EZig}_\z(n,d)$ is Morita equivalent to $\mathcal B$.
}
\end{Conjecture}

This conjecture is in spirit of \cite[Conjecture 178]{T}, although it is not clear to us whether Turner's algebra ${\mathcal Q}_p(n,d)$ appearing in loc. cit. is isomorphic or even Morita equivalent to our algebra $T^Z_\z(n,d)$ in this case.



\begin{thebibliography}{ABC}



\bibitem[BSV]{BSV}
F. Bonetti, D. Senato, and A. Venezia. The Robinson-Schensted correspondence for the fourfold
algebra. {\em Boll. Unione Mat. Ital.  VII Ser. B 2} {\bf 3} (1998) 541--554. 


 
\bibitem[CPS]{CPS}
E. Cline, B. Parshall and L. Scott, Integral and graded quasi-hereditary algebras, I, {\em J. Algebra} {\bf 131} (1990), 126--160.

\bibitem[CT]{CT}
J. Chuang and K.M. Tan, Filtrations in Rouquier blocks of symmetric groups
and Schur algebras, Proc. London Math. Soc. {\bf 86} (2003), 685--706.


\bibitem[DS]{DS}
J. Du and L. Scott, Lusztig conjectures, old and new. I, {\em J. Reine  Angew. Math.} {\bf 455}  (1994), 141--182.

\bibitem[EK${}_1$]{EK1}
A.\ Evseev and A.\ Kleshchev, Turner doubles and generalized Schur algebras, {\em Adv. Math.} {\bf 317} (2017), 665--717. 

\bibitem[EK${}_2$]{EK2}
A.\ Evseev and A.\ Kleshchev, Blocks of symmetric groups, semicuspidal KLR algebras and zigzag Schur-Weyl duality, {\em Ann. of Math.} (2), {\bf 188} (2018), 453--512.

\bibitem[GG]{GG} T. Geetha and F. M. Goodman, Cellularity of wreath product algebras and $A$-Brauer algebras, {\em J. Algebra} {\bf 389} (2013), 151-190. 




\bibitem[Gr${}_1$]{GreenCod}
J. A. Green, Combinatorics and the Schur algebra, {\em J. Pure Appl. Algebra} {\bf 88} (1993), 89--106.

\bibitem[Gr${}_2$]{Green}
J.A. Green, {\em Polynomial representations of $GL_n$}, 2nd edition, Springer-Verlag, Berlin, 2007. 



\bibitem[JLM]{JLM} 
G.D. James, S. Lyle, and A. Mathas, Rouquier blocks, {\em Math. Z.} {\bf 252} (2006), 511--531.

 
 \bibitem[K]{K}
A. Kleshchev, Representation theory of symmetric groups and related Hecke algebras, {\em Bull. Amer. Math. Soc.} {\bf 47} (2010), 419--481. 

\bibitem[KM${}_1$]{greenOne}
A.\ Kleshchev and R. Muth, Based quasi-hereditary algebras, preprint.

\bibitem[KM${}_2$]{greenTwo}
A.\ Kleshchev and R. Muth, Generalized Schur algebras, preprint.


\bibitem[LM]{LM}
B. Leclerc, H. Miyachi, Some closed formulas for canonical bases of Fock spaces, Representation Theory {\bf 6} (2002), 290--312.

\bibitem[LNS]{LNS}
R. La Scala, V. Nardozza and D. Senato. Super RSK-algorithms and super
plactic monoid. {\em Internat. J. Algebra Comput. 16} {\bf 2}, (2006), 377--396

\bibitem[M]{MacD}
I.G. Macdonald, {\em Symmetric Functions and Hall Polynomials}, Second Edition, Oxford University Press, 1995. 

\bibitem[P]{Pag}
R. Paget, The Mullineux map for RoCK blocks, Comm. Algebra {\bf 34} (2006), 3245--3253. 

\bibitem[R]{Ro}
R. Rouquier, $q$-Schur algebras and complex reflection groups, {\em Mosc. Math. J.} {\bf 8} (2008), 119--158. 

\bibitem[T${}_1$]{T}
W. Turner, Rock blocks, Memoirs of the AMS {\bf 202} (2009), no. 947.


\bibitem[T${}_2$]{T2}
W.\ Turner, Tilting equivalences: from hereditary algebras to symmetric groups, {\em J.\ Algebra} {\bf 319} (2008), 3975--4007. 

\bibitem[T${}_3$]{T3}
W.\ Turner, Bialgebras and caterpillars, {\em Q.\ J.\ Math.} {\bf 59}  (2008), 379--388.

\bibitem[W]{Woodcock}
D.J. Woodcock, Straightening codeterminants, {\em J. Pure Appl. Algebra} {\bf 88} (1993), 317--320.




\end{thebibliography}
\end{document}